\documentclass[11pt,reqno]{amsart}
\usepackage{amsmath,amssymb,xypic,mathrsfs,bm}
\usepackage{adjustbox}
\usepackage{tikz-cd}

\numberwithin{equation}{section}

\theoremstyle{plain}
 \newtheorem{theorem}{Theorem}[section]

\theoremstyle{definition} 
 \newtheorem{remark}[theorem]{Remark}
\input xy
\xyoption{all}

\newif\ifPDF
\ifx\pdfoutput\undefined 
\ifx\pdfversion\undefined
\PDFfalse
\fi
\else 
\ifnum\pdfoutput > 0
\PDFtrue
\else
\PDFfalse
\fi
\fi

\makeatletter
\newcommand{\xRightarrow}[2][]{\ext@arrow 0359\Rightarrowfill@{#1}{#2}}
\makeatother

\theoremstyle{plain}

\newtheorem{proposition}[theorem]{Proposition}		
\newtheorem{corollary}[theorem]{Corollary}
\newtheorem{lemma}[theorem]{Lemma}

\newtheorem{definition}[theorem]{Definition}

\newtheorem{example}[theorem]{Example}

\newcommand{\CBbb}{\mathbb C}

\newcommand{\QBbb}{\mathbb Q}

\newcommand{\Bcal}{\mathcal B}
\newcommand{\Ccal}{\mathcal C}

\newcommand{\Ecal}{\mathcal E}

\newcommand{\Hcal}{\mathcal H}

\newcommand{\Lcal}{\mathcal L}

\newcommand{\Pcal}{\mathcal P}

\newcommand{\Tcal}{\mathcal T}

\newcommand{\Vcal}{\mathcal V}

\newcommand{\Xfrak}{\mathfrak X}

\newcommand{\gfrak}{\mathfrak g}
\newcommand{\hfrak}{\mathfrak h}

\newcommand{\Cscr}{\mathscr C}

\newcommand{\SL}{\mathsf{SL}}

\newcommand{\GL}{\mathsf{GL}}

\newcommand{\Sp}{\mathsf{Sp}}
\newcommand{\slfrak}{\mathfrak{sl}}

\newcommand{\TXM}{\mathcal{T}_{{\mathfrak{X}^{par}_{{G}}/{M}_{{G}}^{par,rs}}}}

\newcommand{\AXM}{{}^{spar}{\rm At}_{\mathfrak{X}^{par}_{{G}}/{M}_{{G}}^{par,rs}}}
\newcommand{\TMS}{\mathcal{T}_{{{M}_{{G}}^{par,rs}/S}}}
\newcommand{\AxM}{\At_{\mathfrak{X}^{par}_{{G}}/{M}_{{G}}^{par,rs}} }
\newcommand{\AxxM}{{}^{par}{\rm At}_{\mathfrak{X}^{par}_{{G}}/{M}_{{G}}^{par,rs}}(\Pcal)}

\DeclareMathOperator{\Par}{Par}
\DeclareMathOperator{\Spar}{SPar}

\DeclareMathOperator{\id}{id}

\DeclareMathOperator{\Sym}{Sym}

\DeclareMathOperator{\ad}{ad}

\DeclareMathOperator{\Hit}{Hit}

\DeclareMathOperator{\ParAt}{{}^{\emph{par}}At}
\DeclareMathOperator{\SParAt}{{}^{\emph{spar}}At}
\DeclareMathOperator{\ParEnd}{Par}
\DeclareMathOperator{\SParEnd}{SPar}

\newcommand{\lra}{\longrightarrow}

\newcommand{\At}{{\rm At}}

\newcommand{\isorightarrow}{\xrightarrow{
		\,\smash{\raisebox{-0.5ex}{\ensuremath{\sim}}}\,}}

\begin{document}
\title[]{A Hitchin connection on nonableian theta functions for parabolic $G$-bundles}

\author[Biswas]{Indranil Biswas}

\address{Mathematics Department, Shiv Nadar University, NH91, Tehsil Dadri,
	Greater Noida, Uttar Pradesh 201314, India}

\email{indranil.biswas@snu.edu.in, indranil29@gmail.com}

\author[Mukhopadhyay]{Swarnava Mukhopadhyay}

\address{School of Mathematics, Tata Institute of Fundamental Research,
	Homi Bhabha Road, Mumbai 400005, India}
\email{swarnava@math.tifr.res.in}

\author[Wentworth]{Richard  Wentworth}
\address{Department of Mathematics,
	University of Maryland,
	College Park, MD 20742, USA}
\email{raw@umd.edu}
\thanks{I.B. is
	supported in part by a J.C. Bose fellowship and both I.B. and S.M. by DAE, India
	under project no. 1303/3/2019/R\&D/IIDAE/13820.  }



\begin{abstract}For a simple, simply connected complex affine algebraic group $G$, 
	we prove the existence of a flat projective connection on the bundle of
	nonabelian theta functions on the moduli spaces of semistable
	parabolic $G$-bundles for families of smooth projective 
	curves with marked points. 
\end{abstract}
\maketitle
\allowdisplaybreaks

\thispagestyle{empty}

\section{Introduction}
In this paper, we prove the existence of a flat projective connection on
spaces of generalized theta functions on the moduli spaces of parabolic
$H$-bundles for  a family of  smooth projective curves with marked points, where
$H$ is a connected, complex, simple, affine algebraic group. 
Before stating the precise results, and 
since it is part of the larger and well-studied program
of geometric quantization,  we first provide a brief historical context to
this subject. 

Quantization as envisioned by Dirac,  et al., can be thought of  as a
deformation of a classical mechanical system depending on a 
parameter $\hbar$ that recovers the original classical system in the limit. 
Kostant-Kirillov-Souriau developed and generalized this notion of 
``quantizing a function'', and Auslander-Kostant \cite{ASKT} used it 
to construct unitary  representations  of a connected Lie group (see 
also Kirillov \cite{Kirillov}). 
\paragraph{\bf Geometric Quantization} The starting point of the theory is a 
symplectic manifold $(M,\,\omega)$ where the symplectic form $\omega$ 
is the curvature of a Hermitian line bundle $\Lcal$ with 
connection $\nabla$. The quantum Hilbert space $\mathscr{H}$ 
is then the $L^2$-completion of the space of global 
sections $\Gamma(M,\,\Lcal)$ of this line bundle. The 
Lie algebra of functions on $M$, under the Poisson bracket 
given by the form $\omega$, acts naturally on $\mathscr{H}$. This 
process of assigning a function to this Lie algebra satisfying certain 
commutativity constraints depending on $\hbar$ is known 
as quantization in the present literature. However, 
it is not possible to achieve these commutativity constraints in practice.
To remedy this, Kostant \cite{Kostant70} and Souriau 
\cite{Souriau67} further consider a compatible
almost complex structure $I$ on 
$M$ such that $(M, \omega, I)$ is a K\"ahler manifold.
This  induces a holomorphic structure on the line bundle $\Lcal$	
and leads to the notion of {\em geometric quantization}, where the 
Hilbert space $\mathscr{H}_{I}$ is reduced to the space 
of holomorphic $L^2$-sections of $\Lcal$. Because the quantization process
should arrive at a unique answer, 
it is natural to investigate the dependence of the geometric 
quantization on the choice of almost complex structure $I$ on $M$. 

In \cite{Hitchin:90}, Hitchin analyzes
this question in a very important setting (see also 
\cite{ADW}, \cite{Faltings:93}, \cite{VanGeemenDeJong:98}, \cite{Andersen:12}).
Here, $M\,=\,\operatorname{Hom}^{irr}(\pi_1(\Sigma),\,K)/K$,
is the moduli space of a class of representations of the fundamental 
group $\pi_1(\Sigma)$ to $K$, where $\Sigma$ is a closed oriented surface
and $K\subset G$ is a maximal compact
subgroup of the earlier mentioned simple, simply connected group $G$. The group 
$K$ acts by conjugation on a representation $\rho:\pi_1(\Sigma)\to K$,
and $\rho\, \in\, \operatorname{Hom}^{irr}(\pi_1(\Sigma),\,K)$ if the stabilizer of $\rho$ under this action
is exactly the center of $K$. This space has a  symplectic 
form defined by Atiyah-Bott \cite{AtiyahBott:82}, Narasimhan \cite{NarsElliptic}, and Goldman.  A choice of a complex structure  $I$
on $\Sigma$ endows $M$ with a K\"ahler structure, and via the 
Narasimhan-Seshadri-Ramanathan theorem this complex manifold, which we call $M_I$, can be 
identified with the space of regularly  stable holomorphic principal 
$G$-bundles on $C\,:=\,(\Sigma, \,I)$ (see \cite[Prop.\ 7.7 and Thm.\
7.1]{Ramanathan:75}).  The 
role of $\Lcal$ is played by a determinant of cohomology line bundle 
defined via some linear representation
of $G$, 
and $\mathscr{H}_I\,:=\,H^0(M_I,\,\Lcal^{\otimes k})$ 
is the space of nonabelian theta functions of level $k$.
The connection $\nabla$ is the Chern connection of the Quillen metric. 
Hitchin found a flat projective connection on the bundle of 
nonabelian theta functions over a family of curves of fixed genus. 
His construction may be interpreted as a natural identification
between the spaces $\mathbb{P}(H^0(M_I,\Lcal^{\otimes k})\cong 
\mathbb{P}(H^0(M_{I'}, \Lcal^{\otimes k})$  via parallel transport 
along a path connecting $I$ and $I'$ in the Teichm\"uller space.  

\paragraph{\bf TUY/WZW connection} As mentioned above, 
the vector spaces $\mathscr{H}_I$ that appear in Hitchin's geometric
quantization have a counterpart in the WZNW-model of a 
$2d$ rational conformal field theory constructed by Tsuchiya-Ueno-Yamada 
\cite{TUY:89}, which appears in the quantization of 
a $3d$-Chern-Simons theory to a $3d$-TQFT as considered by Witten 
\cite{Witten89}. Let $\mathfrak{g}$ denote the Lie algebra of  $G$. 
Given a positive integer $k$ and an 
$n$-tuple $\bm{\lambda}$ of dominant weights for $\gfrak$ satisfying a certain
integrability condition depending on $k$, the paper \cite{TUY:89} constructs
a vector bundle   
$\mathbb{V}^{\dagger}_{\bm{\lambda}}(\mathfrak{g},k)$ on the 
Deligne-Mumford compactification $\overline{\mathcal{M}}_{g,n}$ of
stable $n$-pointed curves of genus $g$.   
Over the interior ${\mathcal{M}}_{g,n}$ 
parametrizing smooth curves, 
$\mathbb{V}^{\dagger}_{\bm{\lambda}}(\mathfrak{g},k)$ admits a flat
projective connection. 
These vector bundles of conformal blocks  
satisfy the  axioms of a $2d$-rational conformal field theory. Moreover,
due to work of Beauville-Laszlo \cite{BeauvilleLaszlo:94}  and
Kumar-Narasimhan-Ramanathan \cite{KNR:94}, 
in the case of a single puncture with  trivial weight,
we get a canonical (up to a scalar) 
identification of $\mathscr{H}_{I}$ with the fiber of
$\mathbb{V}^{\dagger}_{\bm{\lambda}}(\mathfrak{g},k)$ at the 
point $C\,=\,(\Sigma,\, I)$ in $\mathcal{M}_{g,n}$. 
It is  natural to ask
whether the connections of Hitchin \cite{Hitchin:90} and 
Tsuchiya-Ueno-Yamada \cite{TUY:89} coincide. 
That this is indeed the case was proven by Laszlo \cite{Laszlo}. 

A generalization of the identification of $\mathscr{H}_I$ with 
conformal blocks also holds for smooth $C$ with an $n$-tuple of marked points
$\bm p$. 
Consider the moduli space 
${M}_{G}^{par,rs}={M}_{G}^{par,rs}(C, {\bm p} , {\bm \lambda})$ of regularly  
stable parabolic $G$ bundles on a compact 
Riemann surface $C$ with $n$-marked points $\bm p$  
and parabolic structures $\bm{\lambda}$ at $\bm p$. Let 
$\mathcal{L}_{\bm{\lambda},k}$ be a parabolic ``determinant of 
cohomology'' line bundle  on ${M}_{G}^{par,rs}$. Then  
there is a canonical (up to scalars) isomorphism between the finite dimensional vector
space of holomorphic sections $H^0({M}^{par,rs}_{G},
\mathcal{L}_{\bm{\lambda},k})$ and the fiber of the space of conformal
blocks
$\mathbb{V}^{\dagger}_{\bm{\lambda}}(\mathfrak{g},k)\bigr|_{(C,\bm{p})}$
(see \cite{Pauly:96} and \cite{LaszloSorger:97}). 
This identification between conformal 
blocks and nonabelian theta functions is a mathematical analog 
of the Chern-Simons/WZNW correspondence of Witten \cite{Witten89}.  
Since 
the vector bundle of conformal blocks is endowed with a 
flat projective connection, it is very natural to ask the following question: 

\medskip

{\bf Question}. \emph{Is there a natural
flat projective connection on the family of spaces $H^0({M}^{par,rs}_{G},
	\mathcal{L}_{\bm{\lambda},k})$ as the pointed Riemann surface
	structure of $C$ moves in a holomorphic family?}
\medskip

For parabolic vector bundles, a construction of the projectively flat  connection was
given
by Scheinost-Schottenloher in 
\cite{SS} for those {\em special cases }
of weights $\bm \lambda$ such that the canonical class of the
corresponding parabolic moduli space, which 
depends only on the rank, number of points, and  the flag types 
of $\bm \lambda$,   admits a square root. 
This condition often appears in the context of geometric quantization under
the term \emph{metaplectic correction} (see also
\cite{AGL}) and it produces a projective connection on the push-forward of
the line bundle obtained by modifying $\mathcal{L}_{\bm \lambda,k}$ by the square root. The proof in the above reference makes use of a correspondence between
parabolic bundles on a curve with rational weights, and holomorphic bundles
on an  associated elliptically fibered complex surface.  However, for moduli spaces of
parabolic bundles, the condition on the  existence of a square root of the canonical bundle
is   {\em not always satisfied}. 

In \cite{BjerreThesis}, Bjerre proved the existence of a (unique) 
flat projective connection
for the moduli space of parabolic vector bundles via a gauge theoretic
description of the moduli space.   
An important step in the proof was to remove the condition on the 
existence of a  square root  by passing to a different
moduli space with altered weights.\footnote{After the present paper was posted to the arXiv
we received a preliminary version  of the work of Andersen-Bjerre attributed here \cite{AndersenBjerre}.} \footnote{Subsequent to the submission of this paper, in
May 2023 a draft of  the thesis of 
Zakaria Ouaras \cite{Zakaria} appeared in which the author 
proves the existence
of a unique flat projective connection in the case of moduli spaces of parabolic vector bundles with arbitrary fixed determinant and genus $g \geq 2$.}The 
results of Bjerre and Scheinost-Schottenloher stated above
work only for curves of genus $g\geq2$ and exclude the important case
of genus zero curves with marked points. 
The connection on conformal blocks for genus zero curves 
is known as the Knizhnik-Zamolodchikov connection,
and it has been extensively studied from different perspectives.

The motivation of the present paper is to 
give an affirmative answer to the above question for general $G$ and curves of all genus 
using algebro-geometric methods applied directly to the moduli spaces
in question. 
To state the result precisely, 
first note that the curve $C$ and  parabolic weights $\bm \lambda$ 
determine an  orbifold curve $\mathscr{C}$ (cf.\ Appendix
\ref{sec:appC} and Lemma \ref{lem:codimension}).  Our main result is the following:

\medskip

{\bf Main Theorem}. \emph{Let
	$\Ccal\rightarrow S$ be a versal family of $n$-pointed smooth
	projective    curves, and let $G$ be a simple, simply connected
    complex  algebraic group.  Assume that the genus $g(\mathscr{C})$ of the 
	orbifold curve  
	determined by the weights $\bm \lambda$ satisfies $g(\mathscr{C})\geq 2$, and if
    $G=\SL_2$ or $\Sp_4$, $g(\mathscr{C})\geq 3$.  Let 
	$\pi:{M}^{par, rs}_{G}\rightarrow S$ be the 
	relative moduli space of regularly stable parabolic 
	$G$ bundles on $\Ccal$ for some fixed parabolic weights
	$\bm\lambda$.
	Let $\mathcal{L}_{\phi}$ be the 
	determinant of cohomology line bundle on 
	${M}_{G}^{par,rs}$ determined by  a  choice of 
	representation $\phi : G\rightarrow \SL_r$. 
	Then for any $a\in \QBbb$, for which $\mathcal{L}_{\phi}^{\otimes a}$ 
	defines a line bundle on ${M}_{G}^{par,rs}$,
	the coherent sheaf $\pi_\ast(\mathcal{L}_{\phi}^{\otimes a})$
	has a natural flat projective connection.
}
\medskip

Observe that we can allow the genus of $C$ to be zero or one in
the above theorem, provided some inequalities are satisfied (cf.\ Example
\ref{example:genuszero} below). It is reasonable to  expect that 
the TUY connection for conformal blocks 
and the parabolic Hitchin connection constructed in 
the  Main Theorem  coincide under the identification mentioned above. 
We postpone this question for a future work. 

\paragraph{\bf Key difference in the parabolic case}Before proceeding further, 
we describe the key difference in the parabolic set-up. 
The moduli space of principal $G$-bundles satisfies a 
``monotone'' condition: the first Chern class of the moduli
space is a multiple of the Chern class of the prequantum line
bundle. This 
property is  an  important technical point in Hitchin's construction of the
connection (cf.\ \cite[eqs.\ (2.8) and (3.9)]{Hitchin:90}), and it leads to
a solution to the van Geemen-de Jong condition in Theorem \ref{thm:hitchimain} (i)
below. 

The main new feature in the case of {\em parabolic bundles }
is the higher rank of the Picard group of the moduli space,
and because of this {\em monotonicity
	no longer holds}.

\paragraph{\bf Main Ideas}The key ideas and methods used this paper
to address the lack of  monotonicity mentioned 
above are the following:
\begin{itemize}
	\item  The
	fiducial symbol coming from
	the usual construction of Hitchin connection can be naturally 
	modified to a new condition that now satisfies 
	the van Geemen-de Jong condition (see
	\eqref{eqn:hitchparsymbol}). 
	
	\item This modification is facilitated by another 
	crucial ingredient, which is a 
	categorical equivalence  of ``$\pi$-bundles'' on a ramified cover
	$\widehat C\to C$ with 
	parabolic bundles on $C$ (\cite{SeshadriI}, \cite{Biswasduke}, \cite{BalajiBiswasNagaraj} and \cite{SeshadriII}). 
	\item We prove and use an equivariant analog of a result of 
	Beilinson-Schechtman \cite{BS88} connecting 
	classes of {\em Atiyah algebras} obtained as equivariant push-forwards 
	of a differential graded Lie algebra
	with those associated to the determinant of cohomology of the universal bundle. 

	\item Finally we use the fact that the line bundles on moduli 
	space of parabolic bundles 
	adapted to the parabolic weights correspond exactly to the restriction of
	the determinants of cohomology to the locus of  orbifold bundles (cf.\
	\cite{BiswasRaghavendra}, \cite{DW97}). 
	
\end{itemize}

We now discuss some applications of the main theorem mentioned above. 
Let ${H}$ be a simple algebraic group with nontrivial fundamental
group, and let
${\widetilde{H}}$ be its simply connected cover.  Let $\pi:{M}^{par, rs,0}_{H}\rightarrow S$  be the neutral component of the
relative moduli space of regularly stable parabolic 
$H$ bundles on $\Ccal\to S$ for some fixed parabolic weights
$\bm\lambda$, which we assume lift to weights for  ${\widetilde{H}}$. 
As before, let $\mathcal{L}_{\bm \lambda, k}$ be the parabolic determinant of 
cohomology. It is natural 
to ask whether the coherent sheaf $\pi_{*}\mathcal{L}_{\bm \lambda, k}$ 
carries a projectively flat connection. A direct corollary of the main theorem is the following:
\begin{corollary}\label{cor:semisimpleG}
	For any simple group ${H}$, the coherent sheaf $\pi_{*}\mathcal{L}_{\bm \lambda, k}$ is 
	locally free and carries a flat projective connection whose 
	symbol is the same that for that for the simply connected cover ${\widetilde{H}}$.
\end{corollary}

Observe that moduli spaces of parabolic bundles are not 
necessarily Fano, and hence 
we cannot use a Grauert-Riemenschneider
type vanishing theorem as in the nonparabolic case to conclude local freeness via vanishing 
of higher cohomologies. Furthermore, since $H$ is not simply connected, 
we cannot reconstruct these space via affine Lie algebraic methods.

We now briefly recall  the earlier constructions of the Hitchin/WZW/TUY connections in the nonparabolic setting.
Hitchin's construction of a projective connection 
in the closed (nonparabolic) case draws parallels with Welters' work on 
theta functions for abelian varieties \cite{welters}.
The starting point is the description of first order deformations of
the 
triple $(M_I,\Lcal^{\otimes k},s)$, where $s\in H^0(M_I,\Lcal^{\otimes
	k})$, in terms of the first hypercohomology group
of the complex
$\At(\Lcal^{\otimes k})\to \Lcal^{\otimes k}$ constructed using $s$. 
Here,  $\At(\Lcal)$ denotes  the Atiyah algebra of $\Lcal$.  
Though Hitchin's methods were differential geometric in nature,
in  \cite{VanGeemenDeJong:98} van Geemen and de Jong reinterpreted
the construction
in an algebraic manner closer to that of \cite{welters}.
Using this  framework, along with the fundamental  results of 
Beilinson-Schechtman \cite{BS88} and Bloch-Esnault \cite{BE}, 
Baier-Bolognesi-Martens-Pauly \cite{BBMP20} 
reproduced Hitchin's connection for $G=\SL_r$. 
Moreover, their proof works over
fields of positive characteristic,  with 
a few extra assumptions.

The Hitchin connection for $G=\GL_r$ bundles had previously been found
by Belkale \cite{Belkale:09}. 
Other algebro-geometric constructions of the Hitchin connection
are given in \cite{Faltings:93}, \cite{Ran},  \cite{RamaHitchin},  and 
by Ginzburg in \cite{Ginzburg}. Ref.\  \cite{ST} uses the results of 
\cite{BS88} to extend Hitchin's connection for logarithmic connections
and the moduli space of semistable torsion-free sheaves on nodal curves. 
The approach in the present paper  is strongly
motivated by \cite{BBMP20} and \cite{Ginzburg}.

\paragraph{\bf Further generalizations}   In fact, it is possible to work in the general setting of 
$\Gamma$-$\operatorname{Aut}(G)$-bundles.
A moduli space of such pairs with  a fixed local
type has been constructed by Balaji-Seshadri \cite{BalajiSeshadri}
(in the case of $\Gamma$-$G$-bundles in characteristic zero) and by 
Heinloth \cite{HeinlothEpiga} (in the more general settings of
Bruhat-Tits torsors in the sense of Pappas-Rapoport \cite{PR1}, and over
fields of arbitrary characteristic).
We note that 
it has been not verified whether the stability conditions of 
\cite{BalajiSeshadri} and  \cite{HeinlothEpiga} coincide. 
Nevertheless, the results in Section \ref{sec:symbol} generalize verbatim 
to moduli spaces of $\Gamma$-$\operatorname{Aut}(G)$-bundles.
However, in order to produce a Hitchin 
connection (as described in Section \ref{sec:proof}),
the following additional information would be required: 
\begin{itemize}
	\item the base of the Hitchin map for the  moduli of parahoric 
	Higgs bundles for $(\Gamma, \operatorname{Aut}(G))$ is 
	affine, and the fibers of the Hitchin map are connected; 
	\item the complement of the cotangent bundle of the moduli space of 
	$\Gamma$-$\operatorname{Aut}(G)$-bundles 
	in the parahoric Higgs bundles moduli space has codimension at least 2.
\end{itemize}
There are some results in the direction of the first point by B. Wang
\cite{BW2},  who extends the  result of Donagi-Pantev \cite{DonagiPantev:12}
to the set-up of  parahoric $\Gamma$-$G$-Higgs bundles. 
In full generality, however, the two items  above are not presently available 
in the literature,  and 
we therefore restrict ourselves here to the setting of parabolic bundles.

\paragraph{\bf Outlook}	The paper \cite{BBMP20} cited above  argues that  it is  of independent interest
to consider the Hitchin connection over field of
positive characteristics from the view point of the 
Grothendieck-Katz $p$-curvature conjecture and the modular
representations of the mapping class group. 
The constructions in this paper follow those of   \cite{BBMP20} and  are likely
to work (after suitable modifications of the techniques used here) over fields of characteristic $p >0$, 
unless $p \in \{2,3,h^\vee(\mathfrak{g}),k, k+h^\vee\}$. 
But even given these constraints on $p$ it is not clear
whether $\pi_*\mathcal{L}_{\bm{\lambda},k}$ is locally free. For this, it
would be  enough
to show that $H^1(M^{par,ss}_{{G}},\mathcal{L}_{\bm{\lambda}, k})$ vanishes. 
However, in the parabolic case the moduli spaces $M^{par,ss}_{{G}}$ are
not Fano in general, 
even in characteristic zero. 
Moreover, there is no suitable Grauert-Riemenschneider vanishing theorem. 

A uniform approach to this vanishing result would follow 
if one can show that $M^{par,ss}_{{G}}$ are Frobenius-split. 
There is some work in this direction for ${G}=\SL_2$ by 
Mehta-Ramadas \cite{MehtaRamadas} and by Sun-Zhou \cite{SZ}, who show that 
semistable parabolic bundles of rank $r$ and fixed determinant are globally
$F$-regular type. A general result on Frobenius splitting for moduli of
parabolic bundles is presently missing in the literature.

\paragraph{\bf Organization}This paper is organized as follows:
In Section 1, we review the construction of the projectively flat connection in the general set-up following Hitchin \cite{Hitchin:90} and 
van Geemen-de Jong \cite{VanGeemenDeJong:98}.
In Section \ref{sec:symbol},
we review the generalizations of Hitchin's symbol and Kodaira-Spencer maps
in the parabolic bundle context.  
The important result here is Theorem \ref{thm:cup-product}, which
relates the fiducial Hitchin symbol to the relative extension classes of
the Atiyah algebras of the $G$-bundle and the determinant of cohomology. 

Finally, in Section \ref{sec:proof}  we prove that the modified Hitchin
symbol satisfies the constraint equations of van Geemen-de Jong. This leads
to the proof of the Main Theorem. 
The last three sections contains some definitions and technical results on parabolic bundles,
invariant push-forwards, and vanishing theorems, that are used at various
points in the paper. 
In particular, the  determinant of cohomology line bundles 
$\mathcal{L}_{\phi}$ associated to a linear representation $\phi$ of $G$ are
defined there. 
\emph{Parabolic} determinant of cohomology line bundles
are  defined in 
\ref{def:parabolcidetSL} and \ref{def:parabolicdeteG}.
We also explain the admissible values of $k$,  how to 
realize the parabolic determinant of 
cohomology bundles via the moduli space of $\Gamma$-$G$-bundles,  
and the invariant push-forward functor construction.

For the rest of the paper we emphasize that the ground field of varieties
and schemes is always  
 $\CBbb$, and we shall freely go back and forth between Zariski and
 analytic topologies.

\section{Flat projection connection following Hitchin--van Geemen--de Jong}

Let $\pi\,:\, M \,\rightarrow\, S$ be a smooth surjective proper map of smooth
varieties with connected fibers and $\mathcal{L}\to M$ a line bundle. In this 
section we briefly recall a general approach for constructing connections on the 
coherent sheaf $\pi_*\mathcal{L}$.
This is due to Hitchin \cite{Hitchin:90} in the K\"ahler setting (generalizing Welters 
\cite{welters}) and to  van Geemen--de Jong \cite{VanGeemenDeJong:98} in the
algebro-geometric setting. 

\subsection{Heat operators}
From \cite[Sec.\ 2.3]{VanGeemenDeJong:98} we recall the notion of a heat operator and associated connections. 
For  $i\geq\ 1$,
let $\mathcal{D}^{\leq i}(\mathcal{L})$ (resp.\  $\mathcal{D}^{\leq i}_{M/S}(\mathcal{L})$) denote the sheaf of
differential operators (resp.\ relative differential operators) of order at most $i$ on the line bundle $\mathcal{L}$. 

Consider the subsheaf 
$\mathcal{W}_{M/S}(\mathcal{L})\,=\,\mathcal{D}^{\leq
1}(\mathcal{L})+\mathcal{D}_{M/S}^{\leq 2}(\mathcal{L})$ of the sheaf of  second order differential operators on $\mathcal{L}$. It fits into the following short exact sequence:
\begin{equation}\label{z1}
0 \rightarrow \mathcal{D}^{\leq 1}_{M/S}(\mathcal{L})\rightarrow \mathcal{W}_{M/S}(\mathcal{L})
\rightarrow \pi^*\mathcal{T}_S \oplus\operatorname{Sym}^2 \mathcal{T}_{M/S} \rightarrow 0\ .
\end{equation}
Note that $\mathcal{O}_S\subset \mathcal{D}^{\leq 1}_{M/S}(\mathcal{L})\subset \mathcal{W}_{M/S}(\mathcal{L})$.

\begin{definition} \label{def:heat}
A heat operator $D$ on $\mathcal{L}$ is a map 
$D\,:\, \pi^*\mathcal{T}_{S}\,\rightarrow \,\mathcal{W}_{M/S}(\mathcal{L})$ whose composition with the natural projection
map $\mathcal{W}_{M/S}(\mathcal{L})\rightarrow\pi^*\mathcal{T}_{S}$, given by \eqref{z1}, is the identity map of $\pi^*\mathcal{T}_{S}$. 
A projective heat operator $\overline{D}$ on $\mathcal{L}$ is an $\mathcal{O}_S$-linear map $\overline{D}:
\mathcal{T}_S \rightarrow (\pi_* \mathcal{W}_{M/S}(\mathcal{L}))/\mathcal{O}_S$ such that any local lifting gives a heat operator.
\end{definition}

Given a heat operator $D$, we can construct a connection $\nabla(D): \pi_*\mathcal{L} \rightarrow \pi_*\mathcal{L}\otimes \Omega^1_S$
on the coherent sheaf $\pi_*\mathcal{L}$ as follows:
Let $\theta\in \mathcal{T}_S(U)$, where $U\subset S$ an open subset. Then by definition, $D(\pi^{-1}\theta)$ is a second order differential
operator on $\mathcal{L}(\pi^{-1}(U))$. Let $s$ be a section of $\pi_*\mathcal{L}(U)$ and $f\in \mathcal{O}_S(U)$. Then
$D(\pi^{-1}\theta)((f\circ\pi)s)\,=\, f\cdot D(\pi^{-1}\theta)(s)+\theta (f)\cdot s$, in other words,
$D(\pi^{-1}\theta)$ satisfies the Leibniz rule. Indeed,
this follows from the requirement in  Definition \ref{def:heat} that the heat operator
is the standard first order operator on the base. Hence, we get a connection $\nabla(D)$.

\subsection{Existence of a heat operator}

The Kodaira-Spencer map is given by:
$$KS_{M/S}: \mathcal{T}_{S}\lra R^1\pi_{*} \mathcal{T}_{M/S}.$$
On the other hand, we have  the
coboundary map $$\mu_{\mathcal{L}}: \pi_*\operatorname{Sym}^2\mathcal{T}_{M/S}
\longrightarrow R^1\pi_* \Tcal_{M/S},$$
occurring in the long exact sequence obtained from the push forward $\pi_{*}$ of the fundamental
short exact sequence of differential operators 
\begin{equation*}
0 \lra \mathcal{T}_{M/S} \cong \mathcal{D}_{M/S}^{\leq
	1}(\mathcal{L})/\mathcal{O}_M\lra \mathcal{D}_{M/S}^{\leq
	2}(\mathcal{L})/\mathcal{O}_{M} \stackrel{s_2}{\lra} \operatorname{Sym}^2
\mathcal{T}_{M/S}\lra 0\ ,
\end{equation*}where $s_2$ is the symbol map.
Given $\rho: \mathcal{T}_{S}\rightarrow
\pi_*(\Sym^2\mathcal{T}_{M/S})$,
van Geemen and de Jong \cite{VanGeemenDeJong:98} analyze necessary conditions
so that this map $\rho$ arises as a symbol of a projective heat
operator. More precisely, one seeks  a map 
$\overline{D}: \mathcal{T}_S \to\bigl(\pi_*\bigl(
\mathcal{D}^{\leq 1}(\mathcal{L})+\mathcal{D}_{M/S}^{\leq
	2}(\mathcal{L})\bigr)\bigr)\big{/}\mathcal{O}_S$,
such the following diagram commutes: 
\begin{equation*}
\begin{tikzcd}
\mathcal{T}_S \arrow[r,"\overline{D}"]\arrow[rrd, "\rho"', bend right=5] &
    \bigl(\pi_*\bigl( \mathcal{D}^{\leq
    1}(\mathcal{L})+\mathcal{D}_{M/S}^{\leq 2}(\mathcal{L})\bigr)\bigr)/\mathcal{O}_S\arrow[r,hook]& \left(\pi_*\mathcal{D}^{\leq 2}(\mathcal{L})\right)\big{/}\mathcal{O}_S \arrow[d,"s_2"]\\
&& \pi_*(\Sym^2\mathcal{T}_{M/S})\ .
\end{tikzcd}
\end{equation*} 
The following theorem is one of the main results in \cite{VanGeemenDeJong:98}
(see \cite[Sec.\ 2.3.7]{VanGeemenDeJong:98}).
It gives an algebro-geometric perspective on Hitchin's 
construction of the flat projective connections for a family of K\"ahler
polarizations in \cite[Thm.\ 1.20]{Hitchin:90}. 

\begin{theorem}[{\sc Existence criteria}] \label{thm:hitchimain}
	Given a symbol map $\rho: \mathcal{T}_S \rightarrow \pi_{*}\operatorname{Sym}^2 \mathcal{T}_{M/S}$, with $M$, $\mathcal{L}$ and $S$
as above, there exists a unique projective heat operator 
	$\overline{D}$ who symbol is $\rho$ if the following three
	conditions are satisfied: 
	\begin{enumerate}
		\item {\rm (Hitchin, van Geemen-de Jong equation)}:
		$KS_{M/S}+ \mu_{\mathcal{L}}\circ \rho=0$ in $\Tcal_S$;
		
\item {\rm (Welters condition)} the cup product: 
		$\cup\,
		[\mathcal{L}]:
		\pi_*\mathcal{T}_{M/S}\rightarrow R^1\pi_*\mathcal{O}_M$
		is an isomorphism;

		\item $\pi_*\mathcal{O}_M=\mathcal{O}_S$.
	\end{enumerate}
	In particular, if the coherent sheaf $\pi_\ast\mathcal{L}$ is locally free,
	then $\mathbb{P}(\pi_\ast\mathcal{L})$ is equipped with a connection.
\end{theorem}

In \cite{BBMP20}, the authors translate Hitchin's proof
of flatness of projective connections into the abstract 
formalism of  \cite{VanGeemenDeJong:98}.
In the set-up of Theorem \ref{thm:hitchimain}, 
they prove the following (see \cite[Thm.\ 4.8.2]{BBMP20}):

\begin{theorem}[{\sc Flatness criteria}]  \label{thm:flat}
	If the following three conditions are satisfied, then the
	projective connection that is a consequence of 
	Theorem \ref{thm:hitchimain} is flat. 
	\begin{enumerate} 
		\item For any local sections 
		$\theta_1$ and $\theta_2$ of $\mathcal{T}_S$, the symmetric vector
		fields $\rho(\theta_i)$ considered as functions on
		$\mathcal{T}^{\vee}_{M/S}$ Poisson commute (for the standard
		symplectic structure).
		\item The map $\mu_{\mathcal{L}}$ is injective.
		\item $\pi_{\ast}\mathcal{T}_{M/S}=0$.
	\end{enumerate}
\end{theorem}

\section{Towards a parabolic Hitchin Symbol} \label{sec:symbol}

In this section we discuss the parabolic analog of the Hitchin symbol. 
This will turn out to be the symbol of a natural second order differential 
operator. The original case of (nonparabolic) vector bundles is due to Hitchin. 
We follow and generalize the discussion 
in \cite{BBMP20}. We begin by recalling the notion of a parabolic Atiyah algebra. 

\subsection{Parabolic bundles and their Atiyah algebras}

Let  $q: \Ccal\rightarrow S$ be a family of smooth projective curves with $n$
marked points given by disjoint sections
$p_1,\cdots,p_n : S\to \Ccal$ of $q$, and let $D=p_1+ \dots + p_n$ be the corresponding
relative divisor in $\Ccal$.
Let $\widehat{\pi}:\widehat{\Ccal} \rightarrow S$
be a family of $\Gamma$-Galois covers of the fibers of $\Ccal$,  ramified along $\widehat{D}$. In particular,
this comes with a natural projection map $p: \widehat{\Ccal}\rightarrow
\Ccal$ such that $p(\widehat{D})=D$. 
In order to analyze parabolic Atiyah algebras for families  of parabolic
bundles on $\Ccal$,
we shall use the notion of $\Gamma$-linearized bundles on the Galois cover
$\widehat \Ccal$.
The reader is referred  to Appendix \ref{sec:gammabundles}  for more details.

Let $\widehat{\Pcal}$ be a family of $\Gamma$-$G$-bundles on
$\widehat\Ccal$, and let $\Pcal$ be the family of parabolic $G$-bundles obtained by applying the invariant push-forward functor.
The \emph{relative parabolic Atiyah algebra} is given by: $$\ParAt_{{\Ccal}/S}(\Pcal):=
p_{*}^{\Gamma}(\At_{\widehat{\Ccal}/S} (\widehat{\Pcal}))\ ,$$ and the
\emph{strongly parabolic Atiyah algebra} is given by:
$$\SParAt_{{\Ccal}/S}({\Pcal}):=p_*^{\Gamma}(\At_{\widehat{\Ccal}/S}(\widehat{\Pcal})(-\widehat{D}))\ .$$
Similarly, we define the sheaf of \emph{parabolic endomorphisms}
$\ParEnd(\Pcal)$ by $p_*^{\Gamma}(\ad(\widehat{\mathcal{P}}))$,
and the \emph{strongly parabolic endomorphisms $\SParEnd(\Pcal)$} by
$p_*^{\Gamma}(\ad(\widehat{\mathcal{P}})(-\widehat{D}))$. 

Just as in
the case of parabolic vector bundles, these sheaves fit into the following fundamental exact sequences 
\begin{align}
\begin{split} \label{eqn:fund-atiyah}
0 \lra \ParEnd(\Pcal) \lra \ParAt_{{\Ccal}/S}(\Pcal)\lra
\mathcal{T}_{{\Ccal}/S}(-D)\lra 0\ ;\\
0 \lra\SParEnd(\Pcal)
\lra\SParAt_{{\Ccal}/S}(\Pcal)\lra\mathcal{T}_{{\Ccal}/S}(-D)\lra 0\ .
\end{split}
\end{align}
Also, as in the case of parabolic vector bundles we get the following quasi-Lie algebra:
\begin{equation}\label{eqn:quasiLiealgebra}
0 \lra\Omega_{{\Ccal}/S}
\lra(\ParAt_{{\Ccal}/S}(\Pcal)(D))^{\vee}\lra(\SParEnd(\Pcal)(D))^{\vee}\lra
    0\ . 
\end{equation}
The Cartan-Killing form $\kappa_{\gfrak}$ on $\gfrak \,=\, \text{Lie}(G)$ gives an  identification 
\begin{equation} \label{eqn:CK}
\nu_\gfrak^{-1}: (\SParEnd(\Pcal)(D))^{\vee}\isorightarrow
\ParEnd(\Pcal)\ .
\end{equation}
A more explicit description of these bundles in Lie theoretic terms goes as follows: Let 
$\mathfrak{n}_i$ be the nilradical 
of the Lie algebra of the parabolic subgroup $P_i$.  
Consider the adjoint bundle $\ad(\mathcal{P})$ of the parabolic bundle $\mathcal{P}$. 
The sheaf of strongly parabolic (respectively, parabolic) endomorphisms is
the subsheaf $\ad(\mathcal{P})$ such that the residue at $p_i$ lies in the Lie algebra $\mathfrak{n}_i$ 
(respectively, in Lie algebra of $P_i$) for each $1\leq i\leq n$.

\subsection{Some canonical maps}\label{sec:canonical}

Now assume the family  $\Ccal\rightarrow S$ to be versal with respect to the
divisor $D$.
Universal bundles on relative moduli spaces of bundles exist locally in the
\'etale topology, and moreover both the associated Atiyah algebra and the
adjoint bundle  glue together to extend  globally. 
For convenience of exposition we can therefore assume the existence of a universal bundle 
$\mathcal{P}$  on the family of curves
$\mathfrak{X}_{{G}}^{par}/{M}_{{G}}^{par,rs}$
with parabolic structure supported on a relative divisor $D$ base changed
to ${M}_{{G}}^{par,rs}$. 
We have the following useful diagram:
\begin{equation}\label{ud}
\begin{tikzcd}
\mathfrak{X}_{{G}}^{par}:= \Ccal \times_{S}{M}_{{G}}^{par,rs} \arrow[rd,"\pi_c"]\arrow[from=r, bend right,"p_i"']\arrow[r, "\pi_n"] \arrow[d, "\pi_w"]
& {M}_{{G}}^{par,rs} \arrow[d, "\pi_e" ] \\
\Ccal \arrow[r, "\pi_s"]\arrow[from=r, bend left,"p_i"]
& S.
\end{tikzcd}
\end{equation}
The above map $\pi_c:\Xfrak_G^{par}\to S$ is defined by $\pi_c:=\pi_s\circ\pi_w=\pi_e\circ\pi_n$. 
Recall the duality in \eqref{eqn:CK}.
There is a canonical inclusion map
\begin{equation}\label{equation:caninclusion}
\Spar(\Pcal) \hookrightarrow \Par(\Pcal)
\end{equation}
whose quotient is supported on $D$. 
Composing the evaluation map 
$$\pi_{n}^*\pi_{n*} \left(\Spar(\Pcal)
\otimes \pi_w^*\Omega_{\Ccal/S}(D) \right)\lra \Spar(\Pcal
) \otimes \pi_w^* \Omega_{\Ccal/S}(D)$$
followed by
\eqref{equation:caninclusion} (tensored with 
$\pi^*_{w}\Omega_{\Ccal/S}$), we obtain the following:
$$\pi_{n}^*\pi_{n*} \left(\Spar(\Pcal)
\otimes \pi_w^*\Omega_{\Ccal/S}(D) \right)\lra \Par(\Pcal)
\otimes \pi_{w}^*{\Omega}_{{\Ccal}/S}(D)\ .$$
Taking duals and  applying Serre duality, and then using the identification via
$\nu_{\gfrak}^{-1}$ in eq. \eqref{eqn:CK}, we get that
\begin{align*}
(\Par(\Pcal))^{\vee}\otimes
    \pi_w^*\mathcal{T}_{{\Ccal}/S}(-D)\lra&\,\,\pi_{n}^*\Bigl(\pi_{n*}
    \bigl(\Spar(\Pcal)
\otimes \pi_w^*\Omega_{\Ccal/S}(D) \bigr)^{\vee}\Bigr)\\
    \cong&\pi_{n}^*R^1\pi_{n*}\Bigl(\bigl(\Spar(\Pcal)(D)\bigr)^{\vee}\Bigr)
\cong \pi_n^*\bigl( R^1\pi_{n*} \Par(\Pcal) \bigr)\ .
\end{align*}
This, in turn, gives a map 
$\pi_{w}^*\mathcal{T}_{{\Ccal}/S}(-D)\to \Par(\Pcal)
\otimes \pi_n^*\left( (R^1\pi_{n*} \Par(\Pcal) )\right)$. 
Applying $R^1\pi_{n*}$ and the push-pull formula, we obtain a morphism
\begin{equation*}\label{equation:rewriting1}
R^1\pi_{s}^*\mathcal{T}_{{\Ccal}/S}(-D) \lra R^1\pi_{n*} \left( \Par(\Pcal)
\right)\otimes \left(R^1\pi_{n*} \left(\Par(\Pcal)\right)\right).
\end{equation*}
Further applying $\pi_{e*}$ and identifying $\pi_{e*}\mathcal{T}_{{M}^{par}/S}$, we get a map
\begin{equation}\label{equation:workingsymbol}
\rho_{sym}: R^1\pi_{s*} \mathcal{T}_{{\Ccal}/S}(-D)\lra \pi_{e*}\left(
\mathcal{T}_{{M}_{{G}}^{par,rs}/S}^{\otimes 2}\right).
\end{equation}

We briefly recall the notion of a strongly parabolic Higgs bundle on the
family $\Ccal\to S$.
Let $\mathcal{P}$ be a parabolic $G$ bundle on a curve $C$ with weights
$\bm \alpha$, and consider  the sheaf of strongly parabolic endomorphisms $\Spar(\Pcal)$. 
A strongly parabolic Higgs pair $(\Pcal, \theta)$ consists of a parabolic
bundle $\Pcal$ and a section $\theta$ of $\Spar(\Pcal)\otimes
\Omega_{\Ccal/S}(D)$. 
We refer the reader to \cite[Sec.\ 3--4]{BiswasRamanujan} for the notion of semistability and 
the construction of the moduli space $\mathcal{H}_{\bm \alpha, G}^{par,ss}$
(or  simply  denoted by $\mathcal{H}_{G}^{par,ss}$) (see also \cite[Sec.\
5]{BaragliaKamgarpourVarma:19}, \cite[Sec.\ 5]{Faltings:94}). 

The Hitchin map assigns to a parabolic Higgs pair $(\mathcal{P},\theta)$
the evaluation on  $\theta$ of  a  basis of invariant polynomials on $\gfrak$.
Since $G$ is simple, the lowest degree is quadratic; it
produces a map:
$$\operatorname{Hit}: \mathcal{H}^{par,ss}_{G} \lra\pi_{s*}{\Omega}_{\Ccal/S}^{\otimes 2}(D)
\ ,$$
where ${\Omega}_{\Ccal/S}^{\otimes 2}(D)$
is the space of holomorphic relative quadratic differentials with simple
poles along the divisor $D$. 
Now consider the multiplication map
$$R^1\pi_{n*}\mathcal{T}_{\mathfrak{X}_{G}^{par}/{M}_{G}^{par,rs}}(-D)\otimes
\pi_{n*}\bigl(\Spar(\Pcal)\otimes
\Omega_{\mathfrak{X}_{G}^{par}/{M}^{par,rs}}(D)\bigr)
\lra R^1\pi_{n*} \Par(\Pcal)\ .$$ 
This gives the following map:
$$
R^1\pi_{n*}\mathcal{T}_{\mathfrak{X}_{G}^{par}/{M}_{G}^{par,rs}}(-D)\lra
\bigl(\pi_{n*}\bigl(\Spar(\Pcal)\otimes
\Omega_{\mathfrak{X}_{G}^{par}/{M}_{G}^{par,rs}}(D)\bigr)\bigr)^{\vee}\otimes
R^1\pi_{n*} \Par(\Pcal)\ ,
$$
which, by relative Serre duality \eqref{eqn:CK}, and after applying $\pi_{e*}$ (see \eqref{ud})
together with symmetrization, gives a map 
\begin{equation}\label{rhit}
\rho_{Hit}: R^1\pi_{s*}\mathcal{T}_{\Ccal/S}(-D)\lra \pi_{e*}
    \operatorname{Sym}^2 \mathcal{T}_{{M}_{G}^{par,rs}/S}\ .
\end{equation}
Observe that the cotangent bundle $\mathcal{T}^{\vee}_{{M}^{par,rs}_{G}/S}$ 
embeds into $\mathcal{H}^{par,ss}_{G}$.
We rewrite the Hitchin map via the following commutative diagram as in the nonparabolic case: 

\begin{equation}\label{equation:Hitchinmapasincommonlit}
\begin{tikzcd}
\mathcal{T}^{\vee}_{{M}^{par,rs}_{G}/S}\arrow[r, "\Delta"] \arrow[rd, "\operatorname{Hit}"']
& \mathcal{T}^{\vee}_{{M}^{par,rs}_{G}/S}
\otimes \mathcal{T}^{\vee}_{{M}^{par,rs}_{G}/S}  \arrow[d, "\operatorname{Tr}"] \\
& \pi_{n*}{\Omega}^{\otimes
    2}_{\mathfrak{X}^{par}_{G}/{M}_{G}^{par,rs}}(D)\ .
\end{tikzcd}
\end{equation}
Here, $\Delta$ is the diagonal map, and the operator $\operatorname{Tr}$ is the
pairing given by symmetric form on $\Spar(\Pcal)$ defined by the
Killing form $\kappa_{\gfrak}$; recall that $\mathcal{T}^{\vee}_{{M}^{par,rs}_{G}/S}$ is given by sections of
$\Spar(\Pcal) \otimes \Omega_{\mathfrak{X}^{par}_{{G}}/{M}_{{G}}^{par,rs}}(D)$. Composing with $\pi_{e*}$ and applying
relative Serre duality we get that the dual of the vertical map $\operatorname{Tr}$ in
\eqref{equation:Hitchinmapasincommonlit} is $\rho_{Hit}$ in \eqref{rhit}.  
The two maps $\rho_{Hit}$ and $\rho_{sym}$ (constructed in \eqref{equation:workingsymbol}) are hence identified.

\begin{proposition}\label{ap} The map $\rho_{Hit}$ in \eqref{rhit} coincides with $\rho_{sym}$ given in 
\eqref{equation:workingsymbol}. 
\end{proposition}

Proposition \ref{ap} was proven in the (nonparabolic) vector  bundle case in  \cite[Lemma 4.3.2]{BBMP20}.  

\subsection{Deformation of ${M}_{{G}}^{par,rs}$ via pointed curves}

Recall that we have an isomorphism between the moduli space of
parabolic bundles with fixed parabolic weights $\bm{\lambda}$ on a curve
$C$ and the moduli space of $\Gamma$-$G$-bundles on a Galois cover 
$\widehat{C}\to C$ of type $\bm\tau$. Here, the cover $\widehat{C}$ and type
are related to the parabolic weights. We refer the reader to Appendix \ref{sec:gammabundles} for more details.
We will need the following lemma, the proof of which is straightforward.

\begin{lemma}
\label{lemma:Grothendieckcollapse}
There is a natural isomorphism $\pi^*_{w}\mathcal{T}_{\Ccal/S}(-D)\, \stackrel{\sim}{\longrightarrow}\, \TXM(-D)$,
where $\pi_w$ is the map in \eqref{ud}. Furthermore:
\begin{enumerate}
	\item $R^1\pi_{c*} \bigl(\TXM (-D)\bigr)\cong
        R^1\pi_{s*}\bigl(\mathcal{T}_{\Ccal/S}(-D)\bigr)$;
	\item $R^1\pi_{n*} \bigl( \TXM (-D)\bigr)\cong \pi_{e}^* R^1\pi_{c*}
        \bigl(\TXM(-D)\bigr)$,
\end{enumerate}
where the maps are as in \eqref{ud}.
\end{lemma}

Consider the relative parabolic Atiyah algebra:
$${}^{par}\At_{\mathfrak{X}^{par}_{G}/{M}_{G}^{par,rs}}(\mathcal{P})
:=p_*^{\Gamma}\bigl(\At_{\widehat{C}\times_{S}{M}^{par,rs}_{G}/
{M}_{G}^{par,rs}}(\widehat{\Pcal})\bigr)\ ,$$ 
and the fundamental exact sequence (cf.\ \eqref{eqn:fund-atiyah}) known as the relative Atiyah sequence:
\begin{equation} \label{eqn:fes1}
0\lra \Par(\Pcal) \lra
{}^{par}{\At}_{\mathfrak{X}^{par}_{G}/{M}_{G}^{par,rs}}(\Pcal)
\lra \TXM(-D) \lra 0\ .
\end{equation}
Now since $\pi_{n*} \TXM(-D)=0$ and $R^2\pi_{n*} \Par(\Pcal)=0$, 
applying $R^1\pi_{n*}$ to the above we get the short exact sequence 
\begin{equation}\label{equation:basicatiyah}
0 \to R^1\pi_{n*} \Par(\Pcal)
\lra R^1\pi_{n*}\bigl( {}^{par}\AxM (\mathcal{P})\bigr)\lra R^1\pi_{n*}
\TXM (-D)\to 0\ .
\end{equation}
The relative extension class of the exact sequence in \eqref{equation:basicatiyah} is an element
\begin{align}
\begin{split} \label{eqn:alphaclass}
\alpha(\mathcal{P}, \bm{\lambda}) &\in  R^1\pi_{e*}(( R^1\pi_{n*}\TXM (-D))^{\vee}\otimes R^1\pi_{n*}\Par(\Pcal))\\
&\cong 
R^1\pi_{e*}(\pi_e^*(R^1\pi_{s*}\mathcal{T}_{\mathcal{C}/S}(-D))^{\vee}\otimes
R^1\pi_{n*}\Par(\Pcal))\\
&\cong R^1\pi_{e*}(\pi_e^*(R^1\pi_{c*}\TXM(-D))^{\vee}\otimes
R^1\pi_{n*}\Par(\Pcal))\ .
\end{split}
\end{align}
The last two isomorphisms are constructed using Lemma \ref{lemma:Grothendieckcollapse}.
The exact sequence of tangent sheaves induced by the map $\pi_e:
{M}_{G}^{par,rs}\rightarrow S$ is: 

\begin{equation} \label{eqn:tangent}
0\lra \mathcal{T}_{{M}_{G}^{par,rs}/S} \lra
\mathcal{T}_{{M}_{G}^{par,rs}}\lra
\pi^*_{e}\mathcal{T}_{S}\lra 0\ .
\end{equation}
Since  by assumption the family of pointed curves is versal, 
the Kodaira-Spencer map 
gives an isomorphism 
$KS_{\Ccal/S}:\mathcal{T}_S\cong R^1\pi_{s*} \mathcal{T}_{\Ccal/S}(-D)$,
which, 
pulling back via $\pi_e$ and using  Lemma \ref{lemma:Grothendieckcollapse}, gives 
\begin{equation}\label{equation:rewritingpullback}
\pi_{e}^*\mathcal{T}_S \cong \pi_e^*
    R^1\pi_{s*}\bigl(\mathcal{T}_{\Ccal/S}(-D)\bigr)\cong R^1\pi_{n*}
    (\TXM(-D))\ .
\end{equation}
The identification in \eqref{equation:rewritingpullback} and the
equivariant 
version of  \cite[eq.\ (3.10)]{ST} together produce the following commutative
diagram,
which relates \eqref{equation:basicatiyah} and \eqref{eqn:tangent}:
\begin{equation*}\label{equation:fundamentalcom}
\begin{tikzcd}
 R^1\pi_{n*}\Par(\Pcal)
\arrow[r,hook] &R^1\pi_{n*} {}^{par}\AxM(\mathcal{P}) \arrow[r, ->>]
    &R^1\pi_{n*}\left(\TXM(-D)\right)\\
\TMS \quad	\arrow[r,hook] \arrow[u,"\cong"] &
    \mathcal{T}_{{M}_{G}^{par,rs}}\arrow[r,->>]\ar[u] &
    \arrow[u,"\cong",swap]   \pi_e^*\mathcal{T}_{S}  
\end{tikzcd}
\end{equation*}
The Kodaira-Spencer class
for the family $\pi_e:{M}_{G}^{par,rs}\rightarrow S$ gives a map
$$KS_{{M}_{G}^{par,rs}/S}: \mathcal{T}_S 
\longrightarrow R^1\pi_{e*} \TMS\cong R^1\pi_{e*}\left(R^1\pi_{n*}\Par(\Pcal) \right).$$
The cup product by $\alpha:=\alpha(\Pcal, \bm \lambda)$ produces maps 
\begin{equation*}
\begin{tikzcd}R^1\pi_{c*}\bigl(\TXM(-D)\bigr)\ar[d,"\cup\, \alpha"] \\
R^1\pi_{c*}\bigl(\TXM(-D)\bigr)\otimes R^1\pi_{e*}(\pi_e^*(R^1\pi_{c*}\TXM(-D))^{\vee}\otimes
R^1\pi_{n*}\Par(\Pcal))\ar[d,"\cong"]\\
\Bigl(R^1\pi_{c*}\bigl(\TXM(-D)\bigr)\otimes
    \bigl((R^1\pi_{c*}\TXM(-D)\bigr)\Bigr)^{\vee}\otimes
R^1\pi_{e*}(R^1\pi_{n*}\Par(\Pcal))\ar[d]\\
R^2\pi_{c*}\Par(\Pcal)\cong R^1\pi_{e*}\bigl(R^1\pi_{n*}\Par(\Pcal)\bigr)\ .
\end{tikzcd}
\end{equation*}
The isomorphism in the last step uses the identification
$R^1\pi_{n*}\Par(\Pcal)\cong\TXM$,
along with the facts that $M_{G}^{par,rs}$ has no global tangent vector fields relative to $S$ (cf. \ Lemma \ref{lem:tangentzero}) and $\pi_{n*}\Par(\Pcal)$ is zero. This forces the Grothendieck spectral sequence to collapse. 

We may summarize the  discussion and identifications above with the
following commutative diagram:
\begin{equation}\label{equation:Phidefinition}
\begin{tikzcd}
\mathcal{T}_S \arrow[r,"\cong"', "KS_{\Ccal/S}"]\arrow[rd,
"KS_{{M}_{G}^{par,rs}/S}"'] &  R^1\pi_{s*}
(\mathcal{T}_{\Ccal/S}(-D)) \arrow[d, "\Phi"] \arrow[r,
"\cong"] & R^1\pi_{c*} \bigl(\TXM( -D)\bigr) \arrow[ld,
"\cup\, \alpha"]\\ 
&  R^1\pi_{e*}\TMS\cong R^1\pi_{e*}\bigl(R^1\pi_{n*}\Par(\Pcal)\bigr)\ .&
\end{tikzcd}
\end{equation}
Here $\Phi$ is the map induced by the cup product with the class $\alpha(\mathcal{P}, \bm{\lambda})$
(see eq. \eqref{eqn:alphaclass}) preceded by the isomorphism of
$R^1\pi_{s*}\bigl(\mathcal{T}_{\Ccal/S}(-D)\bigr)$ with $R^1\pi_{c*}\bigl(
\TXM(-D)\bigr)$ given in Lemma \ref{lemma:Grothendieckcollapse}. 

\subsection{A fundamental commutative diagram} 

Consider  $R^1\pi_{n*}$ of the sequence \eqref{eqn:fund-atiyah} applied to
$\SParAt_{\mathfrak{X}^{par}_{G}/{M}^{par,rs}_{G}}(\Pcal)$, where $\pi_n$ is the map in \eqref{ud}:
\begin{equation}
    \begin{tikzcd}
        0\arrow[r]
        &	R^1\pi_{n*} (\pi_{w}^*\Omega_{\Ccal/S}) \arrow[r]
        \arrow[d, phantom, ""{coordinate, name=Z}]
        & R^1\pi_{n*}((\SParAt_{\mathfrak{X}^{par}_{G}/{M}^{par,rs}_{G}}(\Pcal)(D))^{\vee})
        \arrow[dl, rounded corners, to path={ --
        ([xshift=2ex]\tikztostart.east) |- (Z)  -|
        ([xshift=-2ex]\tikztotarget.west) -- (\tikztotarget)}] 
        \\
        &          R^1\pi_{n*}((\Spar(\Pcal)(D))^{\vee})\lra 0\ . &
    \end{tikzcd}
    \label{eqn:fes2}
\end{equation}
Let $\beta:=\beta(\mathcal{P},\bm{\lambda})$ be the relative extension class with respect to $\pi_{e}$ (see \eqref{ud}) of the extension \eqref{eqn:fes2}.
Then we have a diagram:

\begin{equation}\label{equation:imptriangle}
\begin{tikzcd}
R^1\pi_{s*}\mathcal{T}_{\Ccal/S}(-D) \arrow[r,"-\Phi"] \arrow[rd, "\rho_{sym}",
swap] & R^1\pi_{e*}\TMS\\
& \pi_{e*}\bigl(\operatorname{Sym}^2 \TMS\bigr)  \arrow[u,"\cup\, \beta"']\ .
\end{tikzcd}
\end{equation}
We have the following key result which relates all  three maps. 
In the  (nonparabolic) vector bundle case, this was proven in \cite[Prop.\ 4.7.1]{BBMP20}.

\begin{theorem} \label{thm:cup-product}
    The diagram  \eqref{equation:imptriangle} commutes. In other words,
$$\Phi + \cup\, \beta(\mathcal{P},\bm{\lambda})\circ \rho_{sym}=0$$
as a morphism $R^1\pi_{s*} \mathcal{T}_{\Ccal/S}\rightarrow
R^1\pi_{e*}\TMS$.   
\end{theorem}

\begin{proof}
Pull back the short exact sequence in \eqref{eqn:fes2} to $\mathfrak{X}^{par}_{G}$ via the map $\pi_{n}$
in \eqref{ud}.
Tensoring the resulting sequence with $\Par(\Pcal)$ we obtain the following exact sequence 
\begin{equation*}\label{equation:srew}	\adjustbox{scale=1.0}{
	\begin{tikzcd}
	\Par(\Pcal)\otimes \pi_{n}^*\left(R^1\pi_{n*}\pi_{w}^*\Omega_{\Ccal/S}\right)\arrow[r,hook]
	&\Par(\Pcal) \otimes \pi_{n}^*\bigl(R^1\pi_{n*}\bigl( (  
	\SParAt_{\mathfrak{X}^{par}_{G}/{M}^{par,rs}_{G}}(\Pcal)(D))^{\vee}\bigr)\bigr) \arrow[d, twoheadrightarrow]\\
	& \Par(\Pcal) 
			\otimes\pi_n^*\bigl(R^1\pi_{n*}\bigl((\Spar(\Pcal)(D)
			)^{\vee}\bigr)\bigr)\ .
	\end{tikzcd}
}
\end{equation*}
Using $\kappa_{\gfrak}$, we can rewrite this as 

\begin{equation*}\label{equation:srew1}
\adjustbox{scale=.9}{
	\begin{tikzcd}
	\Par(\Pcal)\otimes
		\pi_{n}^*\left(R^1\pi_{n*}\pi_{w}^*\Omega_{\Ccal/S}\right)\arrow[r,hook]&\Par(\Pcal)
        \otimes\pi_{n}^*\bigl(R^1\pi_{n*}\bigl( ( 
			\SParAt_{\mathfrak{X}^{par}_{G}/{M}^{par,rs}_{G}}(\Pcal)(D))^{\vee}\bigr)\bigr)\arrow[d, twoheadrightarrow]\\
	& \Par(\Pcal) \otimes\pi_n^*\left(R^1\pi_{n*}\left(\Par(\mathcal{P})\right)\right).
	\end{tikzcd}}
\end{equation*}
The assumptions ensure that $R^1\pi_{n*}\pi_{w}^*\Omega_{\Ccal/S}=\mathcal{O}_{{M}^{par,rs}_{G}}
$.  Dualize \eqref{eqn:fes1} to get
\begin{equation*}
0 \lra \pi_w^*\Omega_{\Ccal/S}(D) \lra
\bigl(\ParAt_{\mathfrak{X}_{G}^{par}/{M}^{par,rs}_{G}}(\Pcal)\bigr)^{\vee}\lra
\Par(\Pcal)^{\vee}\lra 0\ .
\end{equation*}
Tensoring by $\Par(\Pcal)\otimes \pi^*_{w}\mathcal{T}_{\Ccal/S}(-D)$ and taking the duals (outside bracket)
we get the short exact sequence 
\begin{equation*}\label{equation:reneww}
\adjustbox{scale=0.85}{
	\begin{tikzcd}
        0\lra 
	\ParEnd(\mathcal{P}) \arrow[r]
        \arrow[d, phantom, ""{coordinate, name=Z}]
        &{{\Par(\Pcal)\otimes}{
			\bigl((\AxxM )\otimes
			\pi_w^*\Omega_{\Ccal/S}(D)\bigr)^{\vee}}}
        \arrow[dl, rounded corners, to path={ --
        ([xshift=2ex]\tikztostart.east) |- (Z)  -|
        ([xshift=-2ex]\tikztotarget.west) -- (\tikztotarget)}] 
       \\ 
	{{\Par(\Pcal)\otimes}{\left(\Par(\Pcal)\otimes 
            \pi_w^*\Omega_{\Ccal/S}(D)\right)^{\vee}}}
            \lra 0\ .
            &
	\end{tikzcd}}
\end{equation*}
Now observe that the dual of the evaluation gives maps 
\begin{align*}
\bigl(\AxxM\otimes \pi_w^*\Omega_{\Ccal/S}(D)\bigr)^{\vee}&\lra
    \bigl(\pi_n^*\pi_{n*}\bigl(\AxxM\otimes
    \pi_w^*\Omega_{\Ccal/S}(D)\bigr)\bigr)^{\vee}\\
&=\pi_n^*\bigl(\pi_{n*}\bigl(\AxxM\otimes
    \pi_w^*\Omega_{\Ccal/S}(D)\bigr)\bigr)^{\vee}\\
&\cong \pi_n^*\bigl( R^1\pi_{n*}\bigl(( \AxxM(D))^{\vee}\bigr)\bigr)\\
&\lra \pi_n^*\bigl( R^1\pi_{n*}\bigl((
    \AXM(\mathcal{P})(D))^{\vee}\bigr)\bigr)\ .
\end{align*}
In the above equation we have used the isomorphism 
$$R^1\pi_{n*}((\AxxM(D))^{\vee})\cong (\pi_{n*}(\AxxM(D)\otimes
\pi_w^*\Omega_{\Ccal/S}))^{\vee}$$ coming from relative Serre duality and
the dual of the natural inclusion map $$\AXM(\mathcal{P}) \hookrightarrow
\AxxM\ .$$

We now reverse engineer the construction of the Hitchin morphism $\rho_{sym}$:
\begin{equation*}\label{equation:compensatingforthehitchintrick}
\begin{aligned}
\left(\Par(\Pcal)\otimes \pi_w^*\Omega_{\Ccal/S}(D)\right)^{\vee}
&\lra \pi_n^*\pi_{n*}\left(\left(\Par(\Pcal)(D) \otimes \pi_{w}^*\Omega_{\Ccal/S}\right) \right)^{\vee}\\
&\lra \pi_n^*\pi_{n*}\left(\left(\Spar(\Pcal)(D) \otimes \pi_{w}^*\Omega_{\Ccal/S}\right) \right)^{\vee}\\
&\cong
\pi_n^*\left(R^1\pi_{n*}\left(\left(\Spar(\mathcal{P})(D)\right)^{\vee}\right)
\right)\qquad \mbox{(by relative Serre duality)}\\
&\cong \pi_n^*\left(R^1\pi_{n*}\left(\Par(\mathcal{P})  \right)\right) \qquad \mbox{(by trace pairing)}.
\end{aligned}
\end{equation*}
Consider the natural inclusion map $\pi_{w}^{*}\mathcal{T}_{\mathcal{C}/S}(-D) \hookrightarrow \Par(\Pcal)\otimes
\Par(\Pcal)^{\vee}{\otimes
	\pi_w^*\mathcal{T}_{\Ccal/S}(-D)}$, and pull back the short exact sequence 
$$
\Par(\Pcal)\hookrightarrow {{\Par(\Pcal)\otimes
		(\AxxM)^{\vee}}{\otimes
		\pi_w^*\mathcal{T}_{\Ccal/S}(-D)}} \twoheadrightarrow ( \Par(\Pcal) \otimes \pi_w^*\Omega_{\Ccal/S}(D))^{\vee}.
$$
Finally, by \cite[Lemma 4.5.1]{BBMP20}, we obtain an isomorphism of the extensions: 
\begin{equation}\label{equation:stupidextension}
\adjustbox{scale=0.9}{
	\begin{tikzcd}
	\Par(\Pcal) \arrow[r,hook] \arrow[d, equal] & \AxxM \arrow[r,twoheadrightarrow,"(-1)"] \arrow[d, hook]& \pi_w^*\mathcal{T}_{\Ccal/S}(-D)  \arrow[d, hook]\\
	\Par(\Pcal) \arrow[r,hook] \arrow[d,equal]& {{\Par(\Pcal)\otimes
			(\AxxM)^{\vee}}{\otimes
			\pi_w^*\mathcal{T}_{\Ccal/S}(-D)}} \arrow[r,twoheadrightarrow ] \arrow[d,equal] &
	{{\Par(\Pcal)\otimes}{
			( \Par(\Pcal) \otimes \pi_w^*\Omega_{\Ccal/S}(D))^{\vee}}} \arrow[d,equal]\\
	\Par(\Pcal) \arrow[r,hook] & {{\Par(\Pcal)\otimes}{
			(\AxxM\otimes \pi_w^*\Omega_{\Ccal/S}(D))^{\vee}}}
	\arrow[r,twoheadrightarrow ] & {{\Par(\Pcal)\otimes}{
			( \Par(\Pcal) \otimes \pi_w^*\Omega_{\Ccal/S}(D))^{\vee}}}.
	\end{tikzcd}}
\end{equation}
Here, the minus sign $(-1)$ indicates  the negative of the projection map.
Following the case of vector bundles in
\cite{BBMP20}, after composing we arrive at a commutative diagram
\begin{equation}\label{equation:snakingaround}
\adjustbox{scale=0.9}{
	\begin{tikzcd}[column sep=small]
	\Par(\Pcal) \arrow[r,hook]\arrow[d, equal] & \AxxM\arrow[r,twoheadrightarrow,"(-1)"]\arrow[d] &\pi^*_w\mathcal{T}_{\Ccal/S}(-D)\arrow[dd]\\
	{	\Par(\Pcal)}\arrow[r,hook]\arrow[dd,equal] & \Par(\Pcal) \otimes
        \bigl( \pi_{n}^*\pi_{n*}( \AxxM\otimes
        \pi_w^*\Omega_{\Ccal/S}(D))\bigr)^{\vee} \arrow[dd]\arrow[rd,twoheadrightarrow, bend right=10] &  \\
	&&\Par(\Pcal)\otimes
	\bigl( \pi_{n}^*\pi_{n*}	 \Par(\Pcal) \otimes
        \pi_w^*\Omega_{\Ccal/S}(D)\bigr)^{\vee}\arrow[dd]\\
	\Par(\Pcal)\arrow[r,hook] & {{\Par(\Pcal) \otimes }
		{ \pi_{n}^*\bigl(R^1\pi_{n*}\bigl( (
			\AXM(\Pcal)(D))^{\vee}\bigr)\bigr)}}
	\arrow[dr,twoheadrightarrow, bend right=10] &\\&& {{\Par(\Pcal)
			\otimes}{ \pi_n^*\bigl(R^1\pi_{n*}(\Par(\Pcal))\bigr)}}.
	\end{tikzcd}}
\end{equation}
Now we take $R^1\pi_{n*}$ of the exact sequences in the first and third
rows in \eqref{equation:snakingaround} to obtain
\begin{equation}\label{equation:movingforward}
\adjustbox{scale=0.88}{
	\begin{tikzcd}
	R^1\pi_{n*} \Par(\Pcal) \arrow[d,equal] \arrow[r,hook] &
	R^1\pi_{n*}(\AxxM)
	\arrow[d]\arrow[r,twoheadrightarrow,"(-1)"]& R^1\pi_{n*} (\pi_w^*\mathcal{T}_{\Ccal/S}(-D))\arrow[d]\\
	R^1\pi_{n*}\Par(\Pcal) \arrow[r,hook] &
	{{R^1\pi_{n*}\Par(\Pcal)\otimes
		}{R^1\pi_{n*}(\AXM(\Pcal)(D))^{\vee}}}\arrow[r,twoheadrightarrow]&
	{{R^1\pi_{n*}\Par(\Pcal) \otimes} {R^1\pi_{n*}(\Par(\Pcal))}}.
	\end{tikzcd}}
\end{equation} 
The connecting homomorphism for $(\pi_{e})_\ast$ gives 
\begin{equation*}
\begin{tikzcd}
R^1\pi_{s*}\mathcal{T}_{\Ccal/S}(-D) \arrow[r,"-\Phi"]\arrow[d, "\rho_{sym}"] & R^1\pi_{e*} \TMS \ar[d,equal] \\
\pi_{e*} \bigl( \TMS \otimes \TMS \bigr) \arrow[r]  & R^1\pi_{e*} \TMS\ .
\end{tikzcd}
\end{equation*}
The negative sign $-\Phi$ 
appears above due to the factor $(-1)$ in  \eqref{equation:stupidextension}; recall that $\Phi$ (see eq.
\eqref{equation:Phidefinition}) is the 
connecting homomorphism for the direct image by $\pi_e$ of
the exact sequence in \eqref{equation:basicatiyah}.
The proof of the theorem will be complete if we can show that the underlying map is 
$\cup\, \beta(\mathcal{P},\bm{\lambda})$.  But this follows 
from the fact that the bottom row of \eqref{equation:movingforward} is 
just the exact sequence 
$$
    \begin{tikzcd}
        0\arrow[r]& \mathcal{O}_{{M}^{par}_{G}}\cong
R^1\pi_{n*}\Omega_{\mathfrak{X}^{par}_{G}/{M}^{par,rs}_{G}}
        \arrow[r] 
        \arrow[d, phantom, ""{coordinate, name=Z}]
        &
R^1\pi_{n*}(\SParAt_{\mathfrak{X}^{par}_{G}/{M}^{par,rs}_{G}}(\Pcal)(D))^{\vee}
        \arrow[dl, rounded corners, to path={ --
        ([xshift=2ex]\tikztostart.east) |- (Z)  -|
        ([xshift=-2ex]\tikztotarget.west) -- (\tikztotarget)}] 
\\
        &
         R^1\pi_{n*} \Par(\Pcal) \lra 0 &
    \end{tikzcd}
    $$
tensored with $\TXM$, and $\beta(\Pcal, \bm{\lambda})$
is the relative extension class of the above with respect to $\pi_e$. 
\end{proof}

\section{Cupping with the parabolic determinant of cohomology}

In this section, we state and prove a 
key result that compares the cupping map by the class of the parabolic
determinant of cohomology to that of the usual determinant of cohomology.
This will be crucial for later arguments.
Let $\vec{P}=(P_1,\dots, P_n)$ be an $n$-tuple of standard parabolic
subgroups, and consider the stack $\mathcal{P}ar_{G}(C,\vec{P})$ of quasi parabolic bundles on a curve as recalled in Definition \ref{def:quasiparabolicstack} and let $\operatorname{Det}(\mathcal{V})$ (or simply $\operatorname{Det}$) denote the determinant of cohomology line bundle on a scheme $T$ parametrizing a family
$\Vcal$ of vector bundles on a smooth projective curve $C$. 
Recall (cf.\  Proposition \ref{prop:Laszlo}) that any line bundle on $\mathcal{P}ar_{G}(C,\vec{P})$ is of the 
form $\operatorname{Det}(\mathcal{E}(\mathcal{V}))^{\otimes a}\bigotimes \mathscr{K}$, where $\mathcal{E}(\mathcal{V})$ is a vector bundle
associated to a chosen representation $\phi: G \rightarrow \SL(V)$, $a \in \mathbb{Q}$ and $\mathscr{K} \in 
\operatorname{Pic}(G/P_1\times\cdots \times G/P_n)\otimes \mathbb{Q}$. We
will refer to the rational number $a$ as the \emph{level} (see 
Definition \ref{def:parabolicdeteG}).

\begin{theorem}\label{thm:indepdence}
	Let $\mathbb{L}$ be an element of
	$\operatorname{Pic}(M^{par,rs}_{G, \bm{\beta}})\otimes \mathbb{Q}$ of
	level $a$. Then as linear maps
	$\pi_{e*}\operatorname{Sym}^2\mathcal{T}_{M^{par,rs}_{G, \bm{\beta}}/S}\,
	\to\, R^1\pi_{e*}\mathcal{T}_{M^{par,rs}_{G,\bm{\beta}}/S}$, we have:  $\cup\, [\mathbb{L}]\,=\,\cup\, a [\operatorname{Det}]$, where $\operatorname{Det}$ is the determinant of cohomology (nonparabolic) line bundle. 
\end{theorem}

Theorem \ref{thm:indepdence} is proved in several steps. The strategy of the proof is to reduce to 
the case of parabolic vector bundles with full flags and apply the technique of 
abelianization by restricting to generic fibers of the Hitchin map.

\subsection{Reduction to the $\SL_r$ case}

Since $G$ is simple (hence semisimple), any short exact sequence of finite dimensional $G$-modules splits. In particular,
for a faithful irreducible $G$-module $V$, the $G$-module $\text{End}(V)$ decomposes as
${\mathfrak g}\oplus W_0$. Fix a complement $W_0$ of the $G$-submodule $\mathfrak g$.
Given an injective homomorphism $G\, \hookrightarrow\, \SL_r({\mathbb C})$,
we have an
embedding $M^{par,ss}_{G,\bm \beta }\hookrightarrow M^{par,ss}_{\SL_r,\bm
\alpha}$ which restricts to a map $f\,:\, M^{par,rs}_{G,\bm \beta
}\,\to\, M^{par,s}_{\SL_r,\bm \alpha}$.
Using the splitting  of the $G$-module ${\slfrak}_r({\mathbb C})$,  the tangent bundle $f^* M^{par,s}_{\SL_r,\bm \alpha}$
splits as
$f^*\mathcal{T}_{M^{par,s}_{\SL_r, \bm{\alpha}}/S}\,=\, \mathcal{T}_{M^{par,rs}_{G, \bm \beta }/S}\oplus W$. This gives
splittings of tensor powers, duals etc. We have the following commutative diagram:
\begin{equation}\label{e21}
\begin{tikzcd}
\pi_*\operatorname{Sym}^2 \mathcal{T}_{ M^{par,s}_{\SL_r,\bm{\alpha}}/S}
\arrow[r, "\cup \mathbb{L}"]\arrow[d]& R^1\pi_*\mathcal{T}_{ M^{par,s}_{\SL_r, \bm{\alpha}}/S}\arrow[d]\\
{\pi_{G}}_*\operatorname{Sym}^2 \mathcal{T}_{ M^{par,rs}_{G, \bm{\beta}}/S}
\arrow[r,"\cup \mathbb{L}"]&R^1{\pi_G}_*\mathcal{T}_{
        M^{par,rs}_{G,\bm{\beta}}/S}\ ,
\end{tikzcd}
\end{equation}
where $\pi\,:\, M^{par,s}_{\SL_r, \bm \alpha}\,\to\, S$ and
$\pi_G\,:\, M^{par,rs}_{G, \bm \tau}\,\to\, S$ are the 
projections (this was earlier denoted by $\pi_{e}$, but here we simply
write $\pi$ and $\pi_G$); the vertical maps in \eqref{e21} are given by the above mentioned splittings.
Here, $\mathbb{L}$ is an element of the rational Picard group of $M^{par,s}_{\SL_r,\bm 
\alpha}$, and $\pi_G\,=\,f\circ \pi$. The homomorphism
$\pi_*\operatorname{Sym}^2 \mathcal{T}_{ M^{par,s}_{\SL_r,\bm{\alpha}}/S} \rightarrow 
{\pi_{G}}_*\operatorname{Sym}^2 \mathcal{T}_{ M^{par,rs}_{G, \bm{\beta}}/S}$ in \eqref{e21} is surjective.
Thus we have proved the following proposition:

\begin{proposition}\label{prop:sum0}
Consider two elements $\mathbb{L}_1$ and 
$\mathbb{L}_2$ in $\operatorname{Pic}(M^{par,s}_{\SL_r,\bm \alpha})\otimes \mathbb{Q}$. If the maps 
$\cup [\mathbb{L}_1]$ and $\cup[\mathbb{L}_2]$ agree on $\pi_*\operatorname{Sym}^2 \mathcal{T}_{ M^{par,s}_{\SL_r,\bm{\alpha}}/S}$, 
then they also agree on ${\pi_{G}}_*\operatorname{Sym}^2 \mathcal{T}_{ 
	M^{par,rs}_{G,\bm{\beta}}/S}$.
\end{proposition}

\subsection{Reduction to the $\SL_r$ with full flags}

In this step, we will show that in order to prove Theorem \ref{thm:indepdence}
it is enough to assume that $\bm \alpha$ corresponds to weights for full flags. This step is only required when $r>2$.

\subsubsection{Changing weights without changing stability}

Let $D\,=\, \{p_1,\, \cdots ,\, p_n\}\, \subset\, C$ be the parabolic divisor.
Consider parabolic vector bundles of rank $r$. For any $1\, \leq\, i\, \leq\, n$, let
\begin{equation}\label{j2}
\alpha_{i,j}\,=\,m_{i,j}/\ell,\ \ \, 1\, \leq\, j\, \leq\, r,
\end{equation}
be the parabolic weights at
$p_i$, where $m_{i,j}$ and $\ell$ are nonnegative integers. Note that for any
$i$, the integers $m_{i,j}$, $1\, \leq\, j\, \leq\, r$, need not be distinct and the weights are assigned to full flags. 
We will reformulate a general notion of parabolic bundles for which the quasiparabolic flags are not necessarily complete in the following way: 
We will set the quasiparabolic flag at each $p_i$ to be complete flags, but two different terms in the filtration can have same parabolic weight. This
reformulation does not alter any of the stability and semistability conditions.

Fix a vector bundle $E$ of rank $r$ on $X$. Let $E_*$ be a parabolic structure
on $E$ of the above type. Let $E'_*$ be another parabolic bundle satisfying the
following conditions:
\begin{enumerate}
	\item The underlying holomorphic vector bundle for $E'_*$ is $E$ itself,
	
	\item the quasiparabolic flag for $E'_*$ coincides with that of $E_*$ at each $p_i$
	(recall that the quasiparabolic flags are complete but two different subspaces of
	$E_{p_i}$ can have same parabolic weight), and
	
	\item for any term $F_{i,j}\, \subset\, E_{p_i}$ of the quasiparabolic flag
	at $p_i$, if $\alpha_{i,j}$ and $\widetilde{\alpha}_{i,j}$ are the weights of $F_{i,j}$ in
	$E_*$ and $E'_*$, respectively, then
	\begin{equation}\label{eqn:choiceofrefinement}
	\big\vert\alpha_{i,j}-\widetilde{\alpha}_{i,j}\big\vert\,\leq\, \frac{1}{3\ell nr^2}\, .
	\end{equation}
\end{enumerate}

\begin{proposition}\label{prs}
The parabolic vector bundle $E'_*$ is stable if the
parabolic vector bundle $E_*$ is stable. Moreover, the parabolic vector bundle $E_*$ is semistable if the
parabolic vector bundle $E'_*$ is semistable.
\end{proposition}
\begin{proof}

Assume that $E_*$ is parabolic stable.
Take any subbundle $0\, \not=\, F\, \subsetneq \, E$. Let $F_*$ denote the parabolic structure on it
induced by $E_*$. Since $E_*$ is parabolic stable, we have
\begin{equation}\label{j4}
\text{par-deg}(F_*)r \, <\, \text{par-deg}(E_*)r'\,,
\end{equation}
where $r'= \text{rank}(F)$. From \eqref{j2} it follows that $\text{par-deg}(E_*)r'- \text{par-deg}(F_*)r$
is an integral multiple of $1/\ell$, and hence \eqref{j4} implies that
\begin{equation}\label{j1}
\text{par-deg}(E_*)r'- \text{par-deg}(F_*)r\, \geq\, \frac{1}{\ell}.
\end{equation}
Let $F'_*$ denote the parabolic vector bundle defined by $F$ equipped with the parabolic structure
induced by $E'_*$. From \eqref{eqn:choiceofrefinement} we have
$$
\text{par-deg}(F'_*) - \text{par-deg}(F_*)\,\leq \, \frac{nr'}{3\ell nr^2}\ \ \text{ and }\ \
\text{par-deg}(E_*) - \text{par-deg}(E'_*)\,\leq \, \frac{nr}{3\ell nr^2}\, .
$$
These imply that
$$
(\text{par-deg}(F'_*) - \text{par-deg}(F_*))r\,\leq \, \frac{1}{3\ell}\ \ \text{ and }\ \
(\text{par-deg}(E_*) - \text{par-deg}(E'_*))r'\,\leq \, \frac{1}{3\ell}\, .
$$
Adding these
$$
(\text{par-deg}(E_*)r'- \text{par-deg}(F_*)r) - (\text{par-deg}(E'_*)r'- \text{par-deg}(F'_*)r)\, \leq\, \frac{2}{3\ell},
$$
and hence using \eqref{j1},
$$
\text{par-deg}(E'_*)r'- \text{par-deg}(F'_*)r\, \geq\, \frac{1}{\ell}- \frac{2}{3\ell}\,=\, \frac{1}{3\ell}\, >\, 0.
$$
Therefore, $E'_*$ is parabolic stable.
Now assume that $E'_*$ is parabolic semistable. So we have
\begin{equation}\label{j5}
\text{par-deg}(F'_*)r \, \leq\, \text{par-deg}(E'_*)r'\,,
\end{equation}
From \eqref{eqn:choiceofrefinement} we have
$$
\text{par-deg}(F_*) - \text{par-deg}(F'_*)\,\leq \, \frac{nr'}{3\ell nr^2}\ \ \text{ and }\ \
\text{par-deg}(E'_*) - \text{par-deg}(E_*)\,\leq \, \frac{nr}{3\ell nr^2}\, .
$$
These imply that
$$
(\text{par-deg}(F_*) - \text{par-deg}(F'_*))r\,\leq \, \frac{1}{3\ell}\ \ \text{ and }\ \
(\text{par-deg}(E'_*) - \text{par-deg}(E_*))r'\,\leq \, \frac{1}{3\ell}\, .
$$
Adding these
$$
(\text{par-deg}(E'_*)r'- \text{par-deg}(F'_*)r) - (\text{par-deg}(E_*)r'- \text{par-deg}(F_*)r)\, \leq\, \frac{2}{3\ell},
$$
So using \eqref{j5},
$$
\text{par-deg}(E_*)r'- \text{par-deg}(F_*)r \, \geq\, - \frac{2}{3\ell} .
$$
But this implies that $\text{par-deg}(E_*)r'- \text{par-deg}(F_*)r \, \geq \, 0$ because
$\text{par-deg}(E_*)r'- \text{par-deg}(F_*)r$ is an integral multiple of
    $1/\ell$.
Hence $E_*$ is parabolic semistable.
\end{proof}

Let ${ \bm{\alpha}}$ be a set of weights defining the parabolic structure. We choose a 
refinement of ${\bm \alpha}$, denoted by ${\widetilde{\bm \alpha}}$, such that for each 
point $p_i$, the weight-tuple $\bm \alpha_i$ consists of distinct weights. The weights
${\widetilde{\bm \alpha}}$ are a choice of weights for full flags such that the 
corresponding weights for the given partial flags is ${\bm \alpha}$. By 
\eqref{eqn:choiceofrefinement}, we can always find ${\widetilde{\bm \alpha}}$ by choosing 
the missing weights small enough such that the natural forgetful map preserves stability 
with respect to ${\widetilde{\bm \alpha}}$ and ${\bm {\alpha}}$. In
particular by Proposition \ref{prs} , we get a 
natural regular map $F\,:\, M^{par,ss}_{\SL_r, \widetilde{\bm \alpha}}\,\to\, 
M^{par,ss}_{\SL_r, {\bm \alpha}}$ fitting in the following commutative diagram:
\begin{equation}\label{eqn:forget}
\begin{tikzcd}
M^{par,ss}_{\SL_r, \widetilde{\bm \alpha}} \arrow[r,"F"]\arrow[rd,"\widetilde{\pi}"'] & M^{par,ss}_{\SL_r, {\bm \alpha}}\arrow[d,"\pi"]\\
& S \ .
\end{tikzcd}
\end{equation}
Let $M_{\widetilde{\bm \alpha}}:=F^{-1}(M^{par,s}_{\SL_r, {\bm \alpha}})$.
Again, by Proposition \ref{prs}, $M_{\widetilde{\bm \alpha}}\subset
M^{par,s}_{\SL_r, \widetilde{\bm \alpha}}$. The map $F$
is fibration by
product of flag varieties. By Lemma \ref{lem:codimension},  the
codimension of the complement of $M_{\widetilde{\bm \alpha}}$ in
$M^{par,s}_{\SL_r, \widetilde{\bm \alpha}}$ is at least three. Hence, we
have the following isomorphisms  (via Hartogs' Theorem):
\begin{equation} \label{eqn:F-iso}
R^1\widetilde{\pi}_*\mathcal{T}_{M^{par,s}_{\SL_r,\widetilde{\bm
    \alpha}}/S}\cong R^1\widetilde{\pi}_*\mathcal{T}_{M^{}_{\widetilde{\bm
    \alpha}}/S} \quad ,\quad \widetilde{\pi}_*\operatorname{Sym}^2\mathcal{T}_{M^{par,s}_{\SL_r,\widetilde{\bm \alpha}}/S}\cong \widetilde{\pi}_*\operatorname{Sym}^2\mathcal{T}_{M^{}_{\widetilde{\bm \alpha}}/S}.
\end{equation}
The differential of $F$, along with the isomorphisms \eqref{eqn:F-iso}, induces natural maps
\begin{eqnarray*}
	&R^1\widetilde{\pi}_*\mathcal{T}_{M^{par,s}_{\SL_r,\widetilde{\bm \alpha}}/S}
	\,\xrightarrow{\ DF\ }\,
    R^1\widetilde{\pi}_*\bigl(DF^*(\mathcal{T}_{M^{par,s}_{\SL_r,{\bm
    \alpha}}/S})\bigr),\\
	&\widetilde{\pi}_*\operatorname{Sym}^2\mathcal{T}_{M^{par,s}_{\SL_r,\widetilde{\bm \alpha}}/S}
	\,\xrightarrow{\ \operatorname{Sym}^2DF\ }\,
    \widetilde{\pi}_*\operatorname{Sym}^2\bigl(DF^*(\mathcal{T}_{M^{par,s}_{\SL_r,{\bm
    \alpha}}/S})\bigr)\ .
\end{eqnarray*}
We have the following lemma:

\begin{lemma}\label{lem:infchange}
	The Leray spectral sequence gives natural isomorphisms:
	\begin{align*}
	&	R^1\widetilde{\pi}_*\bigl(DF^*(\mathcal{T}_{M^{par,s}_{\SL_r,{\bm
        \alpha}}/S})\bigr)
	\,\cong\, R^1{\pi}_*\bigl(\mathcal{T}_{M^{par,s}_{\SL_r,{\bm
        \alpha}}/S}\bigr),\\
	&	\widetilde{\pi}_*\operatorname{Sym}^2\bigl(DF^*(\mathcal{T}_{M^{par,s}_{\SL_r,
			{\bm \alpha}}/S})\bigr)\,\cong\, {\pi}_*\operatorname{Sym}^2
	\bigl(\mathcal{T}_{M^{par,s}_{\SL_r,{\bm \alpha}}/S}\bigr)\ .
	\end{align*}
	
\end{lemma}

\begin{proof}
	For the map $F$ in \eqref{eqn:forget},  space $M^{}_{\widetilde{\bm \alpha}}$ is a
	fiber bundle over the moduli space $M^{par,s}_{\SL_r,\bm\alpha}$, and moreover, the fibers are
	products of flag manifolds. Hence, we have
	\begin{equation}\label{ele1}
	F_*{\mathcal O}_{M^{}_{\widetilde{\bm \alpha}}}\,=\,{\mathcal O}_{M^{par,s}_{\SL_r, \bm \alpha}} \ \ \text{ and }\ \
	R^k F_*{\mathcal O}_{M^{}_{\widetilde{\bm \alpha}}}\,=\, 0
	\end{equation}
	for all $k\, \geq\, 1$. Given any vector bundle $W$ on $M^{par,s}_{\SL_r, \bm \alpha}$, using \eqref{ele1} and the projection
	formula we have
	\begin{equation}\label{ele2}
	F_*F^* W\,=\, W \ \ \text{ and }\ \ R^k F_*F^* W\,=\, 0
	\end{equation}
	for all $k\, \geq\, 1$. From \eqref{ele2} it follows that
	\begin{equation}\label{ele3}
	R^k\widetilde{\pi}_*F^* W\,=\, R^k{\pi}_* W\, .
	\end{equation}
	Now take $W\,=\, {\rm Sym}^2(\mathcal{T}_{M^{par,s}_{\SL_r,{\bm \alpha}}/S})$ in 
	\eqref{ele3}.
\end{proof}

As before, let $\mathbb{L}$ be an element of the rational Picard group
$\operatorname{Pic}(M^{par,s}_{\SL_r, \bm{\alpha}})\otimes \mathbb{Q}$.
Using the isomorphisms in Lemma \ref{lem:infchange} we have the following
diagram:
\begin{equation} \label{eqn:reductiondiagram}
    \hspace{-1.3cm}
      \begin{tikzcd}[column sep=3em]
&&& R^1{\pi}_*\left(\mathcal{T}_{M^{par,s}_{\SL_r,{\bm
          \alpha}}/S}\right)\arrow[from=ddd, bend
          right=80,"\cup \mathbb{L}"]\\
   & R^1\widetilde{\pi}_*\mathcal{T}_{M^{par,s}_{\SL_r,\widetilde{\bm
          \alpha}}/S}  \arrow[r,"\cong"] \arrow[from=d,"\cup \mathbb{L}"]
          &  R^1\widetilde{\pi}_*\mathcal{T}_{M^{}_{\widetilde{\bm
          \alpha}}/S}\arrow[r,"DF"]\arrow[ur]\arrow[from=d,"\cup
          \mathbb{L}"]&
          R^1\widetilde{\pi}_*\left(DF^*\left(\mathcal{T}_{M^{par,s}_{\SL_r,{\bm
          \alpha}}/S}\right)\right)\arrow[u, "\cong", swap]\\
&\widetilde{\pi}_*\operatorname{Sym}^2\mathcal{T}_{M^{par,s}_{\SL_r,\widetilde{\bm
          \alpha}}/S}\arrow[r,"\cong"]&\widetilde{\pi}_*\operatorname{Sym}^2\mathcal{T}_{M^{}_{\widetilde{\bm
          \alpha}}/S}\
          \arrow[r,"\,\,{\operatorname{Sym}^2DF}\,\,",]\arrow[dr]&\
          \widetilde{\pi}_*\operatorname{Sym}^2\left(DF^*\left(\mathcal{T}_{M^{par,s}_{\SL_r,{\bm
          \alpha}}/S}\right)\right)\arrow[d,"\cong"]\arrow[u,"\cup
          \mathbb{L}",swap]\\
&&&{\pi}_*\operatorname{Sym}^2\left(\mathcal{T}_{M^{par,s}_{\SL_r,{\bm \alpha}}/S}\right).
\end{tikzcd}
\end{equation}
We note that we have used the same notation $\mathbb{L}$ for a line bundle on both $M_{\SL_r,\widetilde{\bm \alpha}}^{par,s}$ and also on $M_{\SL_r,{\bm \alpha}}^{par,s}$.
The isomorphisms in Lemma \ref{lem:infchange}, composed with the differential maps, give natural
maps
\begin{equation}\label{eqn:natural1}
R^1\widetilde{\pi}_*\mathcal{T}_{M^{par,s}_{\widetilde{\bm \alpha},\SL_r}/S}\,\lra\,
R^1{\pi}_*\bigl(\mathcal{T}_{M^{par,s}_{{\bm \alpha},\SL_r}/S}\bigr)
\ ;
\end{equation}
\begin{equation}\label{eqn:natural2}
\widetilde{\pi}_*\operatorname{Sym}^2\mathcal{T}_{M^{par,s}_{\SL_r,\widetilde{\bm \alpha}}/S}
\,\lra\,
    {\pi}_*\operatorname{Sym}^2\bigl(\mathcal{T}_{M^{par,s}_{\SL_r,{\bm
    \alpha}}/S}\bigr)\ .
\end{equation}
With the above notation we have the following proposition:

\begin{proposition}
	The maps in \eqref{eqn:natural1} and \eqref{eqn:natural2} are isomorphisms, and
	the diagram in \eqref{eqn:reductiondiagram} is commutative.
\end{proposition}

\begin{proof} Consider the differential 
	$
	DF \, :\, \mathcal{T}_{M^{par,s}_{\SL_r, \widetilde {\bm\alpha}/S}}\, \longrightarrow\, F^*\mathcal{T}_{M^{par,s}_{\SL_r,\bm \alpha}/S}
	$,
	and its second symmetric product
	$$
	{\rm Sym}^2(DF) \, :\, {\rm Sym}^2\mathcal{T}_{M^{par,s}_{\SL_r, \widetilde {\bm\alpha}}/S}\,
	\longrightarrow\,{\rm Sym}^2(F^*\mathcal{T}_{M^{par,s}_{\SL_r,\bm \alpha}/S})\,=\,
	F^*{\rm Sym}^2(\mathcal{T}_{M^{par,s}_{\SL_r,\bm \alpha}/S})\, .
	$$
	Let
	$
	\beta\, :=\, (DF)^* \, :\, F^*\mathcal{T}^{\vee}_{M^{par,s}_{\SL_r, {\bm
				\alpha}}/S}\, \to\, \mathcal{T}^{\vee}_{M^{par,s}_{\SL_r\widetilde {\bm\alpha}}/S}
	$
	be the dual of the above homomorphism $DF$. 
	Note that ${\rm Sym}^2(\mathcal{T}_{M^{par,s}_{\SL_r, \widetilde{\bm\alpha}}/S})$
	(respectively, ${\rm Sym}^2(F^*\mathcal{T}_{M^{par,s}_{\SL_r, {\bm\alpha}}/S})$ defines fiberwise quadratic functions 
	$\mathcal{T}^{\vee}_{M^{par,s}_{\SL_r,\widetilde{\bm \alpha}}/S}$
    (respectively, 
	$F^*\mathcal{T}^{\vee}_{M^{par,s}_{\SL_r, \bm \alpha}/S}$. Take any $z\, \in\,
	M^{par,s}_{\SL_r, \widetilde{\bm \alpha}}$.
	For any $w \in {\rm Sym}^2(\mathcal{T}_{M^{par,s}_{SL_r, \widetilde{\bm\alpha}}/S})_z$ and
	$\nu\, \in\, (F^*\mathcal{T}^{\vee}_{M^{par,s}_{\SL_r, \bm
    \alpha}/S})_z$, we have:
	$({\rm Sym}^2(DF))_z(w)(\nu)= w((DF)^*_z(\nu))$.
    From this we have the following commutative diagram of homomorphisms
    (recall \eqref{ud}):
	\begin{equation}\label{ele5}
	\begin{tikzcd}[column sep=huge]
	\widetilde{\pi}_*{\rm Sym}^2(\mathcal{T}_{M^{par,s}_{\SL_r, \widetilde{\bm \alpha}}/S}) \arrow[r," \widetilde{\pi}_*{\rm Sym}^2(DF)" ]\arrow[d, "\cong"]&
	\widetilde{\pi}_* F^*\mathcal{T}_{M^{par,s}_{\SL_r,\bm \alpha}/S}\arrow[d,"\cong"]\\
	R^1\pi_{s*} \mathcal{T}_{\mathcal{C}/S}(-D) \arrow[r,"\operatorname{Id}"] &R^1\pi_{s*}\mathcal{T}_{\mathcal{C}/S}(-D) 
	\end{tikzcd}
	\end{equation}
	in which $\widetilde{\pi}_*{\rm Sym}^2(DF)$ is an isomorphism, because all other homomorphisms in
	\eqref{ele5} are isomorphisms. This proves that the map in \eqref{eqn:natural2} is an isomorphism.
	The proof that the map in \eqref{eqn:natural1} is an isomorphism is very similar to the proof of it for \eqref{eqn:natural2}. Now it is evident that the diagram in \eqref{eqn:reductiondiagram} is commutative.
\end{proof}
Thus we have proved the following proposition: 
\begin{proposition}\label{prop:sum1}
	Consider two elements $\mathbb{L}_1$ and 
	$\mathbb{L}_2$ in $\operatorname{Pic}(M^{par,s}_{\SL_r, \bm \alpha})\otimes \mathbb{Q}$. If the maps 
	$\cup [\mathbb{L}_1]$ and $\cup[\mathbb{L}_2]$ agree on $\widetilde{\pi}_*\operatorname{Sym}^2 \mathcal{T}_{ M^{par,s}_{\SL_r,\widetilde{\bm{\alpha}}/S}}$, 
	then they also agree on ${\pi}_*\operatorname{Sym}^2 \mathcal{T}_{ 
		M^{par,rs}_{\SL_r,\bm{\alpha}}/S}$.
\end{proposition}
\subsection{Reduction to abelian varieties}

This step is essentially the same as in \cite[Prop.\ 5.2]{Hitchin:90} generalized to 
the parabolic set-up with the additional information about spectral data with one node. For completeness, we include the details by following the exposition 
in \cite{BBMP20}.

\subsubsection{Hitchin Map}

Let $\pi_{s}\,:\, \mathcal{C}\,\to\, S$ be a family of $n$-pointed curves, and let $D$ be the divisor of marked points. Consider the vector bundle 
$\mathcal{B}\,:=\,\bigoplus_{i=2}^{r}\pi_{s*}
K_{\mathcal{C}/S}^{i}((i-1)D)$, and let $\pi_{\mathcal{B}}\,:\,
\mathcal{B}\,\to\
\, S$ be the natural projection map.
Let $\pi_{\mathcal{H}}\,:\, \mathcal{H}^{par,ss}_{\widetilde{\bm
		\alpha},\SL_r}\,\to\, S$ be the relative strongly parabolic Higgs moduli
space parametrizing pairs $(\mathcal{P}, \theta)$, where $\mathcal{P}$ is a
parabolic bundle and $\theta$ is a strongly parabolic
endomorphism of $\mathcal{P}$ twisted by $K(D)$. We refer the reader to \cite{BiswasRamanujan} for notions of stability and semistability for strongly parabolic Higgs bundles. 
Recall the Hitchin morphism $\operatorname{Hit}\,: \,\mathcal{H}^{par,ss}_{\widetilde{\bm \alpha}, \SL_r}
\,\to\, \mathcal{B}$ from Section \ref{sec:canonical}.
We have the  following commutative diagram 
\begin{equation}\label{zl3}
\begin{tikzcd}[column sep=huge]
\mathcal{H}^{par,ss}_{\widetilde{\bm \alpha},\SL_r} \arrow[r,"\operatorname{Hit}"]\arrow[dr, "\pi_{\mathcal{H}}"'] & \mathcal{B} \ar[d, "\pi_{\mathcal{B}}"]\\
&S\ .
\end{tikzcd}
\end{equation}
Let $\mathcal{B}^{0}$ denote the collection of points in $\mathcal{B}$
such that the corresponding spectral
curve (as described in \cite[Sec.\ 3]{BNR}) is smooth.  The complement of
$\mathcal{B}^0$ in $\mathcal{B}$ is a divisor,
since we are in the case of $\SL_r$-Higgs bundles with full flags. 
This follows from the fact (\cite[Lemma 3.1]{gomezlogares} and \cite[Remark 3.5]{BNR}) that $K^rD^{r-1}$ is 
very ample and has sections without multiple zeros in either of the following cases: 
$g\geq 2$; $g=1$ and degree of $D \geq \frac{3}{r-1}$; $g=0$ and degree of $D\geq 2+\frac{3}{r-1}$.
But this is implied by the assumption that the orbifold genus $g(\Cscr)\geq
2$ (see Definition \ref{def:orbigenus} and also Appendix \ref{sec:appC}).
Then via abelianization, it is well-known that the fibers of 
$\operatorname{Hit}^{-1}(\vec{b})$, $\vec{b}\in \mathcal{B}^0$, are families of  abelian varieties $A_{\vec{b}}$ over $S$.  

Consider the divisor $\mathcal{D}:=\mathcal{B} \backslash \mathcal{B}^{\circ}\subsetneq \mathcal{B}$. 
As in \cite[Prop.\ 4.1]{alfayagomez}, for $x\in D$ let $\mathcal{D}_x$ to
be the set of characteristic polynomials whose spectral curves are  singular
over $x$, and let $\mathcal{D}_U$ to be the set of characteristic
polynomials whose spectral curves are smooth over each $x\in D$,  but singular over some $y\notin D$. 
Then $\mathcal{D}=\overline{\mathcal{D}}_U\cup \bigcup_{x\in D}\mathcal{D}_x$. 

Now (\cite[p.\ 28]{alfayagomez})
$\mathcal{D}_x=\bigoplus_{i=2}^{r-1}H^0(K^iD^{i-1})\oplus
H^0(K^rD^{r-1}(-x))$, and hence is irreducible.
By the assumption, $K^rD^{r-1}$ is very ample, which implies that $\dim \mathcal{D}_x < \dim \mathcal{D}$. 
Similarly, the remaining part of the proof of \cite[Prop.\
4.1]{alfayagomez} also goes through under this assumption. We obtain that
$\overline{\mathcal{D}}_U$ is the surjective image of an affine bundle over $C\backslash D$ whose fiber at $y$ is given by 
$$\oplus_{i=1}^{r-2}H^0(K^iD^{i-1})\oplus H^0(K^{r-1}D^{r-2}(-y))\oplus
H^0(K^rD^{r-1}(-2y))\ .$$ 
Hence, $\overline{\mathcal{D}}_U$ is also irreducible.

Thus $\overline{\mathcal{D}}_U$ is the unique irreducible component of
highest dimension in $\mathcal{D}$. Now by Bertini's theorem, a generic
point of $\mathcal{D}_U$ has an irreducible spectral curve with exactly one node over a point $y\notin D$. 

Now we let $\mathcal{B}^{\heartsuit}$ denote the subspace of $\mathcal{B}$ consisting of all points such that the spectral curve is
irreducible and has at most one node outside the divisor $D$. By the previous discussion, we get that the codimension of the complement of
$\mathcal{B}^{\heartsuit}$ in $\mathcal{B}$ is  at least two. The following lemma determines the fibers of the Hitchin map over
points of $\mathcal{B}^{\heartsuit}$. 

\begin{proposition}\label{prop:quasiabelian}
The fiber of the Hitchin map $\mathcal{H}_{\widetilde{\bm \alpha},\SL_r}^{par,ss}\rightarrow \mathcal{B}$
over any point $\vec{b} \in \mathcal{B}^{\heartsuit}$ is a quasi-abelian variety. 
\end{proposition}

\begin{proof}
Fix a $n$-pointed Riemann surface $(X, D)$. Let $\mathcal{B}$ be the base of the strongly parabolic Hitchin
map. For any $\vec{b} \in \mathcal{B}$, let $C_{\vec{b}} \subset K_X(D)$ be the corresponding spectral curve; let
	$p_{\vec{b}} : C_{\vec{b}} \to X$ be the natural projection. By assumptions $\vec{b}$ is such that $C_{\vec{b}}$ is a nodal curve with
	a single node $z$ which is not contained in $p_{\vec{b}}^{-1}(D)$.
    Moreover since the curve $C_{b}$ is integral, we get that the
    pushforward of a torsion free sheaf to $X$ is locally free. 
	
	 Consider the compactified Jacobian
     $\overline{J}^{\delta}(C_{\vec{b}})$ consisting of rank one-torsion
     free sheaves $L$ such that degree of $p_{\vec{b},*}L$ is zero. Since
     the node is not a marked point, we get a natural filtration of sheaves with quotients supported on the divisor $D$. 
	 \begin{equation}\label{eqn:inducedflag0}
	 {p}_{\vec{b}, *}(L\otimes \mathcal{O}_{{C}_{\vec{b}}}(-(r-1)R))\,\subset\, \cdots
	 \,\subset\,{p}_{\vec{b},*}(L\otimes \mathcal{O}_{{C}_{\vec{b}}}(-(r-i)R))\,
	 \subset\, \cdots\,\subset \,{p}_{\vec{b},*}L,
	 \end{equation} where $R$ is the ramification divisor. 
	As in \cite{BNR}, pushing forward a  section $\phi$  of
    $p_{\vec{b}}^*(K_XD)$ induces a map $\phi:  p_{\vec{b},* }L \rightarrow
    p_{\vec{b},* }L \otimes K_X(D)$. Now  since the node and the marked points are disjoint, the section $\phi$ gives the required  Higgs field as in the case of smooth spectral curves \cite{gomezlogares}.  This gives the spectral correspondence in the case of degree zero Higgs bundles.
	Consider the closed variety of $\overline{J}^\delta(C_{\vec{b}})$ defined as follows: $$\overline{\operatorname{Prym}}(C_{\vec{b}},C)=\{ M\in \overline{J}(C_{\vec{b}}) \ |\ p_{\vec{b},*}M=\mathcal{O}_X\}.$$ Clearly the variety $\overline{\operatorname{Prym}}(C_{\vec{b}},C)$ gives the Hitchin fiber at ${\vec{b}}\in \mathcal{B}^{\heartsuit}\backslash \mathcal{B}^0$
    (cf.\ \cite[Thm.\ 6.1]{GothenOliveira:13}). To complete the proof we need to show that $\overline{\operatorname{Prym}}(C_{\vec{b}},C)$ is semi-abelian. 
	
	 Let $n : Y\to  C_{\vec{b}}$ be the
	normalization and $f = p_{\vec{b}}\circ n$ the projection of $Y$ to $X$. The points of $Y$
over $z$ are $a$ and $b$, respectively. Let $P \subset J^\delta(Y)$ be the Prym for $f$. Let  $L \to Y \times P$ be a Poincar\'e line
bundle which is just the restriction of a Poincar\'e bundle on $Y\times J^{\delta}(Y)$. For any point $y$ 
	of $Y$, the line bundle in $P$ (resp.\ also on $J^{\delta}(Y)$)
    obtained by restricting $L$ to ${y} \times P$ (resp.\ also on ${y}\times J^{\delta}(Y)$) will be denoted by
 	$L_y$. Consider the line bundle $A := L_b^*\otimes L_a$ on $P$ (resp.\ $J^{\delta}(Y)$); it is independent of the 
	choice of the Poincar\'e bundle $L$. Now consider the projective bundle $\mathbb{P}(A\oplus \mathcal{O}_P)
\rightarrow P$ (also on $J^{\delta}(Y)$ ) and identify the
	two sections of it given by $A$ and $\mathcal{O}_A$. The resulting
    varieties $B_P \subset B_{J^{\delta}}$ are semi-abelian. 
    By  \cite[Thm.\ 4]{bhosle}, $B_{J^{\delta}}$ is identified with
    $\overline{J}^{\delta}(C_{\vec{b}})$. Moreover, by the choice of $\delta$,
    we get $B_{P}\subseteq \overline{\operatorname{Prym}}(C_{\vec{b}},C)$.
    The equality follows from the fact that the dimensions of both $B_{P}$ and
    $\overline{\operatorname{Prym}}(C_{\vec{b}},C)$ are the same. This completes the proof. 
\end{proof}
\subsubsection{Vector fields tangent to fibers of Hit}
 We get natural functions on $\mathcal{H}_{\widetilde{\bm \alpha},\SL_r}^{par,ss}$ obtained by pulling back sections of $\mathcal{B}^*$ to 
$\mathcal{H}_{\widetilde{\bm \alpha},\SL_r}^{par,ss}$ via the Hitchin map $\operatorname{Hit}$ in \eqref{zl3}. Since
$\mathcal{T}^{\vee}_{M_{\SL_r,\widetilde{\bm \alpha}}^{par,s}} \, \subset\, 
\mathcal{H}_{\widetilde{\bm \alpha},\SL_r}^{par,s}$, and the natural Liouville symplectic form on
$\mathcal{T}^{\vee}_{M_{\SL_r,\widetilde{\bm \alpha}}^{par,s}}$ extends to $\mathcal{H}_{\widetilde{\bm 
		\alpha},\SL_r}^{par,s}$, we get Hamiltonian vector fields on $\mathcal{H}_{\widetilde{\bm \alpha},\SL_r}^{par,s}$ tangent to the fibers of the 
parabolic Hitchin map. As the codimension of the complement of $\mathcal{T}^{\vee}_{M_{\SL_r,\widetilde{\bm \alpha}}^{par,s}}$ in 
$\mathcal{H}_{\widetilde{\bm \alpha},\SL_r}^{par,ss}$ is at least two,
we conclude that any class $\mathbb{L}$ in the rational Picard group of 
$M^{par,ss}_{\SL_r,\widetilde{\bm \alpha}}$ extends to entire $\mathcal{H}_{\widetilde{\bm \alpha},\SL_r}^{par,s}$. Now the cup product 
with the relative Atiyah class of $\mathbb{L}$ gives a natural map
\begin{equation}\label{eqn:crucialHitchin0}
\pi_{\mathcal{H}_*} \mathcal{T}_{\mathcal{H}_{\widetilde{\bm \alpha},\SL_r}^{par,s}/S}\,\lra\,
R^1{\pi_{\mathcal{H}}}_*\mathcal{O}_{\mathcal{H}_{\widetilde{\bm
    \alpha},\SL_r}^{par,s}}\ .
\end{equation}
Since the map $\pi_{\mathcal{B}}$ in \eqref{zl3} is affine, it follows that
$R^1{\pi_{\mathcal{H}}}_*{\mathcal{O}_{\mathcal{H}_{{\widetilde{\bm\alpha}},\SL_r}^{par,s}}}$
is isomorphic to the locally free sheaf ${\pi_{\mathcal{B}}}_*\bigl( R^1
\operatorname{Hit}_*\mathcal{O}_{\mathcal{H}_{\widetilde{\bm
\alpha},\SL_r}^{par,ss}}\bigr)$.  We also have the inclusion 
$\pi_{\operatorname{Hit}_*}\mathcal{T}_{\mathcal{H}_{\widetilde{\bm
\alpha},\SL_r}^{par,s}/\mathcal{B}} \hookrightarrow
\pi_{\operatorname{Hit}_*}\mathcal{T}_{\mathcal{H}_{\widetilde{\bm
\alpha},\SL_r}^{par,s}/S}$.
Now consider the map obtained by restricting \eqref{eqn:crucialHitchin0},  which on pushing forward gives 
\begin{equation}
\label{eqn:crucialHitchin1}
\begin{tikzcd}
\pi_{\mathcal{B}_*}\bigl(\pi_{\operatorname{Hit}_*}
    \mathcal{T}_{\mathcal{H}_{\widetilde{\bm
    \alpha},\SL_r}^{par,s}/\mathcal{B}}\bigr)\arrow[rr,"f_{\mathbb{L}}"', bend right=10]\arrow[r, hook]&\pi_{\mathcal{H}_*}\mathcal{T}_{\mathcal{H}_{\widetilde{\bm \alpha},\SL_r}^{par,s}/\mathcal{B}} \arrow[r]&{\pi_{\mathcal{B}}}_*
    \bigl( R^1 \operatorname{Hit}_*\mathcal{O}_{\mathcal{H}_{\widetilde{\bm
    \alpha},\SL_r}^{par,ss}}\bigr)
    \ .
\end{tikzcd}
\end{equation}
We have the following proposition:

\begin{proposition}\label{eqn:triviality}
	The coherent sheaves $\pi_{\operatorname{Hit}_*} \mathcal{T}_{\mathcal{H}_{\widetilde{\bm \alpha},\SL_r}^{par,s}/\mathcal{B}}$
and $R^1 \operatorname{Hit}_*\mathcal{O}_{\mathcal{H}_{\widetilde{\bm \alpha},\SL_r}^{par,ss}}$ are both trivial and
isomorphic of same rank, where the fibers are just the vector spaces $H^0(A_{\vec{b}}, \,\mathcal{T}_{A_{\vec{b}}})$ and
$H^1(A_{\vec{b}},\, \mathcal{O}_{A_{\vec{b}}})$,  respectively, for any for $\vec{b}\in \mathcal{B}^0$; the isomorphism is given
by cup product by a K\"ahler class on $A_{\vec{b}}$.
\end{proposition}

\begin{proof}

Cupping with the first Chern class of the pull back of the ample line bundle $\mathcal{L}_{\widetilde{\bm \alpha}}$
from $M^{par,ss}_{\SL_r,\widetilde{\bm \alpha}}$ induces a map between coherent sheaves $\pi_{\operatorname{Hit}_*} 
\mathcal{T}_{\mathcal{H}_{\widetilde{\bm \alpha},\SL_r}^{par,s}/\mathcal{B}}$ and $R^1 
\operatorname{Hit}_*\mathcal{O}_{\mathcal{H}_{\widetilde{\bm \alpha},\SL_r}^{par,ss}}$. Over $\mathcal{B}^0$, the 
fibers of the coherent sheaf $R^1\operatorname{Hit}_*\mathcal{O}_{\mathcal{H}_{\widetilde{\bm 
\alpha},\SL_r}^{par,ss}}$ have constant dimension which equals $\dim A_{\vec{b}}$. Similarly over $\mathcal{B}^0$, 
because the fibers of the map $\pi_{\operatorname{Hit}}$ are abelian varieties and the sheaf $\pi_{\operatorname{Hit}_*} 
\mathcal{T}_{\mathcal{H}_{\widetilde{\bm \alpha},\SL_r}^{par,s}/\mathcal{B}}$ is locally free and trivial. Moreover, 
there is an isomorphism between $\pi_{\operatorname{Hit}_*} \mathcal{T}_{\mathcal{H}_{\widetilde{\bm 
\alpha},\SL_r}^{par,s}/\mathcal{B}}$ and $R^1\operatorname{Hit}_*\mathcal{O}_{\mathcal{H}_{\widetilde{\bm 
\alpha},\SL_r}^{par,ss}}$ induced by the natural isomorphism between $H^0(A_{\vec{b}}, \,\mathcal{T}_{A_{\vec{b}}})$ 
and $H^1(A_{\vec{b}},\, \mathcal{O}_{A_{\vec{b}}})$ given by a K\"ahler class.
	
Now for $\vec{b} \in \mathcal{B}^{\heartsuit}\backslash \mathcal{B}^0$, by Proposition \ref{prop:quasiabelian}, we know that the fibers 
are quasi-abelian varieties $\overline{A}_{\vec{b}}$ and in particular $\dim 
H^1(\overline{A}_{\vec{b}},\mathcal{O}_{\overline{A}_{\vec{b}}})=\dim \overline{A}_{\vec{b}}$. Since the codimension of the complement of 
$\mathcal{B}^{\heartsuit}$ in $\mathcal{B}$ is at least two and the Hitchin map is flat \cite[Corollary 11]{BaragliaKamgarpourVarma:19}, \cite[Theorem 1.17]{BaragKam}, it follows that 
$R^1\operatorname{Hit}_*\mathcal{O}_{\mathcal{H}_{\widetilde{\bm \alpha},\SL_r}^{par,ss}}$ is locally free on $\mathcal{B}$. As in the 
case of Abelian varieties, the cup product by a K\"ahler form induces an isomorphism of 
$\operatorname{Ext}^0(\overline{A}_{\vec{b}},\mathcal{O}_{\overline{A}_{\vec{b}}})$ with 
$H^1(\overline{A}_{\vec{b}},\mathcal{O}_{\overline{A}_{\vec{b}}})$. This shows that the coherent sheaf $\pi_{\operatorname{Hit}_*} 
\mathcal{T}_{\mathcal{H}_{\widetilde{\bm \alpha},\SL_r}^{par,s}/\mathcal{B}}$ is trivial over $\mathcal{B}^{\heartsuit}$ with fibers given 
by functions on $\mathcal{B}$. Moreover cupping with the first Chern class of $\mathcal{L}_{\widetilde{\bm \alpha}}$ induces an isomorphism of 
$\pi_{\operatorname{Hit}_*} \mathcal{T}_{\mathcal{H}_{\widetilde{\bm \alpha},\SL_r}^{par,s}/\mathcal{B}}$ with $R^1 
\operatorname{Hit}_*\mathcal{O}_{\mathcal{H}_{\widetilde{\bm \alpha},\SL_r}^{par,ss}}$. Thus the proposition follows from Hartogs' theorem and 
the fact that codimension of the complement of $\mathcal{B}^{\heartsuit}$ is at least two.
\end{proof}

The following result is a direct consequence of Proposition  \ref{eqn:triviality}.
\begin{corollary}\label{cor:triv}
	Let $\mathbb{L}_1$ and $\mathbb{L}_2$ be  two rational line bundles on $M_{\widetilde{\bm \alpha}, \SL_r}^{par,ss}$, and let
$\vec{b}$ be a generic point of
	the Hitchin base. Then $f_{\mathbb{L}_1}\,=\,f_{\mathbb{L}_2}$
(see \eqref{eqn:crucialHitchin1}) if and only if the two homomorphisms
	$H^0(A_{\vec{b}}, \,\mathcal{T}_{A_{\vec{b}}})\,\to\, H^1(A_{\vec{b}},\, \mathcal{O}_{A_{\vec{b}}})$ induced by
	cupping with the first Chern class of the restrictions of $\mathbb{L}_1$ and $\mathbb{L}_2$ are the same.
\end{corollary}

Now the composition of $f_{\mathbb{L}}$  with the natural Hamiltonian vector fields
produces a homomorphism
\begin{equation}\label{eqn:crucialHitchin}
h_{\mathbb{L}}\,:\,{\pi_{\mathcal{B}}}_*\mathcal{O}_{\mathcal{B}}\otimes \mathcal{B}^{*}
\,\longrightarrow\,
R^1{\pi_{\mathcal{H}}}_*\mathcal{O}_{\mathcal{H}_{\widetilde{\bm
    \alpha},\SL_r}^{par,s}}\ .
\end{equation}
Observe that this map $h_{\mathbb{L}}$ is equivariant with respect to the natural
$\mathbb{C}^*$ action on ${\pi_{B}}_*\mathcal{O}_{\mathcal{B}}\otimes \mathcal{B}^*$ and the
natural action of $\mathbb{C}^*$ on $R^1{\pi_{\mathcal{H}}}_*\mathcal{O}_{\mathcal{H}^{par,s}_{\widetilde{\bm \alpha},\SL_r}}$
is of weight $-1$. 
Since $H^0(A_{\vec{b}},\, \mathcal{T}_{A_{\vec{b}}})$ is given by vector fields coming from $\mathcal{B}^*$, we have
the following lemma:

\begin{lemma}\label{lem:obvious}
The two homomorphisms ${h}_{\mathbb{L}_1}$ and ${h}_{\mathbb{L}_2}$ $($see
    \eqref{eqn:crucialHitchin}$)$
coincide if and only if $f_{\mathbb{L}_1}\,=\,f_{\mathbb{L}_2}$ $($see
    \eqref{eqn:crucialHitchin1}$)$. 
\end{lemma}
\noindent
Finally, we would like to relate the map
$\cup[\mathbb{L}]: \pi_*\operatorname{Sym}^2\bigl(\mathcal{T}_{M^{par,s}_{\SL_r,
		\widetilde{\bm \alpha}}/S}\bigr)\to
R^1{\pi}_*\bigl(\mathcal{T}_{M^{par,s}_{\SL_r,\widetilde{\bm
\alpha}}/S}\bigr)$
with the map $h_{\mathbb{L}}$ in \eqref{eqn:crucialHitchin}.
Observe that
$\pi_*\operatorname{Sym}^2\left(\mathcal{T}_{M^{par,s}_{\SL_r,\widetilde{\bm
			\alpha}}/S}\right)$ injects into
$\pi_{\mathcal{H}_*}\mathcal{O}_{\mathcal{H}_{\widetilde{\bm
			\alpha},\SL_r}^{par,s}}$ as the degree two part. Since the Hitchin map
is proper (Lemma \ref{lem:proper}), and its fibers are connected, functions on the Higgs moduli
spaces are all pull-backs of functions on the Hitchin base. As described
earlier, these functions give Hamiltonian vector fields and hence we have a map 
\begin{equation}
\begin{tikzcd}
\pi_*\operatorname{Sym}^2\mathcal{T}_{M^{par,s}_{\SL_r,\widetilde{\bm
    \alpha}}/S}\arrow[r,hook]&\pi_{\mathcal{H}_*}\mathcal{O}_{\mathcal{H}_{\widetilde{\bm
    \alpha},\SL_r}^{par,s}}\arrow[r]&
    {\pi_{\mathcal{H}}}_*\mathcal{T}_{\mathcal{H}_{\widetilde{\bm
    \alpha},\SL_r}^{par,s}}\ .
\end{tikzcd}
\end{equation}
Cupping with any section $\gamma$ of
$R^1{\pi_{\mathcal{H}}}_*{\Omega}_{\mathcal{H}_{\widetilde{\bm \alpha},\SL_r}^{par,s}}$
produces a map  
\begin{equation}\label{eqn:obvious33}
\begin{tikzcd}
\pi_*\operatorname{Sym}^2\mathcal{T}_{M^{par,s}_{\widetilde{\bm
    \alpha},\SL_r}/S}\arrow[r,hook]&\pi_{\mathcal{H}_*}\mathcal{O}_{\mathcal{H}_{\widetilde{\bm
    \alpha},\SL_r}^{par,s}}\arrow[r]&
    {\pi_{\mathcal{H}}}_*\mathcal{T}_{\mathcal{H}_{\widetilde{\bm
    \alpha},\SL_r}^{par,s}/S}\arrow[r,"\cup
    \gamma"]&R^1{\pi_{\mathcal{H}}}_*\mathcal{O}_{\mathcal{H}_{\widetilde{\bm
    \alpha},\SL_r}^{par,s}}\ .
\end{tikzcd}
\end{equation}
Consider the inclusion of $R^1\pi_*\mathcal{T}_{M^{par,s}_{\widetilde{\bm \alpha},\SL_r}/S}$ into $R^1{\pi_{\mathcal{H}}}_*\mathcal{O}_{\mathcal{H}_{\widetilde{\bm \alpha},\SL_r}^{par,s}}$.
On the other hand, we have the following exact sequence 
$$0\,\longrightarrow\, \mathcal{T}_{M^{par,s}_{\widetilde{\bm \alpha},\SL_r}/S}\,\longrightarrow
\, \mathcal{O}_{\mathcal{H}_{\widetilde{\bm \alpha},\SL_r}^{par,s}}/
	\mathcal{I}_{M^{par,s}_{\SL_r,\widetilde{\alpha}}}^2\,\longrightarrow\, \mathcal{O}_{{M}^{par,s}_{\SL_r,\widetilde{\alpha}}}
\,\longrightarrow\, 0\ ,$$
where $\mathcal{I}_{{M}^{par,s}_{\SL_r,\widetilde{\bm \alpha}}}$ is the
ideal sheaf of ${M}^{par,s}_{\SL_r,\widetilde{\bm \alpha}}$ in the moduli
of parabolic Higgs bundles. Since there are no global tangent vector field
on ${M}^{par,s}_{\SL_r,\widetilde{\bm \alpha}}$, it follows from the long exact sequence of cohomology that 
$R^1\pi_*\bigl(\mathcal{T}_{M^{par,s}_{\SL_r,\widetilde{\bm \alpha}}/S}\bigr)
\,\cong\, R^1\pi_*\bigl(\mathcal{O}_{\mathcal{H}_{\widetilde{\bm
\alpha},\SL_r}^{par,s}}/\mathcal{I}_{M^{par,s}_{\SL_r,\widetilde{\bm\alpha}}}^2\bigr)$.
Now the restriction induces  another map
\begin{equation}\label{eqn:obvious34}
\begin{tikzcd}
R^1{\pi_{\mathcal{H}}}_*\mathcal{O}_{\mathcal{H}_{\widetilde{\bm \alpha},\SL_r}^{par,s}}\arrow[r]&
    R^1\pi_*\bigl( \mathcal{O}_{\mathcal{H}_{\widetilde{\bm \alpha},\SL_r}^{par,s}/\mathcal{I}_{{M}^{par,s}_{\SL_r,\widetilde{\bm
    \alpha}}}^2}\bigr)
\end{tikzcd}
\end{equation}
which restricts to the identity map on
$R^1\pi_*\bigl(\mathcal{T}_{M^{par,s}_{\SL_r,\widetilde{\bm
\alpha}}/S}\bigr)$. Hence, combining eqns. \eqref{eqn:obvious33} and \eqref{eqn:obvious34}, we have the following diagram
\begin{equation}
\adjustbox{scale=0.8}{
\begin{tikzcd}[row sep=huge, column sep=huge]
&\pi_*\operatorname{Sym}^2\mathcal{T}_{M^{par,s}_{\SL_r,\widetilde{\bm \alpha}}/S}\arrow[ld, hook]\arrow[rr,"\cup \gamma"]&&R^1\pi_*\left(\mathcal{T}_{M^{par,s}_{\SL_r,\widetilde{\bm \alpha}}/S}\right)\arrow[d,bend left=30]\\
\pi_{\mathcal{H}_*}\mathcal{O}_{\mathcal{H}_{\widetilde{\bm \alpha},\SL_r}^{par,s}}\arrow[r]&{\pi_{\mathcal{H}}}_*\mathcal{T}_{\mathcal{H}_{\widetilde{\bm \alpha},\SL_r}^{par,s}/S}\ar[r,hook]& {\pi_{\mathcal{H}}}_*\mathcal{T}_{\mathcal{H}_{\widetilde{\bm \alpha},\SL_r}^{par,s}/S}\arrow[r,"\cup \gamma"]&R^1{\pi_{\mathcal{H}}}_*\mathcal{O}_{\mathcal{H}_{\widetilde{\bm \alpha},\SL_r}^{par,s}}\arrow[u]
\end{tikzcd}}
\end{equation}
The same proof as in Hitchin \cite[p.\ 379]{Hitchin:90} (see also \cite[Prop.\ C.2.4]{BBMP20}) shows
that the above diagram commutes up to a scalar, and, by construction, the horizontal
map at the bottom is the map $h_{\gamma}$ (cf.\ \eqref{eqn:crucialHitchin}). Thus we proved the following.
\begin{proposition}\label{prop:sum2}
	Consider two elements $\mathbb{L}_1$ and 
		$\mathbb{L}_2$ in $\operatorname{Pic}(M^{par,s}_{\SL_r,\widetilde{\bm \alpha}})\otimes \mathbb{Q}$ and let $A_{\vec{b}}$ be as in Corollary \ref{cor:triv}. If the maps 
		between $H^0(A_{\vec{b}},\mathcal{T}_{A_{\vec{b}}})\rightarrow H^1(A_{\vec{b}}, \mathcal{O}_{A_{\vec{b}}})$ induced by cupping with the first Chern classes of restrictions of $\mathbb{L}_1$ and $\mathbb{L}_2$ are the same, then they also agree on $ \widetilde{\pi}_*\operatorname{Sym}^2 \mathcal{T}_{ M^{par,s}_{\SL_r,\bm{\alpha}}/S}$.

\end{proposition} 

\subsection{Abelianization and determinant of cohomology}

It is enough to consider the case of parabolic Higgs bundles of degree zero and rank $r$ 
with full flag and arbitrary parabolic weights $\widetilde{\alpha}$.
Consider a generic point $\vec{b}$ of the Hitchin base for the parabolic Higgs moduli 
space $\mathcal{H}^{par,ss}_{\widetilde{\bm \alpha}}$ with full flag and weights 
${\widetilde{\bm \alpha}}$, and let
$\widetilde{p}\,:\,\widetilde{C}_{\vec{b}}\,\to
C$ be the spectral cover of $C$ determined by the chosen point $\vec{b}$ of the Hitchin base. 
The map $\widetilde{p}$ is of degree $r$ and is fully ramified at the points $\bm 
{p}\,=\,({p_1,\,\cdots,\, p_n})$. Let $\bm {q}\,=\,(q_1,\,\cdots,\, q_n)$ be the inverse image
$\widetilde{p}^{-1}(\bm{p})$ of the points $\bm{p}$.
It is known \cite{ Faltings:93,LogaresMartens:10} that the generic fiber $A_{\vec{b}}$ of 
the Hitchin map at $\vec{b}$ is exactly the Jacobian $J(\widetilde{C}_{\vec{b}})$.
Let $L$ be a line bundle on $\widetilde{C}$ giving a point of $A_{\vec{b}}$ and consider the push-forward $\widetilde{p}_*L$ on $C$. Consider the divisor $D\,=\,p_1+\dots+p_n$. There is a natural inclusion of sheaves 
\begin{equation}\label{eqn:inducedflag}
\widetilde{p}_*(L\otimes \mathcal{O}_{\widetilde{C}_{\vec{b}}}(-(r-1)R))\,\subset\, \cdots
\,\subset\,\widetilde{p}_*(L\otimes \mathcal{O}_{\widetilde{C}_{\vec{b}}}(-(r-i)R))\,
\subset\, \cdots\,\subset \,\widetilde{p}_*L
\end{equation}
with quotients supported on $D$ giving a quasiparabolic structure on $\pi_*L$ at the
points ${\bm p}$. Here $R$ is the ramification divisor $(\widetilde{p}^*D)_{\rm red}$. 
Hence, this gives a rational map from $A_{\vec{b}}$ to the $M_{\widetilde{\bm
		\alpha}}^{par}$. The fiber of the pull-back of the the parabolic determinant of
cohomology $\operatorname{ParDet}(\bm \alpha)$ to the abelian variety at the point $L \,\in\, A_{\vec{b}}$ is a rational linear combinations of elements of the form 
\begin{enumerate}
	\item  $H^0(\widetilde{C}_{\vec{b}},L)^{\vee}\otimes H^1(\widetilde{C}_{\vec{b}},L) \otimes \det (\widetilde{p}_*L)_{p_i}^{\frac{\chi(\pi_*L)}{r}}$,
	
	\item $\det \operatorname{Gr}^j\mathscr{F}(\widetilde{p}_*L)_{p_i}\otimes \det^{-1} 
	(\widetilde{p}_*L)_{p_i}$ for all $1\,\leq \,i\,\leq\, n$.
\end{enumerate}
However, observe that the second expression for each $p_i$ is independent
of $L$ and is equal to  the line $\mathcal{O}_{\widetilde{C}_{\vec{b}}}(-q_i)_{|q_i}
\,=\,{K_{{\widetilde{C}_{\vec{b}}}}}_{|q_i}$. Indeed, this follows from the facts that
\begin{itemize}
	\item $\det \operatorname{Gr}^j\mathscr{F}(\widetilde{p}_*L)_{p_i}\,=\,
	L_{q_i}\otimes \mathcal{O}_{\widetilde{C}_{\vec{b}}}(-jq_i)_{|q_i}\otimes
	\mathcal{O}_{\widetilde{C}_{\vec{b}}}(-(j-1)q_i)_{|q_i}^{-1} $
	\item $\det (\widetilde{p}_*L)_{p_i}\,=\,L_{q_i}$ 
\end{itemize}
together with the natural flag structure given by \eqref{eqn:inducedflag}.

The calculations above show that the pull-back of 
$\operatorname{ParDet}(\bm \alpha)$ to the abelian variety only depends on the factors of 
type (1). The map $\cup [\mathbb{L}]:
H^0(A_{\vec{b}},\mathcal{T}_{A_{\vec{b}}})\to
H^1(A_{\vec{b}},\mathcal{O}_{A_{\vec{b}}})$  thus depends only on the level for all 
$\mathbb{L}\,\in\,\operatorname{Pic}(M_{\widetilde{\bm \alpha}}^{par})\otimes {\mathbb{Q}}$. 
Thus we have proved the following proposition:
\begin{proposition}\label{prop:sum3}
	Let $\mathbb{L} \in \operatorname{Pic}(M_{\SL_r, \widetilde{\bm \alpha}}^{par,s})\otimes \mathbb{Q}$, then the natural map induced by the first Chern class of the restriction of $\mathbb{L}$ between $H^0(A_{\vec{b}},\mathcal{T}_{A_{\vec{b}}})\to
	H^1(A_{\vec{b}},\mathcal{O}_{A_{\vec{b}}})$ depends only on the level of $\mathbb{L}$. 
	\end{proposition}
\subsection{Proof of Theorem \ref{thm:indepdence}} For the convenience of the reader let us recall the statement of Theorem \ref{thm:indepdence} from the beginning of the section.

	\begin{theorem}
		Let $\mathbb{L}$ be an element of
		$\operatorname{Pic}(M^{par,rs}_{G, \bm{\beta}})\otimes \mathbb{Q}$ of
		level $a$. Then as linear maps
		$\pi_{e*}\operatorname{Sym}^2\mathcal{T}_{M^{par,rs}_{G, \bm{\beta}}/S}\,
		\to\, R^1\pi_{e*}\mathcal{T}_{M^{par,rs}_{G,\bm{\beta}}/S}$, we have:  $\cup\, [\mathbb{L}]\,=\,\cup\, a [\operatorname{Det}]$, where $\operatorname{Det}$ is the determinant of cohomology (nonparabolic) line bundle. 
	\end{theorem}

\begin{proof}
	The proof follows from Propositions \ref{prop:sum0}, \ref{prop:sum1}, \ref{prop:sum2}, and \ref{prop:sum3} and fact that any line bundle on $M_{G,\bm \beta}$ is obtained as pulled back of a rational multiple of a line bundle on the moduli space of parabolic bundles for $G=\SL_r$.
\end{proof}
\section{The parabolic Hitchin connection}\label{sec:proof}

In this section we will use Theorem \ref{thm:hitchimain}
and the results from \cite{BMW0} on Ginzburg dglas and 
the class of the parabolic determinant of cohomology $\mathcal{L}$ 
to construct a flat projective connection on the vector
bundle $\pi_{e*}\mathcal{L}^k$,
where $\pi_e: {M}^{par,rs}_{G}\rightarrow S$ is the projection.

\subsection{Definition of the symbol}

We first seek a candidate for the 
symbol map $$\rho_{par}: \mathcal{T}_S 
\rightarrow \pi_{e*}\operatorname{Sym}^2\mathcal{T}_{{M}^{par,rs}_{G}/S}.$$
As in the nonparabolic case, 
set $\widetilde{\rho}:=\rho_{sym}\circ KS_{\Ccal/S}$. 
Let $k\geq 1$ be a positive integer, and let 
$\mathcal{L}_{\phi}$ be a line bundle on ${M}^{par,rs}_{G}$
constructed via its identification with $\Gamma$-$G$-bundles of 
fixed local type, a representation $\phi: G\rightarrow \SL_r$, 
and the restriction of determinant of cohomology from $\widehat{M}_{\SL_r}$.
We first recall the  main result  \cite[Cor.\ 4.13 and Prop.\ 4.12]{BMW0} 
that relates the class $\beta(\mathcal{P},\bm{\lambda})$ with the Atiyah
class of $[\mathcal{L}_{\phi}]$ of the line bundle $\mathcal{L}_{\phi}$.

\begin{theorem}
	Let $m_{\phi}$ be the Dynkin index of the map $\phi: G \rightarrow \SL_r$. Then 
	\begin{equation}
	\beta(\mathcal{P},\bm{\lambda})=\frac{1}{m_{\phi}}[\mathcal{L}_{\phi}]\ .
	\end{equation}
\end{theorem}
Now we further expand $\mu_{\mathcal{L}_{\phi}^{\otimes k}} \circ \frac{1}{m_{\phi}k}\widetilde{\rho}$ and  get the following: 
\begin{align*}
\mu_{\mathcal{L}_{\phi}^{\otimes k}} \circ \frac{1}{m_{\phi}k}\widetilde{\rho}&
=\frac{1}{m_{\phi}}.\bigl( (\cup\,
    (k[\mathcal{L}_{\phi}]-\frac{1}{2}[\Omega_{{M}^{par,rs}_{G}/S}]))\circ
    \frac{1}{k}\rho_{sym}\circ KS_{\Ccal/S}\bigr)\\ 
&=\frac{1}{m_{\phi}}\bigl(\cup\, [\mathcal{L}_{\phi}]\circ \rho_{sym}\circ
KS_{\Ccal/S}-\cup\, \frac{1}{2k}[\Omega_{{M}^{par,rs}_{G}/S}]\circ
    \rho_{sym} \circ KS_{\Ccal/S}\bigr)\\
&=\cup\, \beta(\mathcal{L}_{\phi})
\circ \rho_{sym} \circ KS_{\Ccal/S}-\cup\,  \frac{1}{2m_{\phi}k}[\Omega_{{M}^{par,rs}_{G}/S}]\circ \rho_{sym} \circ KS_{\Ccal/S}\\
&=-\Phi\circ KS_{\Ccal/S}-\cup\,  \frac{1}{2m_{\phi}k}[\Omega_{{M}^{par,rs}_{G}/S}]
\circ \rho_{sym} \circ KS_{\Ccal/S}\\
&=-KS_{{M}^{par,rs}_{G/S}}-\cup\,
    \frac{1}{2m_{\phi}k}[\Omega_{{M}^{par,rs}_{G}/S}]\circ \rho_{sym} \circ
    KS_{\Ccal/S}\ .
\end{align*}
In the above,  we have used the fundamental equalities 
$$\beta(\mathcal{L}_{\phi})=\frac{1}{m_{\phi}}[\mathcal{L}_{\phi}]\ \text{ and }\
\beta(\mathcal{L}_{\phi})\circ\rho_{sym}+\Phi=0.$$
Thus we get the following equation:
\begin{equation}\label{eqn:lookingforhitchinsymbol}
KS_{{M}^{par}_{G}/S}+ \mu_{\mathcal{L}_\phi^{\otimes k}} \circ \frac{1}{m_{\phi}k}\widetilde{\rho}+
\cup\, \frac{1}{2m_{\phi} k}[\Omega_{{M}^{par,rs}_{G}/S}]\circ \rho_{sym}
    \circ KS_{\Ccal/S}=0\ .
\end{equation}
We now have  a key result.

\begin{proposition} \label{prop:iso}
The map  $\mu_{\mathcal{L}_{\phi}^{\otimes k}}: \pi_{e*}\operatorname{Sym}^2\mathcal{T}_{{M}^{par,rs}_{G}/S}\rightarrow R^1\pi_{e*}\mathcal{T}_{{M}^{par,rs}_{G}/S}$ is an isomorphism. 
\end{proposition}
\begin{proof}
   Let 
    $Y_{G}^{par,rs}:=\phi^{-1}(\widehat{M}_{G}^{rs})\subset M_{G}^{par,rs}$, 
    where $\phi: M_{G}^{par,ss}\rightarrow \widehat{M}_{G}^{ss}$ is the natural forgetful map. 
    By Lemma \ref{lem:codim2}, the codimension of the complement of $Y_{G}^{par,rs}$ 
    in $M_{G}^{par,rs}$ is at least three, 
    so it enough to show that $\mu_{\mathcal{L}_{\phi}}$ is an isomorphism over 
    $Y_{G}^{par,rs}$. Now by Theorem \ref{thm:indepdence}, it follows that it suffices to show that $\cup[\mathcal{L}_{\phi}]$ is an isomorphism.
Observe that in the nonparabolic case, the canonical class is a multiple of
    the ample generator of the Picard group of $M_{G}$. Hence, for 
the nonparabolic case ${\mu}_{\mathcal{L}_{\phi}}$ is a nonzero multiple of $\cup [\mathcal{L}_{\phi}]$.
By construction, the map $\cup{\mathcal{L}_{\phi}^{\otimes k}}: \pi_{e*}\operatorname{Sym}^2\mathcal{T}_{{Y}^{par,rs}_{G}/S}\rightarrow R^1\pi_{e*}\mathcal{T}_{{Y}^{par,rs}_{G}/S}$ is first obtained by restricting the map $\cup{\mathcal{L}_{\phi}^{\otimes k}}:\pi_{e*}\operatorname{Sym}^2\mathcal{T}_{\widehat{M}^{rs}_{G}/S}\rightarrow R^1\pi_{e*}\mathcal{T}_{\widehat{M}^{rs}_{G}/S}$ to $\mathcal{T}_{{Y}^{par,rs}_{G}/S}$ and then taking invariants. 
Consequently, we will be done if we can show that the following map is an isomorphism:
$
\cup{\mathcal{L}_{\phi}^{\otimes k}}:\pi_{e*}\operatorname{Sym}^2\mathcal{T}_{\widehat{M}^{rs}_{G}/S}\rightarrow R^1\pi_{e*}\mathcal{T}_{\widehat{M}^{rs}_{G}/S}
$.
This is proved in \cite{Hitchin:90} and also in \cite{BBMP20}
in the algebro-geometric set-up for $G=\SL_r$, where $\mathcal{L}$ is the determinant of cohomology line bundle. For an
arbitrary $G$, we can choose a faithful irreducible representation $\phi: G \rightarrow \SL_r$ and get a map
$f: \widehat{M}_{G}\rightarrow \widehat{M}_{\SL_r}$ which restricts to a
    map $f: \widehat{M}^{rs}_{G}\rightarrow \widehat{M}^s_{\SL_r}$. Since 
    any short exact sequence of $G$-modules splits,
this induces a splitting of the tangent bundle of the moduli spaces:
$f^*\mathcal{T}_{\widehat{M}_{\SL_r}^{s}/S}=\mathcal{T}_{\widehat{M}^{rs}_{G}/S}\oplus
    W$,
 along with the diagram 
\begin{equation}
\begin{tikzcd}
\pi_*\operatorname{Sym}^2 \mathcal{T}_{ \widehat{M}^{s}_{\SL_r}/S}\arrow[r, "\cup \mathbb{L}"]\arrow[from=d,"\operatorname{Sym}^2Df"]& R^1\pi_*\mathcal{T}_{ \widehat{M}^{s}_{\SL_r}/S}\arrow[d]\\
{\pi_{G}}_*\operatorname{Sym}^2 \mathcal{T}_{ \widehat{M}^{rs}_{G}/S}\arrow[r,"\cup \mathbb{L}"]&R^1{\pi_G}_*\mathcal{T}_{ \widehat{M}^{rs}_{G}/S},
\end{tikzcd}
\end{equation}
where $\pi\,:\, \widehat{M}^{s}_{\SL_r}\,\to\, S$ and
$\pi_G\,:\, \widehat{M}^{rs}_{G}\,\to\, S$ are the natural
projections. Thus, we are again reduced to the case of $G=\SL_r.$
\end{proof}

Since the map $\mu_{\mathcal{L}_{\phi}^{\otimes k}}$ is an isomorphism, from \eqref{eqn:lookingforhitchinsymbol} we get that
\begin{equation}\label{eqn:almosthitchinsymbol}
KS_{{M}^{par,rs}_{G}/S}+ \mu_{\mathcal{L}_{\phi}^{\otimes k}} \circ
\Bigl(\frac{1}{m_{\phi}k}{\rho_{sym}}+\mu_{\mathcal{L}_{\phi}^{\otimes
	k}}^{-1}\circ(\cup\,
    \frac{1}{2m_{\phi}k}[\Omega_{{M}^{par,rs}_{G}/S}])\circ \rho_{sym}
    \Bigr)\circ KS_{\Ccal/S}=0.
\end{equation}
Motivated by \eqref{eqn:almosthitchinsymbol}, we \emph{define
} the parabolic Hitchin symbol $\rho_{par}$ to be: 
\begin{equation}\label{eqn:hitchparsymbol}
\rho_{par}:=\Bigl(\frac{1}{m_{\phi}k}+\mu_{\mathcal{L}_{\phi}^{\otimes
	k}}^{-1}\circ(\cup\, \frac{1}{2m_{\phi}k}[\Omega_{{M}^{par,rs}_{G}/S}])
    \Bigr)\circ {\rho_{sym}}\circ KS_{\Ccal/S}.
\end{equation}

\begin{remark}By Theorem \ref{thm:indepdence}, we see that ${\mu}_{\mathcal{L}_{\phi}}$ is a nonzero multiple of $\cup[\mathcal{L}_{\phi}]$ and hence $\mu_{\mathcal{L}_{\phi}^{\otimes
			k}}^{-1}\circ(\cup\, \frac{1}{2m_{\phi}k}[\Omega_{{M}^{par,rs}_{G}/S}]$ is a nonzero multiple of identity. This is
essentially akin to the nonparabolic situation.
In the case of the moduli space of rank $r$ vector bundles with trivial
determinant, it turns out that the class of the canonical bundle is
$[\Omega_{M^s_{\SL_r}/S}]=-2r[\mathcal{L}]$, where $\mathcal{L}$ is the
ample generator of the Picard group. Hence,  $\mu_{\mathcal{L}^{\otimes
		k}}^{-1}=-\frac{2r}{r+k}(\cup\, [\Omega_{M^s_{\SL_r}/S}])^{-1}$, and $\rho_{par}$ in this case
is just $\frac{1}{r+k}\rho_{sym}\circ KS_{\Ccal/S}$ as in
\cite{Hitchin:90}. Our results also recover and generalize those of \cite{SS}.
\end{remark}

By construction, we get the following: 

\begin{lemma}\label{lemma:parabolic}
The parabolic Hitchin symbol $\rho_{par}$ defined in 
\eqref{eqn:hitchparsymbol} satisfies   the condition in
Theorem \ref{thm:hitchimain} (i).
\end{lemma}

\subsection{Welters' condition}

In this subsection, we show 
that for $M={M}^{par,rs}_{G}$, the condition in
Theorem \ref{thm:hitchimain} (ii) is satisfied.
In fact, 
we will prove a stronger statement in the set-up of parabolic $G$-bundles.

\begin{lemma}\label{lem:welters1}
Let $M^{par,rs}_{G}$ be the moduli space of regularly stable parabolic $G$-bundles
on a curve $C$. Then $H^1(M^{par,rs}_{G}, \mathcal{O}_{M^{par,rs}_{G}})=0$.
\end{lemma}

\begin{proof}It suffices to show that the Picard group of the moduli space $M^{par,rs}_G$ is
discrete, since the space $H^1(M^{par,rs}_{G}, \mathcal{O}_{M^{par,rs}_{G}})$ can be 
considered as the Lie algebra of the Picard group of $M^{par,rs}_{G}$.
Hence, it is enough to show that the Picard group of the corresponding 
moduli stack $\mathcal{P}ar_{G}^{rs}(C,\vec{P})$ is discrete.
By  \cite{LaszloSorger:97}, 
it is known that the Picard group of the moduli stack
$\mathcal{P}ar_{G}(C,\vec{P})$ of quasiparabolic $G$-bundles is
discrete. Thus, we will be done if we can show that the codimension of
the complement of the regularly stable locus has codimension at least
two, as the inclusion will then  induce an isomorphism on the Picard groups
(cf.\ \cite[Lemma 7.3]{BH2}).  But this is the content of 
Lemma \ref{lem:codimension} below.
\end{proof}

\begin{lemma}\label{lem:tangentzero}
With the notation of Lemma \ref{lem:welters1},
$H^0(M^{par,rs}_{G}, \mathcal{T}_{M^{par,rs}_{G}})=0$.
\end{lemma}

\begin{proof} The proof follows the steps given in \cite{Hitchin:90}. Firstly,  $\mathcal{T}^{\vee}_{M_{G}^{par,rs}}$ embeds into the moduli space of strongly parabolic $G$-Higgs bundles $\mathcal{H}_{G}^{par, ss}$. 
Now given a global vector field on $M^{par,rs}_{G}$, pairing it with the cotangent bundles 
produces a function on $\mathcal{T}^{\vee}_{M_{G}^{par,rs}}$, which via Hartogs' theorem extends to
a function of degree one (with respect to the standard $\mathbb{C}^*$-action) 
on the Higgs moduli space $\mathcal{H}_{G}^{par, ss}$. Since 
the Hitchin fibration is proper (Lemma \ref{lem:proper}) with connected fibers
(\cite[Sec.\ 5]{Hitchin:87b},  \cite[Cor.\ III.3]{Faltings:93} and
    \cite[Claim 3.5]{DonagiPantev:12} for nonparabolic Higgs bundles;
    \cite[Cor.\ V.5]{Faltings:93} for strongly parabolic with full flags;
    \cite[Sec.\ 4.5]{BW1}, \cite[Thm.\ 1.2]{BW2} for all strongly parabolic cases)
it descends to a function on the Hitchin base. This is impossible since the
degree of homogeneity is one. Thus we are done.
\end{proof}
The following result is well known (\cite[Thm.\
13]{BaragliaKamgarpourVarma:19}, and also \cite{yokogawa} for $G=\GL_r$), but for completeness  we include a brief proof of it.

\begin{lemma}\label{lem:proper}The Hitchin map
    $\Hit:\mathcal{H}_{G}^{par,ss}\to \Bcal$ is proper. 
	\end{lemma}
\begin{proof}
	In \cite{BiswasRamanujan} strongly parabolic Higgs bundles on a curve $C$ are constructed as  $\Gamma$-$G$-Higgs
bundles on a $\Gamma$-cover $\widehat{C}$ of $C$. 
    Let $\Hcal_G^{ss}(\widehat C)$ denote the moduli of semistable Higgs
    bundles on $\widehat C$, with Hitchin base $\widehat\Bcal$, and Hitchin
    map $\Hit_{\widehat C}$.  Note that in the strongly parabolic setting,
    we have an inclusion $\imath: \Bcal\hookrightarrow\widehat \Bcal$.
    Then we have a commutative diagram:
    $$
    \begin{tikzcd}
        \mathcal{H}_{G}^{par,ss}\arrow[r,"F"]\arrow[d,"\Hit",swap]
        &\Hcal_G^{ss}(\widehat C) \arrow[d,"\Hit_{\widehat C}"] \\
        \Bcal \arrow[r,hook, "\imath"]& \widehat\Bcal
    \end{tikzcd}
    $$
    Here, $F$ is the forgetful map sending a $\Gamma$-$G$-bundle on
    $\widehat C$ to the underlying $G$-bundle.
    Since  $F$  and the Hitchin map $\Hit_{\widehat C}$
    are proper, and the map $\imath$ is a closed embedding, we conclude
    that $\Hit$ is also proper.
\end{proof}

Finally, we are in a position to prove the main theorem.

\begin{proof}[Proof of the Main Theorem]
For the conditions in Theorem  \ref{thm:hitchimain}: (i) is the
statement of Lemma \ref{lemma:parabolic}, (ii) follows from 
Lemma \ref{lem:welters1} and Lemma
\ref{lem:tangentzero}, and (iii) is the connectedness of the moduli space. 
For the conditions in Theorem  \ref{thm:flat}:  (i) follows as in
\cite{Hitchin:90}, using integrability results in
\cite{LogaresMartens:10}, \cite{BaragliaKamgarpourVarma:19} and \cite{BW2} (ii) follows 
from Proposition
\ref{prop:iso}, and  (iii)  is the statement in Lemma
\ref{lem:tangentzero}. This completes the proof.  
\end{proof}

We now apply the main theorem to extend the result to case of
the simple groups which are not necessarily simply connected.
 
\begin{proof}[Proof of Corollary \ref{cor:semisimpleG}] Take $s\in S$ and let
$C_{s}$ be the corresponding smooth $n$-pointed curve. Consider the moduli 
space ${M}^{par, rs,0}_{{{G}}}(C_{s})=\pi^{-1}(s)$
for a connected, simple group $H$. Let ${G}$ be the simply connected 
cover of $H$ and ${M}^{par, rs}_{{{G}}}(C_{s})$ 
the corresponding moduli space. 
Consider the map between moduli spaces 
${M}^{par, rs}_{{{G}}}(C_{s})$ and 
${M}^{par, rs,0}_{H}(C_s)$ induced by the quotient map $\widetilde{G} \rightarrow G$. 
This map is \'etale on the base with Galois group $\Gamma$ 
which is a subgroup 
of the center $Z({{G}})$ of ${G}$. Any element $\gamma \in \Gamma$, acts 
on ${M}^{par, rs}_{{{G}}}(C_{s})$ by twisting. This action of $\gamma$ evidently commutes with the Hitchin map.
Hence, if we  consider the same symbol as in the simply connected case, 
the same arguments in \cite[Cor.\ 5.2 and Lemma 4.1]{Belkale:08} tell us that 
the projective connection constructed for simply connected group
commutes with the action on $\Gamma$. Thus we see that $\pi_*\mathcal{L}_{\bm \lambda,k}$
is a twisted $D$-module, and so it is locally free. 
\end{proof}

\appendix
\section{Parabolic $G$-bundles} \label{sec:app-par}

Let $G$ be a simple, simply connected complex algebraic group 
and $(C,\vec{p})$  an
$n$-pointed smooth projective curve of genus $g$.
Let $\mathfrak{h}$ be a Cartan subalgebra of the Lie algebra $\mathfrak{g}$ of the group $G$. We 
further let $\Delta_+$ denote the set of simple positive roots $\alpha_i$, and let $\theta$ 
denote the highest root of $\mathfrak{g}$. 
Define the fundamental alcove
$$\Phi_0:=\{h \in \mathfrak{h} |\ \alpha_i(h)\geq 0, \  \mbox{and} \
\theta(h)\leq 1 \ \forall \ \alpha_i\}\ .$$
For $h\in \Phi_0$, we denote by $P(h)$ the standard parabolic subalgebra
of $\mathfrak{g}$, and $\mathfrak{p}(h)$ will 
denote the corresponding Lie subalgebra of $\mathfrak{g}$. The following result is standard 
and can be found in \cite[Thm.\ 7.9]{Helgason:78}:

\begin{lemma}\label{lem:exp}Let $K$ be a maximal compact subgroup of $G$.
The exponential map\break $$h\rightarrow \exp(2\pi\sqrt{-1}h)$$ induces a natural bijection 
between $\Phi_0$ and the set of $K$ orbits for the adjoint action of $K$ on itself. 
\end{lemma}

For any one parameter subgroup $\varphi: \mathbb{G}_m \rightarrow G$, 
the {\em Kempf's parabolic subgroup} is defined as:
$P(\varphi):=\{g \in G\,\,\mid\,\, \lim_{t\rightarrow
    0}\varphi(t)g\varphi(t)^{-1} \ \mbox{exists in $G$}\}$.
Every $\tau \in \Phi_0$ determines a $1$-parameter subgroup of $G$ and
hence by above  a parabolic subgroup  $P(\tau)$. It directly follows that
the Lie algebra of $P(\tau)$ is the Kempf's parabolic subalgebra 
\begin{equation*}\label{eqn:kempfparaalgeb}
\mathfrak{p}(\tau):=\{X\in \mathfrak{g}\,\,\mid\,\, \lim_{t\rightarrow
    \infty}\operatorname{Ad}(\exp t\tau)\cdot X \ \mbox{exists in
    $\mathfrak{g}$}\}\ .
\end{equation*}

We now recall the definition of the moduli stack of quasi-parabolic
bundles. We refer the reader to \cite[Ch.\ 5.1]{kumarnewbook}
\begin{definition}\label{def:quasiparabolicstack}The quasi parabolic moduli stack $\mathcal{P}ar_{G}
	(C,\vec{P})$ is the stack parametrizing pairs $(\mathcal{E},\vec{\sigma})$, 
	where $\mathcal{E}$ is a principal $G$-bundle on a smooth curve $C\times T$, 
	with $T$ being any scheme, and $\sigma_i$ are sections over $T$ 
	of $\mathcal{E}_{|p_i\times T}/P_i$ while $\vec{P}=(P_1,\dots,P_n)$ 
	are an $n$-tuple of standard parabolic subgroups of $G$. 
\end{definition}
We now recall the definition of a parabolic $G$-bundle on a smooth pointed curve $(C,\vec{p})$. 

\begin{definition}
A parabolic structures on a principal $G$-bundle $E\rightarrow C$ is given by the following data:
\begin{itemize}
	\item A choice of parabolic weights ${\bm \tau }=(\tau_1,\dots, \tau_n) \in \Phi_0^n$, 
	where $\tau_i$ is the  parabolic weight attached to the point $p_i \in C$. 
	\item a section $\sigma_i$ of the homogeneous space $E_{p_i}/P({\tau_i})$, 
	where $P(\tau_i)$ is the standard parabolic associated to $\tau_i \in \Phi_0$. 
\end{itemize}
\end{definition} 
A family of parabolic $G$-bundles parametrized by a scheme $T$ is defined analogously. 
of a section $\sigma_i$ for every $1\leq i\leq n$.  
Similarly extend the definitions of parabolic structures  when $G$ is connected and reductive. 

\subsection{Uniformization of quasiparabolic bundles}
Most of the results in this section can be easily modified for semi-simple groups, however for simplicity we restrict ourselves to the case when $G$ is simple and simply connected. 

For any simple, simply algebraic group 
$G$, let $\operatorname{L}_{G}$ be the corresponding loop group and 
$\operatorname{L}^{+}_{G}\subset \operatorname{L}_{G}$ the subgroup of positive loops. 
The affine Grassmannian $\mathcal{Q}_{G}$ is defined to be $\operatorname{L}_{G}/\operatorname{L}^{+}_{G}$. 
Let $q$ be a point on a the curve $C$. 

Consider the functor 
$\mathscr{S}_{G, C\backslash q}$ from the category of $k$-algebras $\mathcal{A}lg$ to  the category  $\mathcal{S}ets$ that
assigns to an $k$-algebra $R$, isomorphism classes of pairs $(E_{R},\sigma_R)$, where $E_{R}$ is a principal $G$ bundle
over $X\times \operatorname{Spec}R$ and $\sigma $ is a section of $E_{R}$ over $(C\backslash q)\times
\operatorname{Spec}R$ (cf.\ \cite[Sec.\ 3.5]{LaszloSorger:97}). 
The following statement, which uses a crucial {\em uniformization} result of Drinfeld-Simpson \cite{DS},
gives a geometric realization of the affine Grassmannian $\mathcal{Q}_{G}$.

\begin{proposition}
The affine Grassmannian $\mathcal{Q}_{G}$ represents the functor
$\mathscr{S}_{G, C\backslash q}$. Moreover, there is a universal
principal $G$-bundle $\mathbb{U}\to C\times \mathcal{Q}_{G}$ and  section
$\sigma_{\mathcal{Q}_{G}}$, such for any $[E_R,\sigma_R]\in
\mathscr{S}_{G,C\backslash q}$ and any morphism $f: \operatorname{Spec}R \rightarrow \mathcal{Q}_{G}$,
$$[(\id\times f)^*\mathbb{U}, (id \times
    f)^*\sigma_{\mathcal{Q}_{G}}]=[E_{R}, \sigma_R]\ .$$
\end{proposition}

Let $\operatorname{L}_{C,\vec{p}}(G)$ be the punctured loop ind-group  $\operatorname{Mor}(C\backslash \vec{p}, G)$ that
parametrizes morphisms $C\backslash \vec{p} \rightarrow G$ from the punctured curve. The following result of Laszlo-Sorger
\cite{LaszloSorger:97} expresses the moduli stack of principal $G$-bundles as a quotient stack.

\begin{proposition}\label{prop:Laszlo}
The stacks $\mathcal{P}ar_{G}(C,\vec{P})$ and $\operatorname{L}_{C,q}(G)\backslash \left( \mathcal{Q}_{G}
\times \prod_{i=1}^nG/P_i\right)$ are isomorphic, where $q$ is a point on $C\backslash \vec{p}$ and $\operatorname{L}_{C,q}(G)$
acts on $G/P_i$ by evaluation at the point $p_i$. Moreover, $\operatorname{Pic}(\Par_G(C,\vec{P}))\cong
\mathbb{Z}\times \prod_{i=1}^n \operatorname{Pic}(G/P_i)$ if $G$ is simply connected. 
\end{proposition}

We now describe another uniformization of the moduli stack $\mathcal{P}ar_{G}(C,\vec{P})$ that connects directly to the moduli stack of 
$\Gamma$-equivariant bundles of fixed topological type that will be discussed in Appendix \ref{sec:gammabundles}.
For an $n$-tuple of points $\vec{p}=(p_1,\dots, p_n)$, we choose formal parameters $t_i$ at $p_i$, i.e.,
$\widehat{\mathcal{O}}_{C,p_i}=\mathbb{C}((t_i))$. 
Consider the natural evaluation map at $t_i=0$, 
$ev_0: G[[t_i]] \lra G$, from the Iwahori subgroup $G[[t_i]]$.
For any standard parabolic subgroup $P_i \subset G$, we denote by $\mathcal{P}_j:=ev_0^{-1}P_i$
the standard parahoric subgroup of the loop group.
Now consider the reduced ind-scheme $\operatorname{L}_{C,\vec{p}}(G)$ as discussed above. Then any element of
$\operatorname{L}_{C,\vec{p}}(G)$ acts on $G((t_i))/\mathcal{P}_j$ via Laurent expansion at the point $p_i$
in the local parameter $t_i$. As in Proposition 2.8 of \cite{KNR:94}, we have a family 
of principal $G$-bundles $\mathbb{U}_{par}$ on $C\times \prod_{i=1}^nG((t_i))/G[[t_i]]$ such that the following three hold:
\begin{enumerate}
\item The bundle $\mathbb{U}_{par}$ is $\operatorname{L}_{C,\vec{p}}(G)$ equivariant.

\item There is a section $\sigma_{par}$ of $\mathbb{U}_{par}$ over $(C\backslash \vec{p} )\times \prod _{i=1}^nG((t_i))/G[[t_i]]$
which extends to a section on a formal disc around the punctures $p_i$. 

\item The section $\sigma_{par}$ satisfies the condition
$\gamma\cdot \sigma(q,[g_1],\dots, [g_n])=\sigma(q,[g_1],\dots,
        [g_n])\gamma(q)$, where
$[g_i]$ is the  class of an element $g_i \in G((t_i))$, $\gamma \in \operatorname{L}_{C,\vec{p}}(G)$ and $q\in C\backslash
\vec{p}$. Moreover, the pair $(\mathbb{U}_n, \sigma)$ is unique up to an unique isomorphism satisfying the above properties. 
\end{enumerate} 
Now pulling back $\mathbb{U}_{par}$ via the natural $\operatorname{L}_{C,\vec{p}}(G)$-equivariant  projection
$$\prod_{i=1}^n G((t_i))/\mathcal{P}_i\rightarrow \prod_{i=1}^n
G((t_i))/G[[t_i]]\ ,$$
we obtain a natural
$\operatorname{L}_{C,\vec{p}}(G)$-equivariant principal $G$-bundle on 
$C\times  \prod_{i=1}^n G((t_i))/\mathcal{P}_i$. Hence, using this
$G$-bundle $\mathbb{U}_{par}$ and the section $\sigma_{par}$, we obtain the
following well known result  (\cite[pp.\ 181-182]{kumarnewbook}, and also
\cite[Prop.\ 3.3]{BalajiSeshadri}, \cite[Thm.\ 1.3]{LaszloSorger:97}): 

\begin{proposition}\label{prop:paravacua}
The stack $\mathcal{P}ar_{G}(C,\vec{P})$ is isomorphic to $\operatorname{L}_{C,\vec{p}}(G)\backslash \prod_{i=1}^n G((t_i))/\mathcal{P}_i.$
\end{proposition}

\subsection{Parabolic bundles and associated constructions}

Let $P\subset G$ be a standard parabolic subgroup  with Levi subgroup $L_P$ containing a maximal torus $H$. Consider the set $S_P$ of simple 
roots of the Levi subalgebra $L_P$ of the parabolic $P$. If $P=P(h)$ for some $h\in \Phi_0$, then  $S_{P}:=\{\alpha_i \in
\Delta_+\mid \alpha_i(h)=0\}.$
The group of characters $X(P)$ of the parabolic subgroup $P$ can be identified with the subset of the dual Cartan subalgebra
$$\mathfrak{h}_{\mathbb{Z},P}^{\vee}:=\{\lambda \in \hfrak^\vee\,\,\mid  \ \lambda(\alpha_i^{\vee})\in \mathbb{Z}, \, \forall \alpha_i, \ \mbox{and} \
\lambda(\alpha_i^{\vee})=0, \, \forall \alpha_i \in S_P\}.$$
In terms of the fundamental weights $\omega_1,\dots, \omega_\ell$ of the Lie algebra $\mathfrak{g}$, we get that
$$\mathfrak{h}_{\mathbb{Z},P}^{\vee}:=\bigoplus_{\alpha_i \neq S_P} \mathbb{Z}\omega_i.$$
Let ${\bm \tau}=(\tau_1,\dots, \tau_n) \in \Phi_0^n$ be a choice of parabolic weight. 

We further assume that each $\tau_i \in \Phi_0$ is {\em rational }, i.e., we can write
$\tau=\overline{\tau}_i/d_i $ for some positive integers $d_i$ and
$\exp(2\pi\sqrt{-1}\overline{\tau}_i)=1$, so $d_i\cdot\tau$ is in the coroot lattice (i.e. in the lattice spanned by the set of coroots $\Phi^{\vee}\subset \mathfrak{h})$.  The integers $d_i$ are not unique. 

If $G=\SL_r$ and consider the standard representation of $\SL_r$,  then a choice of a rational $\tau \in \Phi_0$ via the normalized Killing form $\kappa$ is the same as the choice of an integer $k\leq r$, a sequence of integers $\bm{r}:=(r_1,\dots, r_k)$ such that $\sum_{i=1}^{k}r_i=r$ and a nondecreasing sequence $0\leq \alpha_1< \dots < \alpha_{k}< 1$. Hence a rational  parabolic structure on a vector bundle $\mathscr{V}$ on a curve $C$ associated to a parabolic $\SL_r$-bundle at the points $p_1,\dots,p_n$ is equivalent to the following: 
\begin{enumerate}
\item A choice of a flag of the fiber $\mathscr{V}_{|p_i}$  associated to the $k_i$-tuple $\bm r_i$ for each $1\leq i\leq n$
\begin{equation*}
\mathscr{F}_{\bullet, p_i}:=\left(0\subseteq
F_{k_{i}+1}(\mathscr{V}_{|p_i})\subseteq
F_{k_i}(\mathscr{V}_{|p_i})\subseteq \dots \subseteq F_{1}(\mathscr{V}_{|p_i})=\mathscr{V}_{|p_i}\right)
\end{equation*}
such that $\dim \operatorname{Gr}^j \mathscr{F}_{\bullet, p_i}=r_{j,p_i}$.
\item For each $p_i$, a sequence of rational numbers $\bm{\alpha}_{p_i}$
\begin{equation}\label{equation:rationalweights}
0\leq \alpha_{1,i}<  \dots <\alpha_{k_i,i}<1.
\end{equation}
\end{enumerate}
We refer the reader to Mehta-Seshadri \cite{MehtaSeshadri} (for parabolic vector bundles), Ramanathan \cite{Ramanathan:96}, Biswas (\cite{Bisstable} and \cite {Biswasduke}), Balaji-Seshadri \cite{BalajiSeshadri} and Balaji-Biswas-Nagaraj \cite{BalajiBiswasNagaraj} for the notions of stability and semistability which is essential in defining the corresponding moduli spaces.

The following theorem is due to Mehta-Seshadri \cite{MehtaSeshadri} for parabolic vector bundles of rank $r$ and weight data $\bm \alpha$ and we will denote the moduli space by $M^{par,ss}_{\bm \alpha,r}(C)$.
It was proven for arbitrary semi-simple groups by Bhosle-Ramanathan
\cite{DesaleRamanathan}.  Following the work of Seshadri
(\cite{SeshadriI} and \cite{SeshadriII}), Balaji-Biswas-Nagaraj
\cite{BalajiBiswasNagaraj}, Balaji-Seshadri \cite{BalajiSeshadri}, we will discuss an alternative realization in the following section. 

\begin{theorem}
Let $(C,\vec{p})$ be a $n$-pointed smooth projective curve of genus $g$, and let ${\bm \tau}=(\tau_1,\dots, \tau_n)$ be a choice of rational parabolic weights in the fundamental alcove $\Phi_0$. 
We further assume that $\theta(\tau_j)<1$, where $\theta$ is the highest root of $\mathfrak{g}$. Then, the parabolic semistable $G$-bundles with a choice of rational parabolic weights ${\bm \tau}$ admit a coarse moduli space ${M}_{G,{\bm \tau}}^{par,ss}(C)$ which is a normal irreducible projective variety with rational singularities. 
Moreover, if $\iota: G\rightarrow G'$ is an embedding of connected simple, simply connected groups, then the  corresponding map between the moduli space ${M}_{G,{\bm \tau}}^{par,ss}\rightarrow {M}_{G',{\bm \tau}'}^{par,ss}$ is finite. Here ${\bm \tau}'=\iota({\bm \tau})$. 
\end{theorem}

For notational convenience,  when the context is clear we will often suppress the subscript ${\bm \tau}$ and use $M_{G}^{par,ss}$
instead.

\begin{definition}
A parabolic $G$-bundle $\mathcal{P}$ with weights $\bm {\tau}$ is said to be regularly stable if it is stable and
the automorphism group of $\mathcal{P}$ is the center $Z(G)$ of $G$.
\end{definition} 

\subsection{Line bundles on parabolic moduli spaces}\label{sec:det}

In this section, we first recall the {\em determinant of cohomology } line bundle associated to a family of vector bundles $\mathcal{E}$ on a curve $C$ parametrized by a connected Noetherian scheme $T$. Let $\pi_T: T\times X \rightarrow T$ be the projection to the Noetherian
scheme, and consider $R\pi_{T,*} \mathcal{E}$ as an object of the bounded derived category $D^bCoh(T)$. 
We can represent $R\pi_{T,*}\mathcal{E}$ by a complex $\mathcal{E}_0\rightarrow \mathcal{E}_1\rightarrow 0$ of
vector bundles on $T$. We define the determinant line bundle up to a unique isomorphism to be the following: 
\begin{equation*}
\operatorname{Det}\mathcal{E}_T :=\bigwedge^{top} \mathcal{E}_1 \otimes \bigwedge^{top}\mathcal{E}_0^{\vee}.
\end{equation*}
We often drop $T$ in the notation of $\operatorname{Det}\mathcal{E}_T$
when the context is clear. For any closed point $t\in T$, the fiber of $\operatorname{Det}\mathcal{E}_T$
over $t$ is $\bigwedge^{top} \left(H^1(C, \mathcal{E}_t)\right)\otimes \bigwedge^{top}\left(H^0(C,\mathcal{E}_t)\right)^{\vee}.$
The determinant bundle has the following important properties: 
\begin{enumerate}
\item For any morphism $f: T'\rightarrow T$, we have 
$\operatorname{Det}(f\times \operatorname{id})^*\mathcal{E})_{T'}=f^*\operatorname{Det}\mathcal{E}_T.$

\item For any line bundle $L\to T$, we have 
$\operatorname{Det}(\mathcal{E})_T\otimes
        L^{-\chi(\mathcal{E}_0)}=\operatorname{Det}\operatorname(\mathcal{E}\otimes
        \pi_T^*L)_T$,
where $\chi(\mathcal{E}_0)$ is the Euler characteristic of the vector bundle $\mathcal{E}_{|t\times C}$ for
any point $t\in T$. 

\item For any short exact sequence of bundles
$0\to
\mathcal{E}_1\to \mathcal{E} \to \mathcal{E}_2\to 0$
on $T\times C$, we have 
$\operatorname{Det}\mathcal{E}_{1,T}\otimes \operatorname{Det}\mathcal{E}_{2,T}=\operatorname{Det}\mathcal{E}_{T}$.
\end{enumerate}
Let $\mathscr{SU}_C(r,\xi)$ be the moduli space of semistable
vector bundles of rank $r$ on a curve $C$ with determinant $\xi$ of degree $m$. It was proved by \cite{DN} that
the Picard group of $\mathscr{SU}_C(r,\xi)$ is $\mathbb{Z}\cdot\Theta$, where $\Theta$ is the ample
generator. The following result of Drezet-Narasimhan \cite{DN} connects the determinant of cohomology
with this $\Theta$-line bundle. 

\begin{proposition}\label{proposition:pullbacknarasimhan}
Let $\psi_{\mathcal{E}}: T\rightarrow \mathscr{SU}_C(r,\xi)$ be the morphism
corresponding to a family $\mathcal{E}$ of semistable bundles of rank $r$ and determinant $\xi$ parametrized
by a scheme $T$. Then the pullback of $\Theta$ via $\psi_{\mathcal{E}}$ is isomorphic to 
$\left(\operatorname{Det}\mathcal{E}_T\right)^{\frac{r}{(r,m)}}\otimes \left(\det \mathcal{E}_{|T\times p}\right)^{\frac{\chi}{(r,m)}}$, where $p$ is any point on the curve $C$, $m$ is the degree of the line bundle $\xi$, $(r,m)$ is the greatest common divisor and $\chi=\chi(F_{|t\times C})=m+r(1-g)$. 
\end{proposition}

Motivated by the above proposition, we define the following:

\begin{definition}
For any family $\mathcal{E}$ of vector bundles of rank $r$ and determinant $\xi$ of degree $m$, parametrized by a connected Noetherian scheme $T$, we define the theta-bundle
\begin{equation}\label{eqn:theta}
\Theta(\mathcal{E}):=\left(\operatorname{Det}\mathcal{E}_T\right)^{\frac{r}{(r,m)}}\otimes \left(\det \mathcal{E}_{|T\times p}\right)^{\frac{\chi}{(r,m)}},
\end{equation}
where  $\chi$ as in Proposition \ref{proposition:pullbacknarasimhan} is the Euler characteristic.
\end{definition}

Note that for any line bundle $\mathcal{L}$ over $T$, we have an isomorphism $\Theta(\mathcal{E})\cong \Theta 
(\mathcal{E}\otimes \pi_T^*L)$.
Similarly for any simple, simply connected algebraic group $G$ and any family $\mathcal{E}$ of principal $G$-bundles on $C$ parametrized by
a scheme $T$, we can associate a natural line bundle on $T$ as follows: 
Let $(\varphi,V)$ be a representation of the group $G$. Then the associated
vector bundle $$\mathcal{E}(V):=\mathcal{E}\times ^{\varphi}V$$ is a family of vector bundles on $C$
parametrized by $T$. Observe that since $G$ is simple, and hence $G$ does not have any nontrivial
character, it follows that $\mathcal{E}(V)$ has trivial
determinant over $T\times C$. We define a line bundle on $T$
\begin{equation}\label{eqn:determinG}
\operatorname{Det}(\mathcal{E}, \varphi)_T:=\operatorname{Det}(\mathcal{E}(V))_T.
\end{equation}
It follows from \eqref{eqn:theta} that $\Theta(\mathcal{E}(V))=\operatorname{Det}(\mathcal{E},\varphi)_T$.

\subsubsection{The parabolic determinant of cohomology in the  $\SL_r$ case}

We follow the notation and conventions as in \cite{BiswasRaghavendra}. Let $\mathcal{E}$ be a family of
quasiparabolic $\SL_r$ bundles on a pointed curve $(C,\vec{p})$ parametrized by a scheme $T$ considered as a parabolic vector bundle via the standard representation. Let
$\bm \alpha:=(\bm{\alpha}_{p_1},\dots, \bm \alpha_{p_n})$ be a $n$-tuple of sequence of rational
numbers as in \eqref{equation:rationalweights} associated to each marked point $p_i$, $1\leq i\leq n$. 
Consider the following element  in $\operatorname{Pic}(T)\otimes \mathbb{Q}$,
\begin{equation}
\label{equation:parabolicdeterminantsl}
\operatorname{Det}\mathcal{E}_T+\sum_{i=1}^n\sum_{j=1}^{k_i}\alpha_{j,i}\det
\operatorname{Gr}^j\left(\mathscr{F}_{\bullet,
p_i}(\mathcal{E}_{|T\times p_i})\right)\ ,
\end{equation}
where the rational number $0\leq \alpha_{1,i}< \alpha_{2,i}< \dots< \alpha_{k_{i},i}<1$ 
define $\bm \alpha_{p_i}$. Write
$\alpha_{j,i}=b_{j,i}/q_{j,i}$,
where $b_{j,i}$ and $q_{j,i}$ are relatively prime integers. 

\begin{definition}
\label{def:level}
Let $N$ be the least common multiple of all $\{q_{j,i}\}_{i,j}$, 
    $1\leq i\leq n$ and $1\leq j\leq k_{i}$.
    We refer to $N$ as the level of the weight $\bm \alpha$. 
\end{definition}
Consider the  integers $a_{j,i}:=N\cdot\alpha_{j,i}$. Then for each $1\leq i\leq n$ and $1\leq j\leq k_i$,
\begin{equation*}
0\leq a_{1,i}< a_{2,i}< \dots < a_{k_i,i} \leq N-1.
\end{equation*} 
\begin{definition}\label{def:parabolcidetSL}
	Let $\mathcal{E}$ be a family of  degree zero parabolic vector bundles  on $T\times C$ with parabolic data $\{(\bm r_i ,\bm \alpha_i)\}_{i=1}^n$.
	The parabolic determinant bundle on $T$ is defined to be 
	\begin{equation*}
	\operatorname{Det}_{par}\mathcal{E}_T(\bm
        \alpha):=\left(\operatorname{Det}\mathcal{E}_T\right)^{\otimes
        N}\bigotimes \bigl(\otimes_{i=1}^{n}\bigl(
        \otimes_{{j=1}}^{k_i}\det \operatorname{Gr}^j
        \mathscr{F}_{\bullet,p_i}(\mathcal{E}_{|T\times
        p_i})^{a_{j,i}}\bigr)\bigr)\ .
	\end{equation*}
\end{definition}
This is just eq.\ \eqref{equation:parabolicdeterminantsl} multiplied by $N$. When the context is clear, we will 
simply denote $\operatorname{Det}_{par}\mathcal{E}_T(\bm \alpha)$ by $\operatorname{Det}_{par}\mathcal{E}_T$. The 
line bundle $\operatorname{Det}_{par}(\mathcal{E}_T)$ may not descend to the moduli space, so we consider the 
following modification.

\begin{definition}
The parabolic $\Theta_{par}$-line bundle on $T$ is defined to be
the following twist of parabolic determinant of cohomology
\begin{equation*}
\Theta_{par}(\mathcal{E}, \bm{\alpha}):=\left(\operatorname{Det}_{par}\mathcal{E}_T\right)\otimes
\left( \det \mathcal{E}_{|T\times p_0}\right)^{\frac{N.\chi_{par}}{r}}
\end{equation*}
(just as in the nonparabolic case), where $\chi_{par}=\chi- \sum_{i,j}\alpha_{j,i}r_{j,i}$ is the parabolic
Euler characteristic (see \cite[p.\ 60]{BiswasRaghavendra}), $\chi$ is as in Proposition
\ref{proposition:pullbacknarasimhan} and $p_0$ is any point on the curve $C\backslash \vec{p}$.
	
\end{definition}
We remark that the definition of $\Theta_{par}(\mathcal{E}, \bm{\alpha})$
differs from the definition of parabolic determinant \cite[Def.\ 4.8]{BiswasRaghavendra} by a multiplicative
factor of $r$.  The following proposition can
be found in Biswas-Raghavendra \cite{BiswasRaghavendra},  Pauly
\cite{Pauly:96}, and in Narasimhan-Ramadas \cite{NarRam} for $G=\SL_2$. 

\begin{proposition}
	Let $\psi_{\mathcal{E}}: T \rightarrow {M}^{par,ss}_{ \bm{\alpha},r}(C)$ be a map
from a scheme $T$ to the Mehta-Seshadri moduli space ${M}^{par,ss}_{ \bm{\alpha},r}(C)$ of parabolic bundles corresponding to a family
$\mathcal{E}$ equipped with parabolic data $\bm \alpha$. Then there exists an ample line bundle
$\Theta_{par}(\bm{\alpha})$ on ${M}^{par,ss}_{ \bm{\alpha},r}(C)$ such that  $\psi_{\mathcal{E}}^*\Theta_{par}
(\bm{\alpha})$ is isomorphic to the line bundle $\operatorname{Det}_{par}\mathcal{E}_T\otimes
\left( \det \mathcal{E}_{|T\times p_0}\right)^{\frac{N.\chi_{par}}{r}}$. 
\end{proposition}

As discussed, the choice of the standard representation gives a map of the moduli stacks $\xi: M_{\SL_r, \bm \alpha}^{par,ss}(C)\rightarrow M^{par,ss}_{\bm \alpha}(C)$, the map $\psi_{\mathcal{E}}$ factors through $M_{\SL_r, \bm \alpha}^{par,ss}$ and we will use the notation $\Theta_{par}(\bm \alpha)$ to also denote the pull back $\xi^*\Theta_{par}(\bm \alpha)$.
\subsubsection{The case of general groups}

We first recall the notion of Dynkin index of an embedding. Let $\phi: \mathfrak{s}_1\rightarrow \mathfrak{s}_2$ be 
a map of two simple Lie algebras, and let $\kappa_{\mathfrak{s}_1}$ (respectively, $\kappa_{\mathfrak{s}_2}$) be the 
normalized Killing form of $\mathfrak{s}_1$ (respectively, $\mathfrak{s}_2$).

\begin{definition}
The Dynkin index $m_{\phi}$ of a map of simple Lie algebras $\phi$ is the ratio of their normalized Killing
forms, in other words, $\kappa_{\mathfrak{s}_2}(\ , \ )_{|\mathfrak{s}_1}\,=\,m_{\phi}\kappa_{\mathfrak{s}_1}( \ , \ )$.  
\end{definition}

Let $G$ be a simple, simply connected group, and let $\mathcal{E}$ be a principal $G$ bundles on $T\times C$. Let $(\phi,V)$ 
be a representation of $G$, and consider the associated vector bundle $\mathcal{E}(V):=\mathcal{E}\times^{G}V$ on 
$T\times C$.
Since $G$ does not have any nontrivial character (it is simple), it follows that
$\det \mathcal{E}(V)\cong \mathcal{O}_{T\times C}$. This implies 
$\Theta(\mathcal{E}(V))=\operatorname{Det}(\mathcal{E}(V))_T.$ Let $(\mathcal{E},\vec{\sigma})$ be a family of quasiparabolic $G$-bundles of type $\vec{P}=(P_1,\dots, P_n)$ on a $n$-pointed curve $(C,\vec{p})$ parametrized by a scheme $T$. 

\begin{definition}\label{def:parabolicdeteG}
For any positive integer $d$ (usually it will be determined by the weights $\bm \mu$), a finite dimensional
representation $(\phi, V)$ of the group $G$ and a character $\mu_j$ of the parabolic $P_j$, define a line bundle
on $T$ by the following formula:
\begin{equation}\label{eqn:deterG}
\operatorname{Det}_{par}(\mathcal{E}(V),d,\bm{\mu}):=\left(\operatorname{Det}(\mathcal{E}(V))_T\right)^{\otimes
    d}\bigotimes\bigl( \otimes_{j=1}^n
    \sigma_j^*\bigl(\mathcal{E}\times^{P_j}\mathbb{C}_{\mu_j^{-1}}\bigr)\bigr)
\end{equation}
(see \cite{LaszloSorger:97}), where $\bm \mu=(\mu_1,\dots, \mu_n)$ and $\mathbb{C}_{\mu_j^{-1}}$ is the one
dimensional representation of the parabolic subgroup $P_j$ corresponding to the character $\mu_j^{-1}$
of it. This line bundle will be called the {\em quasiparabolic determinant bundle}. We will refer to the
integer $d$ as the level of the quasiparabolic determinant bundle.
\end{definition}
Now let ${\bm \tau}=(\tau_1,\dots, \tau_n)$ be an $n$-tuple of rational parabolic weights such
that $\theta(\tau_i)<1$ for all $1\leq i\leq n$, where $\theta$ is the highest root of $\mathfrak{g}$. Consider
a representation of $G$ of $V$ such that 
\begin{enumerate}
\item  the representation $(\phi, V)$ is faithful; 

\item the topological local type $\phi({\bm \tau})$ of the associated bundle is rational;

\item $\theta_{\mathfrak{sl}(V)}(\phi (\tau_i))<1$ for all $1\leq i\leq n$, where $\theta_{\mathfrak{sl}(V)}$
is the highest root of $\mathfrak{sl}(V)$. 
\end{enumerate}

We now recall the definition of the parabolic theta bundle for any simple group $G$. Using the Killing form 
$\kappa_{\gfrak}$ we will identify $\nu_{\gfrak}:\mathfrak{h}\stackrel{\cong}{\longrightarrow} \mathfrak{h}^{\vee}$ and realize $\bm \tau$ in 
the weight lattice of $P$ of $G$.
Let $(\phi,V)$ be a faithful representation of $G$ satisfying the above conditions, and let $d$ be any positive integer such that
\begin{equation} \label{eqn:integral}
\exp(2\pi\sqrt{-1}\nu_{\mathfrak{sl}(V)}(d\cdot\phi(\tau_i)))=1
\end{equation}
for all $1\leq i\leq n$. This $d$ is not unique but usually one choose a minimal such $d$ and denote it by $N$.  

\begin{definition}\label{def:parGtheta}
The parabolic theta bundle $\Theta_{par, G}(V, \bm \tau)\to{M}^{par,ss}_{G,{\bm \tau}}$ is
defined to be the pull-back  of  $\Theta_{par,\SL(V)}(\phi({\bm \tau}))\to{M}^{par,ss}_{\SL(V),\phi({\bm \tau})}$ via the map
$\overline{\phi}: {M}^{par,ss}_{G,{\bm \tau}} \rightarrow {M}^{par,ss}_{\SL(V),\phi({\bm \tau})}$
induced by the representation $(\phi, V)$ of $G$, i.e., 
$\Theta_{par,G}(V,\bm \tau):=\overline{\phi}^*\Theta_{par,\SL(V)}(\phi(\bm \tau)).$
\end{definition}

The following well known result analogous to the $\SL_r$ case (cf.\
\cite[Lemma 8.5.5]{kumarnewbook}) relates
the parabolic determinant of cohomology for arbitrary simple, simply connected groups $G$ to the parabolic theta bundle. 

\begin{proposition}
Let $\mathcal{E}$ be a family of parabolic $G$-bundles parametrized by a scheme $T$ with parabolic data
${\bm \tau} \in \Phi_0^n$ satisfying the condition $\theta(\tau_i)<1$ for all $1\leq i\leq n$, and let $\psi_{\mathcal{E}}: T\rightarrow M_{G,\bm \tau}^{par,ss}$ be as before the map induced by $\mathcal{E}$. 

Further, let
$(\phi, V)$ be a representation of $G$ satisfying the above conditions. Then the pull-back 
$\psi_{\mathcal{E}}^*\left( \Theta_{par,G}(V,{\bm \tau}
)\right)$ equals $\operatorname{Det}_{par}(\mathcal{E}(V),N\cdot m_{\phi}\cdot
\nu_{\gfrak}({\bm \tau})), $ where $m_{\phi}$ is the Dynkin index of the map $\phi: \mathfrak{g} \rightarrow \mathfrak{sl}(V)$, $\operatorname{Det}_{par}(\mathcal{E}(V),N\cdot m_{\phi}\cdot
\nu_{\gfrak}({\bm \tau}))$ is as in Equation \eqref{eqn:deterG} and $N$ is the minimal positive integer satisfying \eqref{eqn:integral} in Definition \ref{def:parGtheta}.
\end{proposition}

\section{$\Gamma$-equivariant $G$-bundles}\label{sec:gammabundles}

In this section, we recall the correspondence between parabolic bundles on a curve $C$ and equivariant 
bundles on a ramified Galois cover $\widehat{C}\to C$  with
Galois group $\Gamma$. Throughout this section $G$ will be a simple, simply connected (or more generally simple but not simply connected) algebraic group.  
We  start with the well-known genus computation of an orbifold curve. Let
$\vec p =(p_1,\dots, p_n) $ be points in $C$,  and choose
positive integers
$\vec d=(d_1,\dots, d_n)$, respectively. 
\begin{definition}\label{def:orbigenus}
    The orbifold genus associated to $(C, \vec p, \vec d)$ is 
    $$g(\mathscr{C}):=g(C)+\frac{1}{2}\sum_{i=1}^n(1-\frac{1}{d_i})\ ,$$
    where $g(C)$ be the genus of the curve $C$.
\end{definition}
If $C$ is a quotient of $\widehat C$ by $\Gamma$ with ramification
locus $p_1,\dots, p_n$ of degrees $(d_1,\dots, d_n)$,
then by the Riemann-Hurwitz formula the genus  of $\widehat{C}$ is given by the
    formula:\break 
    $2-2g(\widehat C)=|\Gamma|
    \bigl( 2-2g(C)+\sum_{i=1}^n(\frac{1}{d_i}-1)\bigr)
    $.
The genus of the 
 quotient stack $\mathscr{C}:= [\widehat{C}/\Gamma]$ 
 is related to $g(\widehat{C})$ by the formula 
$g(\widehat{C})-1=|\Gamma|(g(\mathscr{C})-1)$, and so we see that this is
the orbifold genus defined above.

Conversely, given $\vec{p}$ 
and $\vec d$,  then provided $g(\Cscr)\geq 1$
 we can find a branched cover $\widehat C\to C$ as
above. For example, if $g(\Cscr)>1$ (we shall only be interested in this
case), then $C$ can be realized as a  quotient of the 
upper half plane $\mathbb{H}$ by a Fuchsian group $\Pi$ (cf.\ \cite[Sec.\
3.2]{Troyanov:91}). The action of
$\Pi$ is not free: it contains elliptic elements of order $d_i$ in the
points over $p_i$. 
 Applying the Selberg lemma  to $\Pi \subset
\operatorname{Aut}(\mathbb{H})$ (cf.\ \cite{Selberg}), we obtain a normal subgroup
$\Pi_0$ of finite 
index that acts freely on $\mathbb{H}$. Let $\widehat{C}=\mathbb{H}/\Pi_0$.
Since the action of $\Pi_0$ is free, 
we get that $\widehat{C}$ is a smooth projective curve. If we set
$\Gamma=\Pi/\Pi_0$, then the natural map 
$\widehat{C}\rightarrow C$ is a ramified Galois cover with Galois group $\Gamma$.

\begin{example}\label{example:genuszero}
Assume that $g(C)\,=\, 0$, $d_1=\dots=d_n=d$ and $d$ divides $n$. Then the super-elliptic curve $\widehat{C}$ given by the
equation $y^d=\prod_{i=1}^n(x-p_i)$ is a ramified Galois covering of $C=\mathbb{P}^1$. The  Galois group is
$\mathbb{Z}/d\mathbb{Z}$ with ramifications of order $d$ exactly at the points $p_1,\dots, p_n$, and \'etale
on the complement. Then we have $g(\mathscr{C})=n(d-1)/2d$.
    Hence, $g(\mathscr{C})\geq 1$ if $n\geq 2d/(d-1)$.  
\end{example}

\begin{definition}
Let $p:\widehat{C} \rightarrow C$ be a ramified Galois cover with Galois
    group $\Gamma$. A $\Gamma$-$G$-bundle 
$\widehat{E}$ on $\widehat{C}$ is a principal $G$ bundle on $\widehat{C}$ together with a lift of the action of 
$\Gamma$ on $\widehat{C}$ to an action of $\Gamma$ on the total space of $\widehat{E}$
as bundles automorphism (meaning the actions of $\Gamma$ and $G$ on $\widehat{E}$ commute).
\end{definition}

Let $\mathcal{R}$ denote 
the set of branch points of $C$. For each point ${p}\in \mathcal{R}$, we
choose a point $\widehat{p} \in 
\widehat C$ in the preimage of $p$, and let $\Gamma_{\widehat{p}}\subset
\Gamma$ denote the stabilizer of the point ${\widehat{p}}$.

\begin{definition}[{Balaji-Seshadri \cite{BalajiSeshadri}}] The {\em type  of a homomorphism} $\rho: 
\Gamma \rightarrow G$ is the set of isomorphism classes of the local representations
$\rho_i: \Gamma_{\widehat{p}_i} \rightarrow G$, or equivalent, it is the set of conjugacy classes in $G$
given by the images of $\rho_i(\gamma_i)$, where $\gamma_i$ is a generator of the cyclic group
$\Gamma_{\widehat{p}_i}=\langle \gamma_i\rangle$. The type of a homomorphism is denoted by ${\bm \tau}=(\tau_1,\dots,\tau_n)$, where $n=|\mathcal{R}|$. 
\end{definition}

Let $\widehat{p}_i$ be any branch point of $\widehat{C}$, and let $\widehat{t}_i$ be a {\em special formal parameter 
} at the point $\widehat{p_i}$, such that $\gamma\cdot\widehat{t}_i:=(\exp(2\pi{\sqrt{-1}}/d_i)\widehat{t}_i,$ 
where $\gamma$ is a generator of the stabilizer $\Gamma_{\widehat{p}_i}$ and $d_i=|\Gamma_{\widehat{p}_i}|$.
Any $(\Gamma, G)$-bundle $\widehat{E}$ is trivial as a $G$-bundle on a formal disk
$D_{\widehat{p}_i}:=\operatorname{Spec}[[\widehat{t}_i]]$, and in particular $\widehat{E}_{|D_{\widehat{p}_i}}$ is
a $(\Gamma_{\widehat{p}_i},G)$ bundle. So any $(\Gamma_{\widehat{p}_i}, G)$-bundle on $D_{\widehat{p}_i}$ is determined by a homomorphism $\rho_i: \Gamma_{\widehat{p}_i} \rightarrow G$ such that $\gamma\cdot (u,g)=(\gamma\cdot u, \rho_i(\gamma)g)$,  where $u \in D_{\widehat{p}_i}$. Moreover such an homomorphism is unique up to conjugation. We refer the
reader to  \cite[Lemma 2.5]{TelWood} and  \cite[Thm.\ 6.1.9]{kumarnewbook}. 

Let ${\bm \tau}=(\tau_1,\dots, \tau_n)$ be the unique element of the Weyl alcove $\Phi_0$ such that 
$\rho_i(\gamma_i)$ is conjugate to $\exp(2\pi{\sqrt{-1}}\tau_i)$ as described by Lemma \ref{lem:exp}. We define the 
{\em local type} of a $\Gamma$-$G$-bundle $\widehat{E}$ to be the $n$-tuple ${\bm \tau}=(\tau_1,\dots,\tau_n)$ and 
consider the following stack:

\begin{definition}
Let $\widehat{C}$ be a ramified Galois cover of $C$. Choose points $\widehat{p}_i$ for each point $p_i$ in
$\mathcal{R}$, and let ${\bm \tau}$ be an $n$-tuple of elements in $\Phi_0$. We define the moduli stack $\mathcal{B}un_{\Gamma,G}^{ {\bm \tau}}( \widehat{C})$ to be the groupoid parametrizing 
    $\Gamma$-$G$-bundles on $\widehat{C}$ of local type ${\bm \tau}$. 
\end{definition}

\subsection{Uniformization of $\Gamma$-$G$-bundles of fixed local type}
We will now discuss a uniformization theorem for $\mathcal{B}un_{\Gamma,G}^{ {\bm \tau}}( \widehat{C})$ under
the further assumption that $\theta(\tau_i)<1$ for $1\leq i\leq n$. We will show that the stack
$\mathcal{B}un_{\Gamma,G}^{ {\bm \tau}}( \widehat{C})$ is isomorphic to $\mathcal{P}ar_{G}(C,\vec{P})$, where
$\vec{P}=(P_1,\dots, P_n)$ are standard parabolic subgroups of $G$ determined by ${\bm \tau}=(\tau_1,\dots, \tau_n)$. 

As in the case  of parabolic bundle we consider the functor $\mathscr{S}_{G}^{{\bm \tau}}: \mathcal{A}lg
\rightarrow \mathcal{S}ets$ that assigns to a finitely generated $k$-algebra $R$ the isomorphism classes 
of pairs $(\widehat{E}_R, \widehat{\sigma}_R)$, where 
\begin{itemize}
\item $\widehat{E}_R$ is a $(\Gamma, G)$ bundle over $\widehat{C}\times \operatorname{Spec}R$ of local type
$\tau_i$ at the points $\widehat{p}_i$, and
 
\item  $\widehat{\sigma}_R$ is a $\Gamma$-equivariant section of $\widehat{E}_R$ over $p^{-1}(C\backslash \vec{p})\times \operatorname{Spec}R$.
\end{itemize}
By \cite[Prop.\ 3.1.1]{BalajiSeshadri}  and 
 \cite[Thm.\ 6.1.12]{kumarnewbook}, the functor
$\mathscr{S}_{G}^{{\bm \tau}}$ is represented by the ind-scheme
$\prod_{i=1}^n G((t_i))/\mathcal{P}_i$, where $t_i=(\widehat{t}_i)^{d_i}$
are local parameters at the points $p_i$ and $\mathcal{P}_i$ are parabolic
subgroups of the loop group $G((t_i))$. The following theorem is due to
Balaji-Seshadri \cite[Prop.\ 3.1.1]{BalajiSeshadri} and it can also be found
in Kumar \cite[Thm.\  6.1.15]{kumarnewbook}.

\begin{theorem}\label{thm:mainuniform}
Let $n\geq 1$ and ${\bm \tau}$ as above Then there is an isomorphism of the stacks $\mathcal{B}un_{\Gamma,
G}^{{\bm \tau}}(\widehat{C})$ and the quotient stack $\operatorname{L}_{C,\vec{p}}(G)\backslash \left(\prod_{i=1}^nG((t_i))/\mathcal{P}_i\right)$.  

\end{theorem}

\begin{remark}
We emphasize that Balaji-Seshadri \cite{BalajiSeshadri} work without the
    assumption that $\theta(\tau_i)<1$. In this general set-up the groups
    $\mathcal{P}_j\subset G((t_i))$ that appear in 
    \cite[Prop.\ 3.1.1]{BalajiSeshadri} are not necessarily contained in $G[[t_i]]$. 
\end{remark}

\subsection{Invariant direct image functor}Let $p:W\rightarrow T$ be a finite flat surjective morphism of
Noetherian integral schemes (as in \cite[Sec.\ 4]{BalajiSeshadri}) such that the corresponding extension of 
function fields is Galois with Galois group $\Gamma$. It follows that $\Gamma$ acts on $W$ and $T=W/\Gamma$.  Let $\mathscr{G}$ be a smooth affine group scheme on $W$.  Following
Balaji-Seshadri \cite{BalajiSeshadri}, Pappas-Rapoport \cite{PR1}, and Edixhoven \cite{Edinhoven}, we define:

\begin{definition}
The invariant direct image of $\mathscr{G}$, namely $p_{*}^{\Gamma}(\mathscr{G}):=(p_*(\mathscr{G}))^{\Gamma}$,
where $p_*\mathscr{G}$ is the group functor Weil restriction of scalars-$\operatorname{Res}_{W/T}(\mathscr{G})$
and $(p_*(\mathscr{G}))^{\Gamma}$ is the smooth  closed fixed point subgroup scheme
of the $\Gamma$-scheme $p_*(\mathscr{G})$. In particular for any $T$-scheme $S$, we get that
$p_*^{\Gamma}(\mathscr{G})(S):=\left(\mathscr{G}(S\times_T W)\right)^{\Gamma}$. 
\end{definition}

In our present set-up we consider $\widehat{C}\rightarrow C$ to 
be a ramified Galois covering with Galois group $\Gamma$, and let $\mathcal{R}$ be the ramification locus. Let $G$
be a connected, simple algebraic group and $\rho: \Gamma \rightarrow G$ and we fixed the local
type  ${\bm \tau}=(\tau_1,\dots, \tau_n)$ such that $\theta(\tau_i)<1$ for all $1\leq i\leq n$. Consider
the invariant push forward  $\mathscr{H}:=p_{*}^{\Gamma}(\widehat{C}\times G)$  of
the constant group scheme $\widehat{C}\times G$ to get a Bruhat-Tits type group scheme on $C$ with the following property: 
\begin{enumerate}
\item The geometric fibers of $\mathscr{H}$ are connected. 

\item On the punctured curve $C\backslash \mathcal{R}$, the group scheme $\mathscr{H}$ is split.

\item For $p_i \in \mathcal{R}$, the group scheme $\mathscr{H}(\widehat{\mathcal{O}}_{C,p_i})$ is the 
subgroup $\mathcal{P}_i:=ev_{p_i}^{-1}(P_i)\subset G[[t_i]]$, where $P_i$ is a standard parabolic subgroup in $G$ 
given by $\tau_i$.
\end{enumerate}
 Pappas-Rapoport \cite{PR1} considered the moduli stack
 $\mathcal{B}un_{\mathscr{H}}(C)$ of $\mathscr{H}$-torsors on a curve $C$,
 where $\mathscr{H}$ is a parahoric Bruhat-Tits group scheme. A
 uniformization theorem for such torsors was proved by 
 Heinloth \cite{Heinloth}.  Using  \cite[Thm.\  4.1.6]{BalajiSeshadri} and the discussion above, we can reformulate  the correspondence in Theorem \ref{thm:mainuniform} by the following:

\begin{proposition}\label{prop:keypushforward}
	The stacks $\mathcal{B}un_{\Gamma, G}^{{\bm \tau}}(\widehat{C})$ and $\mathcal{P}ar_{G}(C,\vec{P})$ are isomorphic under the invariant push-forward functor. In particular if $\widehat{\mathcal{E}}$ is a family of
$\Gamma$-$G$-bundles of type ${\bm \tau}$ on the curve $\widehat{C}$ parametrized by a schemes $T$, the $p_*^{\Gamma}(\widehat{\mathcal{E}})$ is a family of quasiparabolic $G$-bundles with parabolic structures at the ramification points  determined by ${\bm \tau}$.

Moreover, by Proposition \ref{prop:paravacua} and Theorem \ref{thm:mainuniform}, both the stacks $\mathcal{B}un_{\Gamma,G}^{{\bm \tau}}(\widehat{C})$
and $\mathcal{P}ar_{G}(C,\vec{P})$ are isomorphic to $\operatorname{L}_{C,\vec{p}}(G)\backslash \left(\prod_{i=1}^nG((t_i))/\mathcal{P}_i\right)$, where $\vec{P}=(P_1,\dots, P_n)$ and $P_i=P(\tau_i)$
are Kempf parabolic subgroups determined by $\tau_i$. 
\end{proposition}

Let $C_{T}\lra T$ be a family of smooth projective curves parametrized by  $T$ and $p_1,\dots, p_n$ are disjoint sections. Recall that given integers $d_1,\dots, d_n$  and a $n$-points curve $(C_{0},p_1,\dots, p_n)$, we can find a Galois cover $(\widehat{C}, \widehat{p}_1,\dots, \widehat{p}_n)$ with Galois group $\Gamma$ and isotropy of order $d_i$ at $\widehat{p_i}$.  Fixing such a $\Gamma$, we can  find a family of curves $\widehat{C}_T\lra T$  along with a finite map $p: \widehat{C}_T\lra C_T$ such that 
\begin{itemize}
	\item $\Gamma$ acts on $\widehat{C}_T$ preserving $p$ inducing a Galois covering $\pi: \widehat{C}_T \rightarrow {C}_T$.
	\item Section $\widehat{p}_1,\dots, \widehat{p}_n$ such that isotropy at $\widehat{p}_i$ is of order $d_i$ for all $i$. 
	\item The cover just depends on the choice of $\Gamma$ and the integers $d_1,\dots, d_n$. 
\end{itemize} We refer the reader to \cite[Sec.\ 4d]{BiswasRaghavendra2}
for the construction of such families. These covers are called pointed
admissible covers, and a moduli stack for these objects has been constructed in \cite{JKK}.

Now given a $\Gamma$-Galois covering $\widehat{C}_T \rightarrow C_T$, the parabolic orbifold correspondence as 
described in Proposition \ref{prop:keypushforward} works verbatim for 
families of parabolic and orbifold bundles parametrized by $T$.

\subsection{Determinant of cohomology  for $\widehat{C}$ and invariant pushforward}

Let $\widehat{\mathcal{E}}$ be a family of $\Gamma$-$G$-bundles on $\widehat{C}$ of local type ${\bm \tau}$ 
parametrized by a scheme $T$. By Proposition \ref{prop:keypushforward}, we get a family $\mathcal{E}$ of 
quasiparabolic $G$-bundles on $C$ with parabolic structures at the points $\vec{p}=(p_1,\dots, p_n)$ in the 
ramification locus.
Observe that we have an $n$-tuple integers $\vec{d}=(d_1,\dots, d_n)$ which encodes the order of ramification at the points $(p_1,\dots, p_n)$. Moreover $\exp(2\pi\sqrt{-1}d_i\tau_i)=1$ for all $\leq i\leq n$.
Now ignoring the $\Gamma$-action, we get a family of principal $G$-bundles on
$\widehat{C}$  and hence by \eqref{eqn:determinG}, we get a line bundle on
$T$ subject to the choice of a representation $(\phi,V)$ of $G$. On
the other hand, we also get a line bundle on $T$ by starting with a family of
quasiparabolic bundles $\mathcal{E}$ obtained from the invariant push
forward of $\widehat{\mathcal{E}}$ and then applying the construction in
\eqref{eqn:deterG}.  The following proposition, which is minor variation of 
\cite[Prop.\ 4.5]{BiswasRaghavendra}, compares these two line bundles on $T$.

\begin{proposition}\label{prop:comparingdeter}Let $\phi: G\rightarrow \SL(V)$ be a representation of $G$. Choose a local-type ${\bm \tau}$ such that 
	$\theta_{\mathfrak{sl}(V)}(\phi(\tau_i))<1$ for all $1\leq i\leq n$, where $\theta_{\mathfrak{sl}(V)}$ is the highest root.
Then for any family $\widehat{\mathcal{E}}$ of $\Gamma$-$G$-bundles on $\widehat{C}$ parametrized by a scheme $T$
of local type ${\bm \tau}$, we have:
$$\operatorname{Det}(\widehat{\mathcal{E}}(V)) \cong
\operatorname{Det}((\id_T\times p)^*(\mathcal{E}(V)))\otimes
    \bigl(\otimes_{i=1}^n\bigl(\otimes_{j=1}^{k_i}\det
    \operatorname{Gr}^{j} \mathscr{F}_{\bullet, p_i}(\mathcal{E}_{|T\times
    p_i})^{\otimes N\alpha_{j,i}}\bigr)\bigr)\ ,
$$ where
\begin{enumerate}
\item the filtration $\mathscr{F}_{\bullet, p_i}$ and the weights $\bm \alpha=(\bm \alpha_{p_1},\dots, \bm \alpha_{p_n})$ are determined by the associated topological type $\phi(\bm \tau)$,
\item $N > 0$ is the smallest integer such that $N\alpha_{j,i}$ are integers, and
\item $\widehat{C} \rightarrow C$ is a Galois $\Gamma$-cover such that the isotropy of order $N$ at all points $p_i$.
	\end{enumerate}
\end{proposition}
\begin{proof} We will be done by  \cite[Prop.\ 4.5]{BiswasRaghavendra} once
    we can show that
    $p_{*}^{\Gamma}(\widehat{\mathcal{E}})\times^{\phi}(V)$ equals
    $p_{*}^{\Gamma}\bigl( \widehat{\mathcal{E}}\times^{\phi}V\bigr)$ as a family parabolic vector bundles on $C$ parametrized by $T$. This follows from the definition directly.
\end{proof}

Now following \cite{BiswasRaghavendra}, we will construct a curve $\widehat{C}$ from the data $\bm\tau$ and compare 
the determinant of cohomology line bundle on $\mathcal{B}un_{\Gamma,G}^{\bm \tau}(\widehat{C})$ with the parabolic 
determinant of cohomology on $C$ via the functor $p_*^{\Gamma}$. Mimicking the set-up of \cite[Def.\ 
4.10]{BiswasRaghavendra}, given ${\bm \tau}$ in the Weyl alcove $\Phi_0$ choose an integer $N$ such that 
$\exp(2\pi\sqrt{-1}N\nu_{\mathfrak{sl}(V)}(\phi(\tau_i)))=1$ for all $1\leq i\leq n$. By the Selberg lemma, 
\cite{Selberg}, we can find a ramified cover $p:\widehat{C}\rightarrow C$ with ramification exactly over the points 
$p_i$ with cyclic isotropy group of order $N$ at all the fixed points. Let $\Gamma$ be the Galois group. With these 
assumptions,  \cite[Prop.\ 4.11]{BiswasRaghavendra} generalizes to the following:

\begin{proposition}
Let $\mathcal{E}= p_{*}^{\Gamma}\widehat{\Ecal}$ be as in Proposition \ref{prop:comparingdeter}. Then the line
bundles $\operatorname{Det}(\widehat{\mathcal{E}}(V))$ and $\left(
\operatorname{Det}_{par}(\mathcal{E}(V),N\cdot m_{\phi}\cdot{\bm \tau})\right)^{\otimes \frac{|\Gamma|}{N}}$
on $T$ are canonically isomorphic.
\end{proposition}

\section{The properness condition and codimension estimates}\label{sec:appC}

In this section, we will show that the moduli space $M^{par,rs}_{G}$ of regularly stable parabolic $G$-bundles on a 
curve admits no nonconstant functions. This will imply 
Theorem \ref{thm:hitchimain} (iii). Throughout this section we assume that $G$ is simple and simply connected (or more generally semisimple, but we do not need it for applications). 
We have the following key codimension estimate, which essentially follows from the same argument as in Faltings 
\cite{Faltings:93} and Laszlo \cite{Laszlo}. Fix $n\geq 1$, and  let ${\bm \tau}=(\tau_1,\dots,\tau_n)$ be a $n$-tuple of weights in 
the Weyl alcove for a group $G$, and let $d$ be the minimum positive integer such that 
$\exp(2\pi\sqrt{-1}d\cdot\nu_{\gfrak}(\tau_i))\,=\,1$ for all $1\leq
i\leq n$.  Choose a curve $\widehat{C}$ that is 
a Galois cover over $C$ ramified exactly over the points $p_1,\dots,p_n$ with the same ramification order $d$ and 
\'etale on the complement.

\begin{lemma}\label{lem:codimension}
    Let $\mathcal{P}ar_{G}(C,\vec{P})$
(respectively,
$\mathcal{P}ar^{rs}_{G}(C,\vec{P})$) be the moduli stack parametrizing
parabolic $G$-bundles (respectively, regularly stable parabolic $G$-bundles) given by a choice of
weights $\bm \tau$ on a $n$-pointed curve $C$ of genus $g(C)$. Further assume that
$\mathcal{P}ar^{rs}_{G}(C,\vec{P})$ is nonempty. Then the codimension of the complement
$\mathcal{P}ar^{}_{G}(C,\vec{P})
\backslash \mathcal{P}ar^{rs}_{G}(C,\vec{P})\subset \mathcal{P}ar^{}_{G}(C,\vec{P})$ is at least two 
provided  $g(\mathscr{C})\,\geq\, 3$, and $g(\mathscr{C})\,\geq\, 2$ if $G$ does not have
an $\SL_2$ factor. 
Moreover, if $G=\SL_r$ for $r>2$, the codimension of the complement is at least $3$. 
\end{lemma}

\begin{proof}Let $\bm \tau$ be the choice of the weights determining the
    stability conditions and the parabolic subgroups $\vec{P}=(P_1,\dots,
    P_n)$. Consider  an $n$-tuple of
Borel subgroups $\vec{B}$ and the the moduli stack of quasi parabolic bundles with full flags $\mathcal{P}ar_G(C,\vec{B})$. There is a natural forgetful map 
$\mathcal{P}ar_G(C,\vec{B})\rightarrow \mathcal{P}ar_G(C,\vec{P})$ whose fibers are product of flag varieties. Now	consider the substack $\mathcal{P}ar^{ss}_{G}(C,\vec{B}))$ (respectively, $\mathcal{P}ar^{s}_{G}(C,\vec{B}))$)
parametrizing semistable (respectively, stable) parabolic bundles with respect to the same weight data $\bm \tau$. This  preserves stability (hence also regular stability) and hence the  forgetful map restricts to a map  $\mathcal{P}ar^{ss}_{G}(C,\vec{B}))\rightarrow\mathcal{P}ar^{ss}_{G}(C,\vec{P}))$ that preserve both the stable and the regularly stable loci. Consequently, without loss of generality assume that we are in the case of full flags.
	
 It is enough to show the following:
	\begin{enumerate}
		\item{\em The codimension of the complement of
            $\mathcal{P}ar^{ss}_{G}(C,\vec{B})$ (respectively, 
            $\mathcal{P}ar^{s}_{G}(C,\vec{B}))$) in
            $\mathcal{P}ar^{}_{G}(C,\vec{B})$ is at least two:} We will
            freely use the parabolic orbifold correspondence. Let
            $\mathcal{E}$ be a parabolic $G$ bundle admitting a reduction
            to parabolic bundle $\mathcal{E}_Q$ with structure group $Q$, where $Q$ is a parabolic subgroup of $G$ with its Levi subgroup $L_Q$. 
		Consider the sheaf $\mathfrak{n}_Q^{par}(\ad \mathcal{E})$ given by the cokernel of  map 
		$\Spar(\mathcal{E}_{L_Q}) \hookrightarrow \Spar(\mathcal{E})$, where $\mathcal{E}_{L_Q}$ is the induced parabolic bundle with structure group $L_Q$. 
            If $\mathcal{E}$ is in the complement of the
            $\mathcal{P}ar_G^{ss}(C,\vec{B})$ (respectively,
            $\mathcal{P}ar_G^{ss}(C,\vec{B})$), then $\deg
            \mathfrak{n}_Q^{par}(\ad \mathcal{E})$ is  strictly positive
            (respectively, nonnegative).
	Let $B_{L_Q}$ be the Borel of $L_Q$; then the complement has dimension
		\begin{align*}
            \dim \mathcal{P}ar_{G}(C,\vec{B})& -\bigl(\dim \mathcal{P}ar_{L_Q}(C,\vec{B}_{L_Q})+
			h^1(C,\mathfrak{n}_Q^{par}(\ad \mathcal{E}) )\bigr)\\
			& = (g(C)-1)(\dim G-\dim Q) + n( \dim G/B-\dim
            L_{Q}/B_{L_Q})\\
            &\qquad\qquad + \deg (\mathfrak{n}_Q^{par}(\ad \mathcal{E}))-h^0(C,\mathfrak{n}_Q^{par}(\ad \mathcal{E}) )\\
			& \geq (g(C)+n-1)\dim \mathfrak{n}_{Q}-1\ ,
				\end{align*}
                since we may assume  $h^0(C,\mathfrak{n}_Q^{par}(\ad \mathcal{E})\leq 1$.
            Now notice that $g(\mathscr{C})\geq 2$ implies that $g(C)+n-1\geq
            2$, and $g(\mathscr{C})\geq 3$ implies that $g(C)+n-1\geq
            3$.	
           Further observe that if $G=\SL_r$, then $\dim \mathfrak{n}_Q>1$ if $r>2$. 

\item {\em The codimension of the complement of $\mathcal{P}ar^{rs}_{G}(C,\vec{B})$ in 
$\mathcal{P}ar^{s}_{G}(C,\vec{B})$ is at least two:} Here we can assume $G\neq \SL_r$.  If 
$\mathcal{P}$ is a stable orbifold bundle on $\mathscr{C}$ which has a noncentral automorphism, then by \cite[Lemma 
11.1]{Laszlo} $\mathcal{P}$ has an $L$-structure where $L$ is a reductive subgroup of $G$ with Borel $B_L$. Then the required codimension is at least
\begin{align*}
	&\dim \mathcal{P}ar_G(C,\vec{B})-\dim \mathcal{P}ar_L(C,\vec{B}_{L})\\
	&\quad =(g(C)-1)(\dim G-\dim L)+ \sum_{i=1}^n (\dim G/B-\dim L/B_L)\\
	&\quad=(g(C)-1)(\dim G-\dim L)+n(\dim G/B-\dim L/B_L)\ .
				\end{align*}
			Now $\dim G/B-\dim L/B_L\geq 1$ and  $\dim G -\dim L\geq 2$,
            so if $g(C)\geq 1$, then the codimension is at least
            $2g(C)-2+n\geq 3$, by the assumption that $g(\mathscr{C})\geq 2$.
             Thus, we are left to consider the case where $g(C)=0$. 
             
             Since
             $g(\mathscr{C})\geq 2$, we have
	$n\geq 5$. Suppose first that  $L$ is not a torus. Then $n\dim
            L/B_L-\dim L$ is an increasing function of $L$. This implies
            that $n(\dim G/B-\dim L/B_L)-(\dim G-\dim L)$ is
            decreasing function of $L$. Hence, the codimension is at least 
	\begin{align*}
	\min_{L=L_Q}\left( n(\dim G/B-\dim L/B_L)-(\dim G-\dim L)\right)\ ,
	\end{align*}where $L$ ranges over the Levi subgroups $L_Q$ of proper maximal parabolics $Q$ in $G$. Thus we get that the codimension of the complement is at least
$$	
	\min_{L=L_Q}\left( n(\dim G/B-\dim L/B_L)-(\dim G-\dim L)\right)
            =\min_{Q}((n-2)\dim \mathfrak{n}_{Q})
            \geq 3 \ .
        $$
\end{enumerate} 
Now suppose $L$ is a torus. In this case, the codimension is simply 
$$
n(\dim G/B)-(\dim G-\dim L)
\geq n(\dim G/B)-\dim G\geq (n-3)\dim \mathfrak{n}_{B}\geq 2\ . 
$$
This completes the proof of the Lemma.
\end{proof}

Let $M_{G}^{par,ss}$ (respectively, $\mathcal{P}ar^{ss}_{G}(C,\vec{P})$)
be the moduli space (respectively, moduli stack) of semistable parabolic
$G$-bundles on $C$ with parabolic structures at $n$-marked points. It is
well-known  that $M_{G}^{par,ss}$
(respectively, the regularly stable part $M_{G}^{par,rs}$) is a GIT quotient
(respectively, good quotient) of a smooth scheme $R_{G}^{par, ss}$ (respectively, $R_{G}^{par,rs}$) by a
reductive group (cf.\ \cite{BalajiBiswasNagaraj, BalajiSeshadri}). Moreover, $M_{G}^{par,ss}$ is a seminormal projective variety with rational singularities.
Now Lemma \ref{lem:codimension} implies that 
$\operatorname{codim}(R_{G}^{par,ss}\backslash R_{G}^{par,rs})\geq 2$, provided $R^{par,rs}_{G}$ is nonempty. Hence, by Hartogs' theorem we get the following:

\begin{corollary}The natural inclusion map $M_{G}^{par,rs} \to M_{G}^{par,ss}$ induces isomorphisms between 
	$H^0(M_{G}^{par,rs},\mathcal{O}_{M_{G}^{par,rs}})$ and $H^0(M_{G}^{par,ss},\mathcal{O}_{M_{G}^{par,ss}})$.
\end{corollary}

Recall $Y_G^{par,rs}$ from the proof of Proposition \ref{prop:iso}.
Then we have the following lemma, the proof of
which is analogous to that in \cite[Prop.\ 11.6]{Laszlo}.

\begin{lemma}\label{lem:codim2}
	The codimension of the complement of $Y_{G}^{par,rs}$ in $M^{par,rs}_{G, \bm \tau}$ is at least $3$ if
$g(\mathscr{C})\geq 3$ for arbitrary $\mathfrak{g}$, or $g(\mathscr{C})\geq 2$ when $\mathfrak{g}$ has no factor of type $A_1$ or $C_2$. 
\end{lemma}

\begin{proof}Suppose $\mathcal{E}$ be a regularly stable  $\Gamma$-$G$-bundle which is not
    stable as a $G$-bundle. Then we can realize it as the image of a rational map from the moduli space of  $\Gamma$-$L$-bundles on $\widehat{C}$,
where $L$ is a Levi subgroup of a parabolic subgroup $Q$ of $G$. 
 If $\mathcal{E}$ is stable we can realize  it as the image of rational map $M_{L}(\widehat{C})$, where $L$ is a reductive subgroup  (\cite[Prop.\ 11.6]{Laszlo}) of $G$.
Thus, the complement of $Y_{G}^{par,rs}$ in $M_{G,\bm \tau}^{par,rs}$ is dominated by union of the moduli spaces of $\Gamma$-$L$-bundles on the curve
$\widehat{C}$. of type $\bm \tau$, where $L$ is a reductive subgroup. Now
    as in the proof of Lemma \ref{lem:codimension}, without loss of
    generality assume that $\bm \tau$ corresponds to a tuple of Borel
    subgroups. Then the required codimension is at least 
\begin{eqnarray*}
& (g(C)-1)(\dim G -\dim L) + n(\dim G/B-L/B_L)-\dim Z(L)\ .
\end{eqnarray*}
 Now $\dim G -\dim L$ is at least 4 unless $\mathfrak{g}$ has a factor of
    type $A_1$ or $C_2$. Thus, we are done by the assumptions on
    $g(\mathscr{C})$ and the  calculations as in the proof of Lemma
    \ref{lem:codimension}.
\end{proof}

\subsection*{Acknowledgements}The authors warmly thank Johan Martens for enlightening discussions.
They also thank the anonymous referee for numerous helpful suggestions.

\bibliographystyle{amsplain}
	\bibliography{./papers}

\end{document}